\RequirePackage{rotating}
\documentclass[11pt,francais]{amsart}
\usepackage{ae}
\usepackage{aeguill}
\usepackage[utf8]{inputenc}
\usepackage[T1]{fontenc}
\usepackage{amsmath,amsthm,amssymb}
\usepackage[francais]{babel}
\usepackage{graphicx}
\usepackage{latexsym}
\usepackage{CJK}
\usepackage[all]{xy}
\usepackage{amsmath}
\usepackage{cancel}
\usepackage{mathtools}
\usepackage{tikz}
\usepackage{slashbox}
\usepackage{pict2e}
\usetikzlibrary{chains}
\usepackage{rotating}
\usepackage[linktoc=all]{hyperref}
\hypersetup{colorlinks=true,linkcolor=blue,citecolor=red,urlcolor=blue,pdfborderstyle={/S/U/W1}}
\usepackage{changepage}

\tikzset{node distance=2em, ch/.style={circle,draw,on chain,inner sep=2pt},chj/.style={ch,join},every path/.style={shorten >=4pt,shorten <=4pt},line width=1pt,baseline=-1ex}

\let\dlabel=\alabel

\newcommand{\dnode}[2][chj]{%
\node[#1,label={\dlabel{#2}}] {};
}

\newcommand{\dnodebr}[1]{%
\node[chj,label={\dlabel{#1}}] {};
}

\addto\captionsfrench{}

\def\R{\mathbb{R}}
\def\C{\mathbb{C}}
\def\Q{\mathbb{Q}}

\def\Z{\mathbb{Z}}
\def\N{\mathbb{N}}

\def\F{\mathbb{F}}

\usepackage{babel}
\usepackage{array}

\makeatletter
\newskip\@bigflushglue \@bigflushglue = -100pt plus 1fil

\def\bigcentering{\let\\\@centercr\rightskip\@bigflushglue%
\leftskip\@bigflushglue
\parindent\z@\parfillskip\z@skip}

\makeatother

\newtheorem{defi}{Définition}[subsection]
\newtheorem{theo}[defi]{Théorème}
\newtheorem*{theo*}{Théorème}
\newtheorem{prop}[defi]{Proposition}
\newtheorem*{prop*}{Proposition}
\newtheorem{propdef}[defi]{Proposition-Définition}

\newtheorem{lem}[defi]{Lemme}
\newtheorem{conj}[defi]{Conjecture}
\newtheorem{cor}[defi]{Corollaire}

\usepackage[all]{xy}

\title[Traces d'opérateurs de Hecke sur les espaces de formes automorphes]{Traces des opérateurs de Hecke sur les espaces de formes automorphes de $\mathrm{SO}_7$, $\mathrm{SO}_8$ ou $\mathrm{SO}_9$ en niveau $1$ et poids arbitraire}
\author{Thomas Mégarbané}
\address{CMLS, {\'E}cole polytechnique, CNRS, Universit{\'e} Paris-Saclay, 91128 Palaiseau Cedex, France}
\email{thomas.megarbane@polytechnique.edu}
\begin{document}
\maketitle
\begin{abstract}
Dans cet article, nous déterminons la trace de certains opérateurs de Hecke sur les espaces de formes automorphes de niveau 1 et poids quelconque des groupes spéciaux orthogonaux des réseaux euclidiens $\mathrm{E}_7$, $\mathrm{E}_8$ et $\mathrm{E}_8 \oplus \mathrm{A}_1$. En utilisant la théorie d'Arthur, nous en déduisons des informations sur les paramètres de Satake des représentations automorphes des groupes linéaires découvertes par Chenevier et Renard dans \cite{CR}. Nos résultats corroborent notamment une conjecture de Bergström, Faber et van der Geer sur la fonction zêta de Hasse-Weil de l'espace de module des courbes de genre $3$ à $17$ points marqués.
\end{abstract}
\section{Introduction.}
\indent Donnons-nous $n\equiv 0,\pm 1\mathrm{\ mod\ }8$ un entier positif, et plaçons-nous dans l'espace euclidien $\R^n$ muni de son produit scalaire usuel $(x_i)\cdot (y_i) = \sum_i x_i y_i$. On définit $\mathcal{L}_n$ l'ensemble des réseaux pairs $L\subset \R^n$ tels que $\mathrm{det}(L) =1$ si $n$ est pair, et $\mathrm{det}(L) =2$ sinon. Le groupe $\mathrm{O}_n(\R)$ a une action naturelle sur $\mathcal{L}_n$, et on pose $X_n = \mathrm{O}_n(\R) \setminus \mathcal{L}_n$.
\\ \indent On s'intéressera plus particulièrement aux cas où $n\in \{7,8,9\}$. Ces cas ont notamment la propriété que $\mathrm{X}_n$ est réduit à un seul élément, à savoir respectivement la classe des ``réseaux de racines" $\mathrm{E}_7$, $\mathrm{E}_8$ et $\mathrm{E}_8 \oplus \mathrm{A}_1$ dont les définitions sont rappelées au paragraphe \ref{2.1}.
\\ \indent Si on considère $(W,\rho)$ une représentation de dimension finie de $\mathrm{SO}_n(\R)$ sur $\C$, on définit alors l'espace des formes automorphes de poids $W$ pour $\mathrm{SO}_n $ comme :
$$\mathcal{M}_W(\mathrm{SO}_n) := \left\{ f:\mathcal{L}_n \rightarrow W \vert \left( \forall \gamma \in \mathrm{SO}_n(\R) \right) f(\gamma \cdot L) = \rho(\gamma) \cdot f(L) \right\} .$$
\indent C'est un espace vectoriel de dimension finie.
\\ \\ \indent Soient $A$ un groupe abélien fini, et $L_1, L_2\in \mathcal{L}_n$. On dit que $L_1$ et $L_2$ sont $A$-voisins si :
$$L_1/(L_1\cap L_2) \simeq L_2/(L_1\cap L_2) \simeq A$$
\indent Si $A=(\Z/d\Z)^k$, on parle de $d,\dots ,d$-voisins (et de $d$-voisins si $k=1$). Si $p$ désigne un nombre premier, et $q$ une puissance de $p$, alors il est facile de construire tous les $q$-voisins ou les $p,\dots ,p$-voisins d'un $L\in \mathcal{L}_n$ donné (comme rappelé aux proposition-définitions \ref{p-voisin} et \ref{ppp-voisin}).
\\ \\ \indent \`A la notion de $A$-voisins est associé un ``opérateur de Hecke" $\mathrm{T}_A$ sur chaque espace de formes automorphes $\mathcal{M}_W(\mathrm{SO}_n)$ défini par la formule :
$$\left( \forall L\in \mathcal{L}_n \right) \left( \forall f\in \mathcal{M}_W  (\mathrm{SO}_n) \right)\ \mathrm{T}_A(f)(L) = \sum_{L'\ A-\text{voisin de }L} f(L')$$
\indent Le premier but de notre travail est de déterminer la trace de $\mathrm{T}_A$ sur les espaces $\mathcal{M}_W(\mathrm{SO}_n)$ pour toute représentation irréductible $W$ de $\mathrm{SO}_n$ et $n\in \{7,8,9 \}$. Notre point de départ est le suivant :
\begin{prop}Supposons que $X_n$ est réduit à un élément (c'est-à-dire que $n\leq 9$). Soit $L_0\in \mathcal{L}_n$, et $\mathrm{vois}_A(L_0)$ l'ensemble de ses $A$-voisins. Le groupe $\mathrm{SO}(L_0)$ a une action naturelle sur $\mathrm{vois}_A(L_0)$. On note $\mathcal{V}_j$ les orbites de cette action, et pour chaque $j$ on choisit un élément $g_j \in \mathrm{SO}_n(\R)$ tel que $g_j \cdot L_0 \in \mathcal{V}_j$ (ce qui est toujours possible comme $\vert X_n \vert =1$). Alors on a l'égalité :
$$ \mathrm{tr}\left( \mathrm{T}_A \vert \mathcal{M}_W(\mathrm{SO}_n)\right) = \dfrac{1}{\vert \mathrm{SO}(L_0) \vert} \cdot \left( \displaystyle{\sum_j} \left( \vert \mathcal{V}_j\vert \cdot \displaystyle{\sum_{\gamma \in \mathrm{SO}(L_0)}}\mathrm{tr}\left( \gamma g_j \vert W \right) \right) \right).$$
\end{prop}
\indent Dans cet énoncé, on désigne par $\mathrm{SO}(L)$, pour $L\in \mathcal{L}_n$ le sous-groupe des éléments $g\in \mathrm{SO}_n(\R)$ tels que $gL=L$, qui est un groupe fini. On dispose d'un énoncé analogue mais plus technique, sans l'hypothèse $\vert X_n \vert =1$, détaillé au paragraphe \ref{3.1}.
\\ \\ \indent Afin de calculer explicitement cette formule, nous devons déterminer les termes qui y interviennent, ce qui fait l'objet des chapitres \ref{4} et \ref{5}.
\\ \indent Au paragraphe \ref{4.2} : on explique comment déterminer le groupe $\mathrm{SO}(L_0)$, qui est très proche du groupe de Weyl du système de racines associé à $L_0$.
\\ \indent Au paragraphe \ref{4.3} : on donne un algorithme pour déterminer les orbites de $\mathrm{vois}_A(L_0)$ pour l'action de $\mathrm{SO}(L_0)$. On se restreint ici au cas où $A\simeq \Z/q\Z$ pour $q$ une puissance d'un nombre premier impair. Notre algorithme nous retourne pour chaque orbite $\mathcal{V}_j$ la quantité $\vert \mathcal{V}_j \vert$ ainsi qu'un élément $x_j\in C_{L_0}(\Z/q\Z)$ dont le $q$-voisin associé $L_0'(x_j)$ est dans l'orbite $\mathcal{V}_j$.
\\ \indent Au paragraphe \ref{4.4} : on donne un algorithme qui, à partir d'une droite isotrope $x_j\in C_{L_0}(\Z/q\Z)$, détermine une transformation $g_j\in \mathrm{SO}_n(R)$ vérifiant : $g_j (L_0) = L_0'(x_j)$. Ce même algorithme permet dans un cadre plus général de déterminer, à partir d'une famille $\Z$-génératrice d'un réseau $L'_0$ isomorphe à $L_0$, une isométrie transformant $L_0$ en $L'_0$.
\\ \indent Enfin, les traces de la forme $\mathrm{tr}(\gamma \vert W)$ pour $\gamma \in \mathrm{SO}_n(\R)$ sont calculées au moyen de la formule des caractères de Weyl, ou plus exactement de sa version ``dégénérée" étudiée dans \cite[ch. 1]{CC} et dans \cite[ch. 2]{CR}, rappelée ici au paragraphe \ref{3.2}.
\\ \indent Notre algorithme est d'autant plus long à exécuter que $\vert A\vert$ et que $n$ sont grands. C'est pourquoi nous nous restreignons à $n\leq 9$ et même à $q\leq 53$ (pour $n=7$), $q\leq 13$ (pour $n=8$) et $q\leq 7$ (pour $n=9$).
\\ \indent Les cas où $A$ est un $2$-groupe présentent certaines particularités. On les étudie au chapitre \ref{5}. Les orbites des $(\Z/2\Z)^k$-voisins et des $4$-voisins des réseaux $\mathrm{E}_7$ et $\mathrm{E}_8$ sont étudiées au paragraphe \ref{5.1}. On a notamment le résultat suivant :
\begin{prop} Pour $L=\mathrm{E}_7$ ou $L=\mathrm{E}_8$, pour $1\leq i\leq 3$, le groupe $\mathrm{SO}(L)$ agit transitivement sur l'ensemble des $(\Z/2\Z)^i$-voisins de $L$.
\\ \indent Pour $L=\mathrm{E}_7$ ou $L=\mathrm{E}_8$, le groupe $\mathrm{SO}(L)$ agit transitivement sur l'ensemble des $4$-voisins de $L$.
\\ \indent Il y a deux orbites de $2,2,2,2$-voisins de $\mathrm{E}_8$ pour l'action de $\mathrm{SO}(\mathrm{E}_8)$, dont la réunion est l'unique orbite des $2,2,2,2$-voisins de $\mathrm{E}_8$ pour l'action de $\mathrm{O}(\mathrm{E}_8)$.
\end{prop}
\indent Au final, nous obtenons dans tous ces cas de tables des valeurs de $\mathrm{tr}(\mathrm{T}_A \vert \mathcal{M}_W (\mathrm{SO}_n))$ pour des représentations irréductibles $W$ arbitraires. Certaines de ces valeurs sont disponibles dans \cite{Meg}.
\\ \\ \indent Au chapitre \ref{7}, nous rappelons, en suivant Arthur \cite{Art13} et Chenevier-Renard \cite{CR}, comment les formes automorphes pour $\mathrm{SO}_n$ étudiées ci-dessus sont ``construites" à partir de certaines représentations automorphes des groupes linéaires. Cela nous permet, par un procédé de ``récurrence sur $n$" décrit au paragraphe \ref{7.4}, d'utiliser nos calculs pour déterminer des paramètres de Satake des représentations des groupes linéaires mises en jeu. Soyons plus précis.
\\ \indent Soit $n \geq 1$ un entier. Soit $\Pi_{\mathrm{alg}}^\bot(\mathrm{ PGL}_n)$ l'ensemble des (classes d'isomorphisme de) représentations automorphes cuspidales de $\mathrm{GL}_n$ sur $\Q$ ayant les propriétés suivantes :
\begin{itemize}  
\item[(i)] $\pi_p$ est non ramifiée pour tout premier $p$, 
\item[(ii)] $\pi_\infty$ est algébrique régulière, 
\item[(iii)] $\pi$ est isomorphe à sa contragrédiente $\pi^\vee$. 
\end{itemize}

Rappelons la signification de (ii). Suivant Harish-Chandra, $\pi_\infty$ admet un caractère infinitésimal, que l'on peut voir suivant Langlands comme une classe de conjugaison semisimple dans $\mathrm{M}_n(\C)$ (\cite[\S 2]{Lan}). La condition (ii) signifie que les valeurs propres de cette classe de conjugaison sont de la forme $\lambda_1 < \lambda_2 < \cdots < \lambda_n$ avec $\lambda_i \in \frac{1}{2}\Z$ et $\lambda_i - \lambda_j \in \Z$ pour tout $i,j$. Les $\lambda_i$ sont appelés les poids de $\pi$, et vérifient $\lambda_{n-i+1}+\lambda_i =0$ pour $1 \leq i \leq n$ grâce à la condition (iii).
\\ \indent Rappelons enfin que si $\pi \in \Pi_{\mathrm{alg}}^\bot(\mathrm{PGL}_n)$ et si $p$ est premier, alors suivant Langlands \cite{Lan} l'isomorphisme de Satake associe à $\pi_p$ une classe de conjugaison semisimple dans $\mathrm{GL}_n(\C)$, qui sera notée $\mathrm{c}_p(\pi)$.
\\ \indent Dans leur travail \cite{CR}, Chenevier et Renard ont déterminé pour $n \leq 8$ (et $n \neq 7$ en général), le nombre d'éléments de $\Pi_{\mathrm{alg}}^\bot(\mathrm{PGL}_n)$ de poids donné. La question qui s'est posée, en fait le but de notre travail, est d'étudier les paramètres de Satake des ces représentations, du moins pour les premiers poids pour lesquelles il en existe (auquel cas il y en a le plus souvent seulement une ou deux). Soulignons que les résultats de \cite{CR} ne sont plus conditionnels, grâce notamment aux travaux récents de Waldspurger \cite{Wal14}, Kaletha \cite{Kal}, Taïbi \cite{Tai16} et Arancibia-Moeglin-Renard \cite{AMR}.
\\ \indent Nous obtenons les résultats suivants. Les notations $\Delta_{w_1,\dots ,w_n}$, $\Delta^k_{w_1,\dots ,w_n}$, $\Delta^*_{w_1,\dots,w_n}$ et ${\Delta^*}^k_{w_1,\dots,w_n}$ qui interviennent dans les tables \ref{SO9} à \ref{tableau4SO8} sont expliquées au paragraphe \ref{7.3}. 
\begin{theo}\label{theo1} Soient $25\geq w_1>w_2>w_3\geq 1$ des entiers impairs, tels que l'ensemble $\Pi$ des éléments $\pi\in \Pi_{\mathrm{alg}}^\bot(\mathrm{PGL}_6)$ dont les poids sont les $\pm \frac{w_1}{2},\pm \frac{w_2}{2}, \pm \frac{w_3}{2}$est non vide (et possède dans ce cas un ou deux éléments).
\\ \indent $(i)$ Si $\Pi=\{ \pi \}$ est un singleton, alors le polynôme $\mathrm{det}\left( 2^{w_1/2} X\cdot \mathrm{Id} - \mathrm{c}_2(\pi)\right) \in \Z[X]$ est donné par la table \ref{tableau1SO7}.
\\ \indent $(ii)$ Si on a $\vert \Pi \vert =2$, alors le polynôme unitaire dont les racines sont les $2^{w_1/2}\cdot \mathrm{Trace}(\mathrm{c}_2 (\pi) \vert V_{\mathrm{St}})$ pour $\pi \in \Pi$ est donné par la table \ref{tableau2SO7}.
\\ \indent $(iii)$ Pour $p\leq 53$ un nombre premier impair, la quantité $p^{w_1 /2 }\cdot \sum_{\pi \in \Pi } \mathrm{Trace} \left( \mathrm{c}_p (\pi) \vert V_\mathrm{St} \right) \in \Z$ est donnée par les tables \ref{tableau3SO7}, \ref{tableau4SO7} et \ref{tableau5SO7}.
\end{theo}
\begin{theo}\label{theo2} Soient $25\geq w_1>w_2>w_3>w_4\geq 1$ des entiers impairs, tels que l'ensemble $\Pi$ des éléments $\pi\in \Pi_{\mathrm{alg}}^\bot(\mathrm{PGL}_8)$ dont les poids sont les $\pm \frac{w_1}{2},\pm \frac{w_2}{2}, \pm \frac{w_3}{2}, \pm \frac{w_4}{2}$ est non vide.
\\ \indent Pour $p\leq 7$ un nombre premier impair, la quantité $p^{w_1 /2 }\cdot \sum_{\pi \in \Pi } \mathrm{Trace} \left( \mathrm{c}_p (\pi) \vert V_\mathrm{St} \right) \in \Z$ est donnée par la table \ref{tableau1SO9}.
\end{theo}
\begin{theo}\label{theo3} Soient $26\geq w_1>w_2>w_3\geq 2$ des entiers pairs, tels que l'ensemble $\Pi$ des éléments $\pi\in \Pi_{\mathrm{alg}}^\bot(\mathrm{PGL}_7)$ dont les poids sont les $\pm \frac{w_1}{2},\pm \frac{w_2}{2}, \pm \frac{w_3}{2}, 0$ est non vide (et possède dans ce cas un seul élément).
\\ \indent $(i)$ Le polynôme $\mathrm{det}\left( 2^{w_1/2} X\cdot \mathrm{Id} - \mathrm{c}_2(\pi)\right) \in \Z[X]$ pour $\pi$ l'unique élément de $\Pi$ est donné par la table \ref{tableau1SO8}.
\\ \indent $(ii)$ Pour $p\leq 13$ un nombre premier impair, la quantité $p^{w_1 /2 }\cdot  \mathrm{Trace} \left( \mathrm{c}_p (\pi) \vert V_\mathrm{St} \right) \in \Z$ pour $\pi$ l'unique élément de $\Pi$ est donnée par la table \ref{tableau2SO8}.
\end{theo}
\begin{theo}\label{theo4} Soient $26\geq w_1>w_2>w_3>w_4\geq 0$ des entiers pairs, tels que l'ensemble $\Pi$ des éléments $\pi\in \Pi_{\mathrm{alg}}^\bot(\mathrm{PGL}_7)$ dont les poids sont les $\pm \frac{w_1}{2},\pm \frac{w_2}{2}, \pm \frac{w_3}{2}, \pm \frac{w_4}{2}$ est non vide (et possède dans ce cas un seul élément).
\\ \indent $(i)$ Le polynôme $\mathrm{det}\left( 2^{w_1/2} X\cdot \mathrm{Id} - \mathrm{c}_2(\pi)\right) \in \Z[X]$ pour $\pi$ l'unique élément de $\Pi$ est donné par la table \ref{tableau3SO8}.
\\ \indent $(ii)$ Pour $p\leq 13$ un nombre premier impair, la quantité $p^{w_1 /2 }\cdot  \mathrm{Trace} \left( \mathrm{c}_p (\pi) \vert V_\mathrm{St} \right) \in \Z$ pour $\pi$ l'unique élément de $\Pi$ est donnée par la table \ref{tableau4SO8}.
\end{theo}
\indent Pour $\pi \in \Pi_{\mathrm{alg}}^\bot(\mathrm{PGL}_n)$, si l'on pose $\sum_{i=0}^n a_i\cdot X^i= \mathrm{det}\left(2^{w_1/2}X \cdot \mathrm{Id}-\mathrm{c}_2(\pi)\right)$, alors les $a_i$ vérifient : $a_{n-i} = 2^{(n-2i)\cdot w_1/2} \cdot a_i$, et on n'a pas explicité tous les monômes des polynômes donnés aux tables \ref{tableau1SO7}, \ref{tableau1SO8} et \ref{tableau3SO8}.
\\ \\ \indent Signalons que nous disposons de nombreuses indications que nos calculs finaux sont corrects ! Par exemple, notre méthode permet également de déterminer des paramètres de Satake de représentations associées à des formes modulaires classiques, ou de Siegel en genre $2$, cas où ils étaient déjà connus (par exemple par van der Geer \cite{vdG} ou Chenevier-Lannes \cite{CL}).
\\ \indent De plus, nous pouvons souvent calculer de différentes manières un paramètre de Satake donné, et vérifier que les résultats sont bien les mêmes. Signalons enfin que nos résultats montrent que pour les trois représentations $\pi$ de $\Pi_{\mathrm{alg}}^\bot(\mathrm{PGL}_7)$ dont les poids sont de la forme $a+b>a>b>0>-b>-a>-a-b$ avec $a+b \leq 13$ (voir les trois premières lignes de la table \ref{tableau1SO8}), alors le paramètre de Satake $\mathrm{c}_2(\pi)$ est conjugué à $\mathrm{G}_2$, conformément à une conjecture de \cite{CR} (voir la page 10 de l'introduction ainsi que la table 10).
\\ \indent Terminons par mentionner un lien entre ce travail et une conjecture de Bergström, Faber et van der Geer \cite{Fab} sur la fonction zêta de Hasse-Weil de l'espace $\overline{\mathcal{M}_{3,n}}$ de module des courbes stables de genre $3$ munies de $n$ points marqués (qui est propre et lisse sur $\Z$). En effet, ces auteurs ont mis en évidence de manière expérimentale l'existence de deux ``motifs" de poids $23$ et de dimension $6$ dans $\mathrm{H}^{23}(\overline{\mathcal{M}_{3,17}})$, et ont déterminé le polynôme caractéristique de leur Frobenius en $2$. D'autre part, Chenevier et Renard ont trouvé exactement $7$ représentations $\pi \in \Pi_{\mathrm{alg}}^\bot(\mathrm{PGL}_6)$ dont le plus grand poids est $\frac{23}{2}$ (et aucune de plus grand poids $<\frac{23}{2}$). Les calculs faits ici montrent que les polynômes caractéristiques des paramètres de Satake en $p=2$ de deux des $7$ représentations susmentionnées, à savoir celles de poids $\pm \frac{23}{2}, \pm \frac{13}{2}, \pm \frac{5}{2}$ et $\pm \frac{23}{2}, \pm \frac{15}{2}, \pm \frac{3}{2}$, sont exactement ceux trouvés par Bergström, Faber et van der Geer ! Cela répond à une question que nous avaient posée ces auteurs.
\newpage
\tableofcontents
\newpage
\section{Résultats préliminaires.}\label{2}
\indent Dans toute la suite, on se place dans $V=\R^n$ muni de sa structure euclidienne, avec pour base canonique associée $(e_i)_{i\in \{1,\dots ,n\}}$. On notera $(x_i)\cdot (y_i)=\sum_i x_i y_i$ le produit scalaire usuel, et $q:V\rightarrow \R,x\mapsto \frac{x\cdot x}{2}$ la forme quadratique associée.
\subsection{Les réseaux de $\R^n$.}\label{2.1}
\begin{defi}[Réseaux entiers et pairs] Soit $L\subset V$ un réseau. On dit que $L$ est entier si :
$$(\forall x,y\in L)\ x\cdot y \in \Z.$$
\indent Si l'on se donne un réseau $L\subset V$ entier, il est dit pair si :
$$(\forall x\in L)\ x\cdot x\in 2 \Z.$$
\end{defi}
\begin{defi}[Dual et résidu d'un réseau] Soit $L\subset V$ un réseau. On définit $L^\sharp$ le dual de $L$ par :
$$L^\sharp = \{ y\in V \vert (\forall x\in L)\ y\cdot x \in \Z \}.$$
\indent En particulier, $L$ est entier si, et seulement si, $L\subset L^\sharp$. Dans ce cas on définit le résidu de $L$ comme le quotient :
$$\mathrm{r\acute{e}s}\,  L = L^\sharp /L .$$
\indent Ce quotient est muni d'une forme quadratique $\mathrm{r\acute{e}s}\,  L \rightarrow \Q/\Z$ définie par $x\mapsto q(x)\mathrm{\ mod\ }\Z$ appelée forme d'enlacement.
\end{defi}
\begin{defi}[Déterminant d'un réseau] Soit $L$ un réseau entier. On note $\mathrm{det}(L)$ son déterminant, qui est encore le déterminant de la matrice de Gram d'une base quelconque de $L$. On a la relation bien connue :
$$\mathrm{det}(L) = \vert \mathrm{r\acute{e}s}\, L \vert . $$
\end{defi}
\begin{defi}[Racines d'un réseau]
Soit $L\subset V$ un réseau entier. On définit l'ensemble des racines de $L$ comme l'ensemble $R(L)$ (qui est fini, et éventuellement vide) :
$$ R(L) = \{x \in L \vert x\cdot x = 2\}. $$
\indent C'est un système de racines du $\R$-espace vectoriel qu'il engendre au sens de \cite[ch. VI, \S 1.1, définition 1]{Bo}, ce qui justifie la terminologie (c'est même un système de racines de type ADE).
\end{defi}
\indent On reprend les notations de \cite[ch. VI, \S 1]{Bo} pour les notions relatives aux systèmes de racines (systèmes de racines, chambre et groupe de Weyl, longueur d'un élément du groupe de Weyl, diagramme de Dynkin, etc.). On expose ici quelques notations et résultats qu'on utilisera.
\begin{prop}[Les chambres de Weyl et les générateurs du groupe de Weyl] \label{gen Weyl} Soient $R$ un système de racines de $V$, $W$ son groupe de Weyl, et $C$ une chambre de $R$. Alors :
\\ \begin{tabular}{rp{14cm}}
$(i)$ & Pour tout $x\in V$, il existe un élément $w\in W$ tel que $w(x)\in \overline{C}$.\\
$(ii)$ & Pour toute chambre $C'$, il existe un unique élément $w\in W$ tel que $w(C')=C$.\\
$(iii)$ & Le groupe $W$ est engendré par l'ensemble des réflexions orthogonales par rapport aux murs de $C$.\\
\end{tabular}
\end{prop}
\begin{proof}
voir \cite[ch. V, \S 3, théorème 1]{Bo}.
\end{proof}
\indent Le corollaire suivant est une conséquence du $(ii)$ : 
\begin{cor} \label{transitif} Soit $R$ un système de racines, $W$ son groupe de Weyl, $C$ une chambre, et $\rho$ un élément de $C$. Alors :
$$(\forall w,w'\in W) \ w=w'\Leftrightarrow w (\rho) = w'(\rho).$$
\end{cor}
\begin{prop}\label{longueur} Soient $R$ un système de racines, $W$ son groupe de Weyl, $C$ une chambre, et $l:W\rightarrow \N$ la longueur associée à $C$. Soit $h$ un mur de $C$, et $s$ la symétrie orthogonale associée. Si $w\in W$, alors :
\\ \begin{tabular}{rp{14cm}}
$(i)$ &  $l(s\circ w) = l(w) \pm 1$, \\
$(ii)$ & $l(s\circ w) >l(w)$ si, et seulement si, les chambres $C$ et $w(C)$ sont du même côté de $h$.\\
\end{tabular} 
\end{prop}
\begin{proof}
voir \cite[ch. V, \S 3, théorème 1 (ii)]{Bo}. 
\end{proof}
\indent On adoptera les notations suivantes :
\\ \\ \begin{tabular}{rp{14cm}}
$\mathrm{A}_n$ : & On pose $\mathrm{A}_n = \{ (x_i)\in \Z^{n+1} \vert \sum_i x_i=0\}$. On a $R(\mathrm{A}_n)=\{ \pm (e_i-e_j) \vert i\neq j\}$. \\
$\mathrm{D}_n$ : & On pose $\mathrm{D}_n = \{ (x_i)\in \Z^n \vert \sum_i x_i\equiv 0\mathrm{\ mod\ }2\}$. On a $R(\mathrm{D}_n)=\{ \pm e_i \pm e_j \vert i\neq j\}$. \\
$\mathrm{E}_8$ : & On pose $\mathrm{E}_8 = \mathrm{D}_8+ \Z \cdot e$, avec $e=\frac{1}{2}(1,\dots ,1)$. On a $R(\mathrm{E}_8)=R(\mathrm{D}_8) \cup \left\{ (x_i)=\frac{1}{2} (\pm 1, \dots, \pm 1) \vert \prod _i x_i >0 \right\}$. \\
$\mathrm{E}_7$ : & On pose $\mathrm{E}_7 = e^\perp \cap \mathrm{E}_8 =\left\{ (x_i)\in \mathrm{E}_8 \vert \sum_i x_i =0 \right\}$. On a $R(\mathrm{E}_7)=e^\perp \cap R(\mathrm{E}_8) =  R(\mathrm{A}_7)\cup \left\{ (x_i)=\frac{1}{2} (\pm 1, \dots, \pm 1) \vert \sum _i x_i =0 \right\} $. \\
\end{tabular}
\\ \\ \indent Dans la suite, on fera l'abus de langage suivant : si l'on se donne $m>n$ deux entiers, et $L\subset \R^m$ un $\Z$-module libre de rang $n$, alors dira que $L$ est un réseau de $\R^n$. En particulier, le groupes $\mathrm{E}_7$ (respectivement $\mathrm{A}_n$) introduit ci-dessous est un $\Z$-module libre de rang $7$ (respectivement $n$) et sera vu comme un réseau de $\R^7$ (respectivement $\R^n$) et non comme un sous-ensemble de $\R^8$ (respectivement $\R^{n+1}$).
\begin{prop} \label{Weyl et SO} Soient $L$ un réseau entier, $R=R(L)$ et $W$ le groupe de Weyl de $R$. On suppose que $R$ est un système de racines de $V$ (en particulier, $R$ engendre $V$ comme $\R$-espace vectoriel). Soient $D$ le diagramme de Dynkin associé à un ensemble de racines simples $\{ \alpha_1 ,\dots ,\alpha_n \}$ et $G$ le sous-groupe des permutations de $\{ \alpha_1, \dots , \alpha_n \}$ qui sont des automorphismes de $D$. On pose $\mathrm{A}(R)$ le sous-groupe des automorphismes de $V$ qui laissent stable $R$, et $\mathrm{O}(L)$ le sous-groupe de ceux qui laissent stable $L$. On a les inclusions de groupes suivantes :
$$W\subset \mathrm{O}(L) \subset \mathrm{A}(R) \simeq W\ltimes G.$$
\indent De plus, si $L$ est engendré $\Z$-linéairement par $R$, on a l'égalité : $\mathrm{O}(L)=\mathrm{A}(R)$.
\end{prop}
\begin{proof}
\indent Les inclusions $W\subset \mathrm{O}(L)$ et $\mathrm{O}(L)\subset \mathrm{A}(R)$ viennent respectivement du fait que $L$ est un réseau entier, et que le groupe $\mathrm{O}(L)$ préserve l'ensemble $R$.
\\ \indent L'isomorphisme $\mathrm{A}(R) \simeq W\ltimes G$ vient de \cite[ch. VI, \S 1, proposition 16]{Bo}, comme $R$ est un système de racines de $V$.
\\ \indent Enfin, le cas où $L$ est engendré par $R$ est évident.  
\end{proof}
\indent On précise dans le corollaire qui suit les cas que l'on rencontrera le plus souvent, où les notations sont les mêmes qu'à la proposition précédente.
\begin{cor} Pour $L=\mathrm{E}_7$, $L=\mathrm{E}_8$ ou $L=\mathrm{E}_8\oplus \mathrm{A}_1$, on a : $\mathrm{A}(R) =\mathrm{O}(L) =W$.
\\ \indent Pour $L=\mathrm{A}_n$ ($n\geq 2$) ou pour $L=\mathrm{D}_n$ ($n\geq 5$), on a $G\simeq \mathcal{S}_2 \simeq \Z/2\Z$. Pour $L=\mathrm{D}_4$, on a $G\simeq \mathcal{S}_3$. Dans tous ces cas, on a : $\mathrm{O}(L)=\mathrm{A}(R)$.
\end{cor}
\begin{proof}
On vérifie dans un premier temps que tous ces réseaux sont bien engendrés comme $\Z$-modules par leurs racines. Il suffit ensuite de calculer le groupe $G$ de la proposition précédente, ce qui se fait facilement. Notons par exemple que ce groupe est trivial lorsque le réseau $L$ considéré est $\mathrm{E}_7$, $\mathrm{E}_8$ ou $\mathrm{E}_8\oplus \mathrm{A}_1$.  
\end{proof}
\subsection{Les formes automorphes et les opérateurs de Hecke.}\label{2.2}
\subsubsection{Les formes automorphes.}
\begin{defi}[L'ensemble $\mathcal{L}_n$] Soit $n\equiv 0,\pm 1\ \mathrm{mod}\ 8$. On définit $\mathcal{L}_n$ comme l'ensemble des réseaux pairs $L\subset \R^n$ tels que $\mathrm{det}(L)=1$ si $n$ est pair et $\mathrm{det}(L)=2$ sinon.
\end{defi}
\indent On rappelle que $\mathcal{L}_n$ est non vide pour $n\equiv 0,\pm 1\mathrm{\ mod\ }8$. Par exemple, suivant les notations précédentes, $\mathcal{L}_n$ contient :
\\ \indent - le réseau $\mathrm{E}_8^{(n-7)/8}\oplus \mathrm{E}_7$ si $n\equiv -1\mathrm{\ mod\ }8$ ;
\\ \indent - le réseau $\mathrm{E}_8^{n/8}$ si $n\equiv 0\mathrm{\ mod\ }8$ ;
\\ \indent - le réseau $\mathrm{E}_8^{(n-1)/8}\oplus \mathrm{A}_1$ si $n\equiv 1\mathrm{\ mod\ }8$.
\\\\ \indent Si on se donne $L_0$ un élément de $\mathcal{L}_n$, on définit $\mathrm{O}_n$ le schéma en groupes affine sur $\Z$ associé à la forme quadratique $L_0 \rightarrow\Z$, $x\mapsto q(x)$. Il s'agit de l'objet noté $\mathrm{O}_{L_0}$ dans \cite[ch. II, \S1]{CL}. On définit de même $\mathrm{SO}_n \subset \mathrm{O}_n$ (introduit aussi dans \cite[ch. II, \S1]{CL}).
\\ \indent Pour faire court, on appellera dans la suite $\Z$-groupe un schéma en groupes affine sur $\Z$ et de type fini, de sorte que $\mathrm{O}_n$ et $\mathrm{SO}_n$ sont des $\Z$-groupes (ce dernier étant même réductif).
\\ \indent La définition générale de la théorie des formes automorphes s'y applique, et se réduit à la définition suivante qui sera amplement suffisante pour nos besoin (voir par exemple \cite[ch. IV, \S3]{CL}).
\begin{defi}[Les formes automorphes pour $\mathrm{O}_n$ et $\mathrm{SO}_n$] Soit $(W,\rho)$ une représentation de dimension finie sur $\C$ de $\mathrm{O}_n(\R)$. L'espace des formes automorphes de poids $W$ pour $\mathrm{O}_n$ est défini comme :
$$\mathcal{M}_W(\mathrm{O}_n) = \left\{ f:\mathcal{L}_n \rightarrow W \vert \left( \forall \gamma \in \mathrm{O}_n(\R) \right) f(\gamma \cdot L) = \rho(\gamma) \cdot f(L) \right\}.$$
\indent Pour $W'$ une représentation de dimension finie de $\mathrm{SO}_n(\R)$ sur $\C$, on définit de même l'espace $\mathcal{M}_{W'}(\mathrm{SO}_n)$ des formes automorphes de poids $W'$ pour $\mathrm{SO}_n$.
\\ \indent Ces deux espaces sont de dimension finie.
\end{defi}

\subsubsection{Les $A$-voisins.}
\indent On renvoie par exemple à \cite[ch. II, \S 1]{CL} pour les définitions des lagrangiens et des lagrangiens transverses associés à une forme quadratique. Rappelons la définition des $A$-voisins :
\begin{propdef}[Les $A$-voisins] Soient $A$ un groupe abélien fini, et $L_1,L_2$ deux éléments de $\mathcal{L}_n$. Les conditions suivantes sont équivalentes :
\\ \begin{tabular}{rp{14cm}}
$(i)$ & Le quotient $L_1/(L_1\cap L_2)$ est isomorphe à $A$.\\
$(ii)$ & Le quotient $L_2/(L_1\cap L_2)$ est isomorphe à $A$.\\
\end{tabular}
\indent Si ces conditions sont vérifiées, on dit que $L_1$ et $L_2$ sont $A$-voisins, ou que $L_2$ est un $A$-voisin de $L_1$.
\end{propdef}
\begin{proof}
voir \cite[ch. III, \S 1]{CL} et \cite[Annexe B, \S 3]{CL} selon la parité de $n$.  
\end{proof}
\indent Dans le cas particulier où $A$ est de la forme $\Z/d\Z$, on parlera de $d$-voisin (et plus généralement de $d,\dots,d$-voisin si $A=\Z/d\Z\times \dots \times \Z/d\Z$). Dans la suite, $A$ désignera un groupe abélien quelconque. On s'intéressera plus particulièrement aux cas où $A$ est de la forme $\Z/d\Z$ (où $d\in \N^*$) ou de la forme $\Z/p\Z\times \dots \times \Z/p\Z$ (où $p$ est un nombre premier).
\\ \indent Le lemme technique suivant nous sera utile par la suite :
\begin{lem} \label{voisins} Soient $L_1$ et $L_2$ deux $A$-voisins. On pose : $M=L_1 \cap L_2$, $I_1 = L_1/M$, $I_2 = L_2 /M$, et $R=(L_1^\sharp \cap L_2^\sharp )/(L_1 \cap L_2)$.
\\ \indent Les inclusions de $L_1$, $L_2$ et $L_1^\sharp \cap L_2^\sharp$ dans $M^\sharp$ induisent l'isomorphisme canonique de groupes abéliens :
$$I_1 \oplus I_2 \oplus R \simeq \mathrm{r\acute{e}s}\, M$$
\indent De plus, l'accouplement $I_1 \times I_2 \rightarrow \Q/\Z$ induit par la forme d'enlacement de $\mathrm{r\acute{e}s}\, M$ est non dégénéré. Pour cette forme, les sous-modules $I_1 \oplus I_2$ et $R$ sont orthogonaux, $\mathrm{r\acute{e}s}\, M$ est canoniquement isomorphe à $\mathrm{H}(I_1) \oplus \mathrm{r\acute{e}s}\, L_1$, et $I_2$ est alors un lagrangien de $\mathrm{H}(I_1)$ transverse à $I_1$ et orthogonal à $\mathrm{r\acute{e}s}\,  L_1$.
\end{lem}
\begin{proof}
voir \cite[ch. III, \S 1, proposition 1.1]{CL} et \cite[Annexe B, \S 3, proposition 3.1]{CL} selon la parité de $n$.  
\end{proof}
\indent L'utilisation que l'on fera des $A$-voisins nous impose de prendre un point de vue asymétrique : on souhaite déterminer, pour un réseau $L$ fixé, l'ensemble de ses $A$-voisins. On rappelle que la notation $C_L(\Z/d\Z)$ a été introduite en introduction. En tant que sous-groupes de $\mathrm{O}(V)$, $\mathrm{O}(L)$ et $\mathrm{SO}(L)$ agissent naturellement sur l'ensemble $C_L(\Z/d\Z)$ et sur l'ensemble des $A$-voisins de $L$. La proposition suivante nous donne, selon les choix de $A$, une paramétrisation des $A$-voisins d'un réseau $L$ donné :
\begin{prop}  \label{paramètre A-vois} Soit $L\in \mathcal{L}_n$.
\\ \begin{tabular}{rp{14cm}}
$(i)$ & Si $A=\Z/d\Z$, alors les $A$-voisins de $L$ sont en bijection naturelle avec les points de la quadrique $C_L (\Z/d\Z)$.\\
$(ii)$ & Si $A=\left( \Z/p\Z \right) ^r$, alors les $A$-voisins de $L$ sont en bijection avec les couples $(X,I)$, où $X$ est un espace totalement isotrope de $L/pL$ de dimension $r$, et $I$ un lagrangien de $\mathrm{H}(L/M)$ (avec $M$ l'image réciproque de $X^\perp$ par $L\rightarrow L/pL$) transverse à $L/M$. \\
\end{tabular}
\indent Ces deux bijections sont détaillées dans les propositions-définitions \ref{p-voisin} et \ref{ppp-voisin} qui suivent. De plus, elles commutent aux actions naturelles de $\mathrm{SO}(L)$. 
\end{prop}
\begin{proof}
Pour le point $(i)$ : la bijection entre les $A$-voisins de $L$ est détaillée dans \cite[ch.III, \S 1]{CL}. On montre à la proposition-définition \ref{p-voisin} qu'elle commute bien aux actions de $\mathrm{SO}(L)$ sur l'ensemble des $A$-voisins et sur $C_L(\Z/d\Z)$ (en donnant explicitement cette bijection).
\\ \indent Pour le point $(ii)$ : on se donne $L\in \mathcal{L}_n$, et $L'$ un $A$-voisin de $L$, et on se place pour simplifier dans le cas où $n\equiv 0\mathrm{\ mod\ }8$ (les autres cas se traitant de la même manière). On note $M=L\cap L'$, $I=L/M$ et $I'=L'/M$. Pour simplifier, on pose aussi $\phi : L\rightarrow L/pL$ la réduction modulo $p$ dans $L$. On fait les constatations suivantes :
\\ \indent - Du fait des inclusions $pL' \subset M\subset L$, le $\Z/p$-espace vectoriel $X=\phi (pL')$ a bien un sens, et c'est même un $\Z/p$-espace vectoriel totalement isotrope de dimension $r$ dans $L/pL$. L'isotropie vient du fait que $L'$ est entier (donc l'image par $\phi$ de tous les éléments de $pL'$ sont isotropes dans $L/pL$), et la dimension vient des isomorphismes évidents : $X\simeq pL'/(pL\cap pL') \simeq L'/M \simeq A$ (où le dernier isomorphisme provient du fait que $L$ et $L'$ sont des $A$-voisins).
\\ \indent - Notons que $M$ et $X$ satisfont bien à l'égalité $\phi (M) = X^\perp$. Comme $M=L\cap L'\subset L'$ et que $L'$ est entier, on déduit que : $(\forall x\in pL')(\forall y\in M) \ x\cdot y\equiv 0\mathrm{\ mod\ }p$, et ainsi on a déjà l'inclusion $\phi(M)\subset X^\perp$. L'égalité vient alors de l'égalité des dimensions de $\phi (M)$ et de $X^\perp$ (vus comme $\Z/p$-espaces vectoriels). On a en effet : $\mathrm{dim}_{\Z/p} \phi(M) = n-r = \mathrm{dim}_{\Z/p} \phi(L)-\mathrm{dim}_{\Z/p} X = \mathrm{dim}_{\Z/p} (X^\perp)$. L'inclusion $pL\subset M$ nous permet de dire que $M$ est bien l'image réciproque de $X^\perp=\phi (M)$ par $\phi$.
\\ \indent - D'après le lemme \ref{voisins}, les $\Z/p$-espaces vectoriels $I$ et $I'$ sont deux lagrangiens transverses de $\mathrm{r\acute{e}s}\, M\simeq \mathrm{H}(A)$. De plus, $L'$ est entièrement déterminé par le choix de $M$ et de $I'$, puisque $L'$ est l'image inverse de $I'$ par l'application $M^\sharp \rightarrow \mathrm{r\acute{e}s}\, M$.
\\ \indent D'après ce qui précède, l'application $L'\mapsto (X,I')$ est bien définie et est injective. De plus, le réseau $M$ est bien l'image inverse de $X^\perp$ par $L\rightarrow L/pL$. La proposition-définition \ref{ppp-voisin} montre qu'elle est aussi surjective. On montrera en effet dans cette proposition-définition que, une fois les réseaux $L$ et $M$ fixés, on peut créer autant de $A$-voisins $L'$ de $L$ tels que $L\cap L'=M$ qu'il y a de lagrangiens de $\mathrm{H}(L/M)$ transverses à $L/M$. On montre dans la proposition-définition \ref{ppp-voisin} que cette bijection commute bien aux actions de $\mathrm{SO}(L)$.  
\end{proof}
\begin{propdef}[La création des $d$-voisins]\label{p-voisin} Soient $L$ un réseau de $\mathcal{L}_n$ et $d\in \N^*$. Si l'on se donne une droite isotrope $x\in C_L(\Z/d\Z)$, on peut lui associer le module $M$, image inverse de $x^\perp$ par l'homomorphisme $L\rightarrow L/dL$. Choisissons enfin $v\in L$, dont l'image dans $L/dL$ engendre $x$, et tel que $v\cdot v \equiv 0\ \mathrm{mod}\ 2d^2$. Alors le réseau $L'(x)$ défini par :
$$L'(x) = \Z\cdot \dfrac{v}{d}+M$$
est un $d$-voisin de $L$ qui ne dépend que du choix de $x$.
\\ \indent De plus, l'application $x\mapsto L'(x)$ est une bijection entre $C_L(\Z/d\Z)$ et l'ensemble $\mathrm{Vois}_d(L)$ des $d$-voisins de $L$.
\end{propdef}
\begin{proof}
la nature bijective de cette application est développée dans \cite[ch. III, \S 1, propositions 1.4 et 1.5]{CL}. Il suffit de vérifier que cette bijection commute aux actions de $\mathrm{SO}(L)$. Si l'on se donne $\gamma \in \mathrm{SO}(L)$, alors $\gamma$ conserve l'orthogonalité ainsi que le produit scalaire, et on a les implications suivantes :
$$\left\{ 
	\begin{aligned}
		M=x^\perp &\Rightarrow \gamma (M) = (\gamma (x))^\perp , \\
		v\cdot v \equiv 0\ \mathrm{mod}\ 2d^2 &\Rightarrow \gamma (v) \cdot \gamma (v) \equiv 0\ \mathrm{mod}\ 2d^2 , \\
		(\forall p\vert d)\  v\notin p L & \Rightarrow (\forall p\vert d)\  \gamma (v) \notin p L.
	\end{aligned}
\right.$$
Ainsi, on a l'égalité : $\gamma (L'(x)) = L'(\gamma (x))$.  
\end{proof}
\begin{propdef}[La création des $p,\dots ,p$-voisin] \label{ppp-voisin} Soient $L$ un réseau de $\mathcal{L}_n$ et $A=\left( \Z/p\Z\right) ^r$. Soient $X$ un espace totalement isotrope de $L/pL$ de dimension $r$, et $(x_i)$ une base de $X$. On considère une famille $(v_i)$, avec $v_i\in L$ et $v_i\equiv x_i\mathrm{\ mod\ }pL$, qui vérifie : 
$$
	\left\{ 
	\begin{aligned}
		&v_i \cdot v_i \equiv 0 \mathrm{\ mod\ }2p^2  ,\\
		&v_i \cdot v_j\equiv 0\mathrm{\ mod\ }p^2\text{ pour }i\neq j .\\
	\end{aligned}
	\right.
$$
\indent Alors le réseau $L'((v_i)_i)$ défini par :
$$\left\{
	\begin{aligned}
		& M =\{ v\in L \vert (\forall i ) v_i \cdot v\equiv 0\mathrm{\ mod\ }p \} ,\\
		& L'\left( (v_i)_i \right) =M+ \sum_i \Z \dfrac{v_i}{p} ,\\
	\end{aligned}
	\right.$$
est un $A$-voisin de $L$ tel que $M=L\cap L'\left( (v_i)_i \right)$. Le réseau $M$ ne dépend que du choix de $X$, et est égal à l'image réciproque de $X^\perp$ par la projection $L\rightarrow L/pL$.
\\ \indent De plus, une fois la famille $(x_i)$ choisie (et donc une fois le réseau $M$ fixé), l'ensemble des réseaux $L'((v_i)_i)$ ainsi obtenus (qui ne dépendent que du choix des relèvements $(v_i)$) décrit l'ensemble des $A$-voisins $L'$ de $L$ tels que $L\cap L' =M$.
\\ \indent Le réseau $L'((v_i)_i)$ sera appelé le $p,\dots ,p$-voisin de $L$ associé à la famille $(v_i)$.
\end{propdef}
\begin{proof}
Montrons d'abord que $L'$ est un $A$-voisin de $L$ tel que $M=L\cap L'$. Comme la famille $(x_i)$ est $\Z/p$-libre dans $L/pL$, on déduit déjà que $M=L\cap L'$. Pour la même raison, l'image par $L'\rightarrow L'/M$ de la famille $(v_i/p)$ est $\Z/p$-libre dans $L'/M$ : par définition de $L'$, c'est une $\Z/p$-base de $L'/M$, et on déduit que $L'/M\simeq A$.
\\ \indent Il faut maintenant montrer que $L'\in \mathcal{L}_n$, c'est-à-dire que $L'$ est pair et que $\mathrm{det}(L)=\mathrm{det}(L')$. Le premier point provient des congruences satisfaites par les $v_i\cdot v_j$ et de la définition de $M$. L'égalité $\mathrm{det}(L)=\mathrm{det}(L')$ vient du fait que l'on a aussi $L/M\simeq A$. En effet, comme le produit scalaire est non dégénéré dans $L/pL$, on peut trouver une famille $(u_i)$ d'éléments de $L$ vérifiant :
$$
	\left\{ 
	\begin{aligned}
		&u_i \cdot x_i \equiv 1 \mathrm{\ mod\ }p  ,\\
		&u_i \cdot x_j\equiv 0\mathrm{\ mod\ }p\text{ pour }i\neq j .\\
	\end{aligned}
\right.$$
L'image de la famille $(u_i)$ par $L\rightarrow L/M$ est une $\Z/p$-base de $L/M$, et on a bien $L/M\simeq A$.
\\ \indent Ainsi, $L$ et $L'$ sont bien des $A$-voisins qui vérifient $L\cap L'=M$.
\\ \\ \indent Montrons maintenant que tous les $A$-voisins $L'$ de $L$ tels que $L\cap L'=M$ sont obtenus de cette façon. La proposition \ref{paramètre A-vois} nous dit déjà qu'il y en a au plus autant que de lagrangiens transverses à $A$ dans $\mathrm{H}(A)$. Il suffit donc de construire autant de tels $A$-voisins qu'il y a de lagrangiens transverses à $A$ dans $\mathrm{H}(A)$ pour conclure. On rappelle au passage que les lagrangiens transverses à $A$ dans $\mathrm{H}(A)$ sont en bijection avec les formes alternées sur $A$ : il y en a donc autant que de matrices antisymétriques à diagonale nulle de taille $r\times r$ à coefficients dans $\Z/p$ (puisque $A=(\Z/p\Z)^r$ ici).
\\ \indent Reprenons les notations de la proposition. Soient $(v_i)$ et $(v'_i)$ deux familles avec $v_i,v'_i \in L$ et $v_i\equiv v'_i \equiv x_i\mathrm{\ mod\ }pL$ qui vérifient les congruences :
$$
	\left\{ 
	\begin{aligned}
		&v_i \cdot v_i \equiv v'_i \cdot v'_i \equiv 0 \mathrm{\ mod\ }2p^2 ,\\
		&v_i \cdot v_j\equiv v'_i \cdot v'_j\equiv 0\mathrm{\ mod\ }p^2\text{ pour }i\neq j .\\
	\end{aligned}
	\right.
$$
\indent On pose $v'_i=v_i+p\cdot w_i$, avec $w_i\in L$. Les $v'_i$ vérifient les congruences précédentes si, et seulement si, les $w_i$ vérifient :
$$
	\left\{ 
	\begin{aligned}
		&w_i \cdot x_i \equiv 0 \mathrm{\ mod\ }p ,\\
		&w_i \cdot x_j + w_j \cdot x_i \equiv 0\mathrm{\ mod\ }p\text{ pour }i\neq j .\\
	\end{aligned}
	\right.
$$
\indent Enfin, les réseaux $L'((v_i)_i)$ et $L'((v'_i)_i)$ sont égaux si, et seulement si : $(\forall i)\  w_i \in M$, c'est-à-dire si, et seulement si :
$$(\forall i,j)\  w_i \cdot x_j \equiv 0 \mathrm{\ mod\ }p.$$
\indent Ainsi, la matrice $\left( w_i \cdot x_j\mathrm{\ mod\ }p \right)_{i,j}$ est une matrice antisymétrique à diagonale nulle. De plus, elle est nulle si, et seulement si, les réseaux $L'((v_i)_i)$ et $L'((v'_i)_i)$ sont égaux.
\\ \indent Il reste donc à montrer que toutes les matrices antisymétriques à diagonale nulle de taille $r\times r$ et à coefficients dans $\Z/p$ peuvent être ainsi obtenues, ce qui vient du fait que le produit scalaire est non dégénéré dans $L/pL$. Au final, on déduit qu'il y a exactement autant de $A$-voisins $L'$ de $L$ tels que $L\cap L'=M$ que de lagrangiens de $\mathrm{H}(A)$ transverses à $A$.
\\ \indent De même que dans la démonstration de la proposition-définition \ref{p-voisin}, il est facile de voir que, pour $\gamma\in \mathrm{SO}(L)$, on a l'égalité : $\gamma(L'((v_i)_i) )= L'(\gamma(v_i)_i)$. On en déduit finalement que la bijection entre les $A$-voisins de $L$ et les couples de la forme $(X,I)$, où $X$ est un espace totalement isotrope de $L/pL$ de dimension $r$, et $I$ un lagrangien de $\mathrm{H}(L/M)$ (avec $M$ l'image réciproque de $X^\perp$ par $L\rightarrow L/pL$) commute bien aux actions naturelles de $\mathrm{SO}(L)$.  
\end{proof}
\subsubsection{L'anneau des opérateurs de Hecke associé aux $A$-voisins.}
\indent On rappelle qu'on désigne ici par $A$-groupe un $A$-schéma en groupes affine et de type fini. On donne ici quelques rappels classiques qui suivent la présentation et les notations de \cite[ch. IV, \S 2]{CL}.
\begin{defi}[L'anneau des opérateurs de Hecke] \label{H(X)} Soit $\Gamma$ un groupe, et soit $X$ un $\Gamma$-ensemble transitif. On définit l'anneau des opérateurs de Hecke de $X$ comme le sous-anneau $\mathrm{H}(X)\subset \mathrm{End}_\Z (\Z[X])$ des endomorphismes commutant à l'action de $\Gamma$.
\end{defi}
\begin{defi}[L'anneau de Hecke d'un $\Z$-groupe] \label{H(G)} Soit $G$ un $\Z$-groupe. Si l'on note $P$ l'ensemble des nombres premiers, on note $\widehat{\Z}=\prod_{p\in P}\Z_p$, et $\mathbb{A}_f =\Q \otimes \widehat{\Z}$ l'anneau des adèles finis de $\Q$. On définit alors le $G(\mathbb{A}_f)$-ensemble : $\mathcal{R}(G)=G(\mathbb{A}_f)/G(\widehat{\Z})$. L'anneau de Hecke de $G$ est alors défini comme :
$$\mathrm{H}(G) = \mathrm{H}(\mathcal{R}(G))$$
où $G(\mathbb{A}_f)$ joue le rôle de $\Gamma$ dans la définition précédente.
\end{defi}
\begin{propdef} \label{Hp(G)} On considère $G$ un $\Z_p$-groupe, et on garde les notations précédentes. Pour $p\in P$, on définit alors le $G(\Q_p)$-ensemble : $\mathcal{R}_p(G)=G(\Q_p)/G(\Z_p)$. On pose :
$$\mathrm{H}_p(G) = \mathrm{H}(\mathcal{R}_p(G))$$
où $G(\Q_p)$ joue le rôle de $\Gamma$.
\\ \indent On a un homomorphisme d'anneaux injectif canonique : $\mathrm{H}_p (G) \rightarrow \mathrm{H}(G)$, et on verra donc simplement $\mathrm{H}_p(G)$ comme un sous-ensemble de $\mathrm{H}(G)$.
\\ \indent On a alors l'isomorphisme suivant :
$$\bigotimes_{p\in P} \mathrm{H}_p(G) \overset{\sim}{\rightarrow} \mathrm{H}(G).$$
\end{propdef}
\begin{proof}
voir \cite[ch. IV, \S 2.5]{CL}.  
\end{proof}
\indent On s'intéressera aux cas où $G$ est le $\Z$-groupe $\mathrm{O}_n$ ou $\mathrm{SO}_n$, définis au paragraphe \ref{2.2}. Dans ces cas, les anneaux $\mathrm{H}(G)$ sont décrits en détail dans \cite[ch.IV, \S 2]{CL}, et on rappelle ici quelques points importants pour nous.
\\ \indent Fixons $G$ l'un des deux $\Z$-groupes $\mathrm{O}_n$ ou $\mathrm{SO}_n$. On rappelle que $G$ a été défini au moyen d'un réseau $L_0 \in \mathcal{L}_n$. On vérifie facilement que $\mathcal{R}(G)$ s'identifie naturellement à l'ensemble $\{ L\in \mathcal{L}_n \vert L\subset L_0\otimes \Q \}$ (voir \cite[ch. IV, \S 1.2 et 4.4]{CL}).
\\ \indent On montre ensuite que l'application $G(\R)\times \mathcal{R}(G) \rightarrow \mathcal{L}_n,\ (g,L)\mapsto g^{-1}L$ induit une bijection :
$$G(\Q)\setminus G(\R)\times \mathcal{R}(G) \overset{\sim}{\rightarrow} \mathcal{L}_n ,$$
d'après \cite[ch. IV, \S 4.5]{CL}.
\\ \indent L'action naturelle de $\mathrm{H}(G)$ sur $\Z[\mathcal{R}(G)]$ induit une action naturelle de $\mathrm{H}(G)$ sur $\Z [\mathcal{L}_n]$ (par endomorphismes $G(\R)$-équivariants). Cette action est très concrète. Par exemple, pour tout groupe abélien fini $A$, on dispose d'un opérateur $\mathrm{T}_A$ associé à la notion de $A$-voisin (voir \cite[ch.IV, \S 2.6]{CL}) :
\begin{defi}[Les opérateurs de Hecke sur les réseaux] Soit $A$ un groupe abélien fini. L'opérateur de Hecke $\mathrm{T}_A$ associé à $A$ est l'élément de $\mathrm{H}(\mathrm{O}_n)$ défini par :
$$ (\forall L\in \mathcal{L}_n)\ \mathrm{T}_A(L) = \displaystyle{\sum_{L'\ A\mathrm{-voisin\ de\ }L}} L'.$$
\indent De plus, l'inclusion $\mathrm{SO}_n\rightarrow \mathrm{O}_n$ identifie canoniquement $\mathrm{H}(\mathrm{O}_n)$ à un sous-anneau de $\mathrm{H}(\mathrm{SO}_n)$. L'opérateur $\mathrm{T}_A$ défini ci-dessus peut donc aussi être vu comme un élément de $\mathrm{H}(\mathrm{SO}_n)$.
\end{defi}
\begin{defi}[Les opérateurs de Hecke sur les formes automorphes] Soit $f\in \mathcal{M}_W(\mathrm{SO}_n)$. On peut voir $f$ comme une application $\Z[\mathcal{L}_n] \rightarrow W$, donc on peut faire agir à droite l'opérateur de Hecke $\mathrm{T}_A$ sur $f$. L'action de $\mathrm{T}_A$ est donnée plus précisément par l'égalité :
$$(\forall L\in \mathcal{L}_n)\  \mathrm{T}_A(f)(L) = \displaystyle{\sum_{L'\ A\mathrm{-voisin\ de\ }L}} f(L').$$
\end{defi}
\section{Détermination d'une formule pour calculer la trace d'un opérateur de Hecke.}\label{3}
\indent Notre but est de calculer la trace des opérateurs de Hecke $\mathrm{T}_A$ agissant sur l'espace $\mathcal{M}_W(\mathrm{SO}_n)$  des formes automorphes de poids $W$, où $W$ est une représentation irréductible de $\mathrm{SO}_n(\R)$.
\subsection{La méthode utilisée.}\label{3.1}
\indent Soient $n\equiv 0,\pm 1\mathrm{\ mod\ }8$, et $W$ une représentation de $\mathrm{SO}_n(\R)$.
\\ \indent On pose $\widetilde X_n=\mathrm{SO}_n(\R)\setminus \mathcal{L}_n$, avec $L_1,\dots ,L_h$ des représentants de chaque classe. Si l'on se donne $f\in \mathcal{M}_W(\mathrm{SO}_n)$, alors chacun des $f(L_i)$ est un élément de $W$ invariant par $\mathrm{SO}(L_i)$ (où $\mathrm{SO}(L_i)=\{\gamma \in \mathrm{SO}_n (\R) \vert \gamma L_i=L_i\}$). Pour $i=1,\dots h$, on pose $W^{\mathrm{SO}(L_i)}=\{v\in W\vert (\forall \gamma \in \mathrm{SO}(L_i) )\ \rho (\gamma)(v) =v\}$ l'espace des éléments de $W$ stables par l'action de $\mathrm{SO}(L_i)$. On dispose d'une application $\C$-linéaire :
$$
	\begin{array}{rcl}
		\mathcal{M} _W (\mathrm{SO}_n)  & \longrightarrow & \displaystyle\prod_{i=1}^h W^{\mathrm{SO}(L_i)}\\
		f & \longmapsto & \left( f(L_i)\right)_{i\in \{1,\dots,h\}} .\\
	\end{array}
$$
\indent Le lemme suivant est évident :
\begin{lem} Le morphisme ci-dessus est un isomorphisme.
\end{lem}
\indent La formule qui sera notre point de départ est donnée par la proposition suivante :
\begin{prop} Soient $n\geq 1$, $(W,\rho)$ une représentation de dimension finie de $\mathrm{SO}_n (\R)$ sur $\C$, et $A$ un groupe abélien fini. Pour tout $i\in \{1,\dots ,h\}$, on pose $\mathrm{vois}_A(L_i)$ l'ensemble des $A$-voisins de $L_i$, et $\mathcal{V}_i$ le sous ensemble de $\mathrm{vois}_A(L_i)$ des $A$-voisins de $L_i$ isomorphes à $L_i$. L'ensemble $\mathcal{V}_i$ est stable pour l'action de $\mathrm{SO}(L_i)$. On peut donc écrire $\mathcal{V}_i$ comme réunion de ses orbites par l'action de $\mathrm{SO}(L_i)$ :
$$\mathcal{V}_i = \underset{j\in J_i}{\coprod} \mathrm{SO}(L_i) g_{i,j}  L_i = \underset{j\in J_i}{\coprod}\mathcal{V}_{i,j} $$
où $g_{i,j}\in \mathrm{SO}_n(\R)$, et $\mathcal{V}_{i,j}$ est l'orbite des $A$-voisins de $L_i$ pour l'action de $\mathrm{SO}(L_i)$ contenant $g_{i,j} L_i$. On a l'égalité :
$$ \mathrm{tr}\left( \mathrm{T}_A \vert \mathcal{M}_W(\mathrm{SO}_n)\right) = \displaystyle{\sum_{i=1}^h}\left( \dfrac{1}{\vert \mathrm{SO}(L_i) \vert } \cdot \left({\displaystyle{\sum_{\substack{j\in J_i}} \vert \mathcal{V}_{i,j} \vert \cdot \sum_{\gamma \in \mathrm{SO}(L_i)}}\mathrm{tr}\left( \gamma g_{i,j} \vert W \right) }\right)\right) .$$
\end{prop}
\begin{proof}
Soit $i\in \{1,\dots,h\}$. On choisit une famille $\gamma_{i,k}$ d'éléments de $\mathrm{SO}_n(\R)$, et $a_{i,k}$ une famille d'éléments de $\{1,\dots,h\}$ telles que l'on ait l'écriture :
$$\mathrm{vois}_A(L_i) = \left\{ \gamma_{i,k} L_{a_{i,k}} \vert k\in K_i \right\},$$
où $K_i$ est un ensemble quelconque de même cardinal que $\mathrm{vois}_A(L_i)$. Avec ces notations, notons que l'on a l'égalité : $\mathcal{V}_i = \left\{ \gamma_{i,k} L_{a_{i,k}} \vert k\in K_i\text{ et }a_{i,k}=i \right\}$. Si $f$ est dans $\mathcal{M}_W(\mathrm{SO}_n)$, on a alors :
$$
	\begin{aligned}
		\mathrm{T}_A(f) (L_i) &= \underset{L\in \mathrm{vois}_A(L_i)}{\sum} f(L) = \underset{k\in K_i}{\sum} f(\gamma_{i,k}\cdot L_{a_{i,k}})= \underset{k\in K_i}{\sum} \rho(\gamma_{i,k}) \cdot f(L_{a_{i,k}}) . \\
	\end{aligned}
$$
\indent Pour $i_1,i_2\in \{1,\dots,h\}$, on définit les endomorphismes :
$$\underset{\substack{k\in K_{i_1} \\ a_{i_1,k}=i_2}}{\sum} \rho (\gamma_{i_1,k}) = u_{i_1,i_2}\in \mathrm{End}(W) .$$
\indent En particulier, $u_{i_1,i_2}$ satisfait $u_{i_1,i_2}\left(W^{\mathrm{SO}(L_{i_1})} \right) \subset W^{\mathrm{SO}(L_{i_2})}$ et on a le diagramme commutatif suivant :
$$\xymatrix{
    \mathcal{M}_W(\mathrm{SO}_n) \ar[r]^\sim \ar[d]_{\mathrm{T}_A} & \displaystyle{\prod_{i=1}^h} W^{\mathrm{SO}(L_i)} \ar[d]^{(u_{i_1,i_2})_{i_1,i_2\in \{1,\dots ,h\}} } \\
    \mathcal{M}_W(\mathrm{SO}_n) \ar[r]^\sim & \displaystyle{\prod_{i=1}^h} W^{\mathrm{SO}(L_i)}
}$$
\indent Pour $i\in \{ 1,\dots ,h\}$, on définit les projecteurs $p_i \in \mathrm{End}(W)$ par :
$$p_i = \dfrac{1}{\vert \mathrm{SO}(L_i) \vert} \underset{\gamma \in \mathrm{SO}(L_i)}{\sum} \rho (\gamma) .$$
\indent L'image de $p_i$ est $W^{\mathrm{SO}(L_i)}$, et $\mathrm{Ker}(p_i)$ en est un supplémentaire dans $W$.
\\ \indent Dans le calcul de la trace qui nous intéresse, seuls les termes diagonaux (c'est-à-dire les $u_{i,i}$) interviendront. De plus, ils vérifient les égalités :
$$
\left\{
	\begin{aligned}
		&(u_{i,i}\circ p_i)\vert _{W^{\mathrm{SO}(L_i)}}  = u_{i,i}\vert _ {W^{\mathrm{SO}(L_i)}}\\
		&(u_{i,i}\circ p_i)\vert _{\mathrm{Ker}(p_i)}  = 0\\
	\end{aligned}
\right.
$$
\indent Et ainsi :
$$\mathrm{tr}(u_{i,i}\vert W^{\mathrm{SO}(L_i)}) = \mathrm{tr}(u_{i,i}\circ p_i \vert W) =\mathrm{tr}(p_i\circ u_{i,i} \vert W). $$
\indent D'où finalement :
$$
	\begin{aligned}
		\mathrm{tr}\left(\mathrm{T}_A \vert \mathcal{M}_W(\mathrm{SO}_n)\right) & = \mathrm{tr}\left( (u_{i_1,i_2})_{i_1,i_2\in \{1,\dots ,h\}} \vert \prod_{i=1}^h W^{\mathrm{SO}(L_i)} \right) = \displaystyle{\sum_{i=1}^h} \mathrm{tr}\left( u_{i,i} \vert W^{\mathrm{SO}(L_i)} \right) \\
		&= \displaystyle{\sum_{i=1}^h} \mathrm{tr}\left( p_i \circ u_{i,i} \vert W \right)= \displaystyle{\sum_{i=1}^h}\left( \dfrac{1}{\vert \mathrm{SO}(L_i) \vert } \cdot \left({\displaystyle{\sum_{\substack{k\in K_i \\ a_{i,k}=i}} \sum_{\gamma \in \mathrm{SO}(L_i)}}\mathrm{tr}\left( \gamma \cdot \gamma_{i,k} \vert W \right) }\right)\right)\\
		&= \displaystyle{\sum_{i=1}^h}\left( \dfrac{1}{\vert \mathrm{SO}(L_i) \vert } \cdot \left({\displaystyle{\sum_{\substack{j\in J_i}} \vert \mathcal{V}_{i,j} \vert \cdot \sum_{\gamma \in \mathrm{SO}(L_i)}}\mathrm{tr}\left( \gamma \cdot g_{i,j} \vert W \right) }\right)\right) .\\
	\end{aligned}
$$
\indent La dernière égalité vient de la constatation suivante. Si on se donne $j\in J_i$, et $L_1,L_2 \in \mathcal{V}_{i,j}$ tels que $L_1=g_1 L_i$ et $L_2=g_2 L_i$, alors on a l'égalité : $\sum_{\gamma \in \mathrm{SO}(L_i)} \mathrm{tr}\left( \gamma \cdot g_{1} \vert W \right) =\sum_{\gamma \in \mathrm{SO}(L_i)} \mathrm{tr}\left( \gamma \cdot g_{2} \vert W \right)$.
\\ \indent D'où finalement le résultat cherché.  
\end{proof}
\indent Un résultat comparable a été obtenu indépendamment par Neil Dummigan, dans l'article \cite{Dum}. De la même manière que dans cet article, on peut regarder ce que devient cette formule lorsque $A$ est trivial (c'est-à-dire lorsque $\mathrm{T}_A$ est l'identité). Dans ce cas, pour tout $i$, on a l'égalité $\mathrm{vois}_A(L_i) = \{L_i \} = \mathcal{V}_i$, et la formule précédente devient :
$$ \begin{aligned}
 \mathrm{tr}(\mathrm{Id} \vert \mathcal{M}_W(\mathrm{SO}_n)) & = \sum_{i=1}^h \left( \frac{1}{\vert \mathrm{SO}(L_i) \vert} \sum_{\gamma \in \mathrm{SO}(L_i)} \mathrm{tr}(\gamma \vert W) \right) \\
 &= \sum_{i=1}^h \mathrm{dim}\left( W^{\mathrm{SO}(L_i)} \right) = \mathrm{dim}\left( \mathcal{M}_W (\mathrm{SO}_n ) \right), \\
 \end{aligned}
 $$
 \indent où les égalités $\mathrm{dim}\left( W^{\mathrm{SO}(L_i)} \right) = \frac{1}{\vert \mathrm{SO}(L_i) \vert} \sum_{\gamma \in \mathrm{SO}(L_i)} \mathrm{tr}(\gamma \vert W)$ sont des résultats classiques.
\\ \\ \indent Lorsque $h=1$, c'est-à-dire lorsque $n=7,8,9$, on a les corollaires plus simples suivants :
\begin{cor} On suppose que $\widetilde{X}_n$ est réduit à un élément. Soient $(W,\rho)$ une représentation de $\mathrm{SO}_n(\R)$ de dimension finie, et $L_0$ un élément de $\mathcal{L}_n$. Soit $(g_j) \in \mathrm{SO}_n(\R)$ famille finie telle que $\mathrm{vois}_A(L_0) = \underset{j}{\coprod} \mathrm{SO}(L_0)\cdot g_j \cdot L_0 = \underset{j}{\coprod} \mathcal{V}_j$. Alors :
$$ \mathrm{tr}\left( \mathrm{T}_A \vert \mathcal{M}_W(\mathrm{SO}_n)\right) = \dfrac{1}{\vert \mathrm{SO}(L_0) \vert} \cdot \left( \displaystyle{\sum_j} \left( \vert \mathcal{V}_j\vert \cdot \displaystyle{\sum_{\gamma \in \mathrm{SO}(L_0)}}\mathrm{tr}\left( \gamma g_j \vert W \right) \right) \right).$$
\end{cor}
\begin{cor} \label{formuletrace} On suppose que $\widetilde{X}_n$ est réduit à un élément. Soient $(W,\rho)$ une représentation de $\mathrm{SO}_n(\R)$, $q$ une puissance d'un nombre premier $p$, $A=\Z/q\Z$, et $L_0$ un élément de $\mathcal{L}_n$. On considère la quadrique projective $C = C_{L_0}(\Z/q\Z)$. On rappelle qu'il existe une bijection entre $\mathrm{vois}_A(L_0)$ et $C$ qui commute aux actions de $\mathrm{SO}(L_0)$ (voir proposition-définition \ref{p-voisin}). Pour tout élément $x\in C$, on associe le $q$-voisin $L_x$ (d'après la méthode expliquée en proposition-définition \ref{p-voisin}), que l'on écrit $L_{x} = g_{x}\cdot L_{0}$, où $g_{x} \in \mathrm{SO}_n(\Q)$.
\\ \indent Le groupe spécial orthogonal $\mathrm{SO}(L_0)$ agit sur $C$. On considère $Y$ un système de représentants des orbites de $C$ pour l'action de $\mathrm{SO}(L_0)$, et pour $y\in Y$ on note $\Omega_{y}$ l'orbite associée. Alors :
$$ \mathrm{tr}\left( \mathrm{T}_q \vert \mathcal{M}_W(\mathrm{SO}_n)\right) = \dfrac{1}{\vert \mathrm{SO}(L_0) \vert } \cdot \left({\displaystyle{\sum_{\substack{y\in Y }} \vert \Omega_{y} \vert \cdot \sum_{\gamma \in \mathrm{SO}(L_i)}}\mathrm{tr}\left( \gamma g_{y} \vert W \right) }\right) .$$
\end{cor}
\indent Les corollaires précédents réduisent le calcul de $\mathrm{tr}(\mathrm{T}_A \vert \mathcal{M}_W(\mathrm{SO}_n))$ aux calculs des quantités suivantes :
\\ \indent - $\mathrm{tr}(g \vert W)$, où $g\in \mathrm{SO}_n(\R)$ et $W$ une représentation irréductible de $\mathrm{SO}_n(\R)$, ce qui fait l'objet du paragraphe \ref{3.2}.
\\ \indent - les orbites de $\mathrm{vois}_A(L_i)$ pour l'action de $\mathrm{SO}(L_i)$ (dont il suffit de connaître un représentant et le cardinal), ce qui fait l'objet des chapitres \ref{4} et \ref{5}.
\subsection{Les poids dominants et la formule des caractères de Weyl.}\label{3.2}
\indent On reprend ici le raisonnement fait dans \cite[ch. 1]{CC} pour exprimer la version dégénérée de la formule des caractères de Weyl. Rappelons quelques notations utilisées dans cette référence.
\\ \indent On se donne $n\geq 1$ et note $G$ le groupe de Lie semi-simple connexe $\mathrm{SO}_n(\R)$. On pose pour simplifier $m=[n/2]$. On se place dans l'espace $\R^n$ muni de son produit scalaire usuel, que l'on décompose en somme orthogonale :
$$\R^n =
\left\{ \begin{aligned}
	&\bigoplus_{i=1}^{m} P_i \text{ si $n$ est pair},\\
	&\bigoplus_{i=1}^{m} P_i\oplus D \text{ si $n$ est impair},\\
	\end{aligned} \right.$$
où les $P_i$ sont des plans deux-à-deux orthogonaux de $\R^n$ fixés, et $D$ est l'unique droite orthogonale à tous les $P_i$ lorsque $n$ est impair.
\\ \indent On définit le tore maximal $T\subset G$ comme :
$$T =
\left\{ \begin{aligned}
	&\{ g\in G \vert (\forall i )\ g(P_i)\subset P_i\text{ et }\mathrm{det}(g\vert_{P_i})=1 \} \text{ si $n$ est pair},\\
	&\{ g\in G \vert (\forall i )\ g(P_i)\subset P_i\text{, }\mathrm{det}(g\vert_{P_i})=1 \text{ et }\mathrm{det}(g\vert_{D})=1\} \text{ si $n$ est impair},\\
	\end{aligned} \right.$$
de telle sorte qu'on a un isomorphisme naturel $T\tilde\rightarrow \prod_i \mathrm{SO}(P_i)$. Fixons une fois pour toute des isomorphismes $\mathrm{SO}(L_i) \simeq \mathbb{S}^1$, et donc l'isomorphisme $T\simeq (\mathbb{S}^1)^m$. On note $(t_1,\dots,t_m)$ les éléments de $T$, pour $t_i \in \mathbb{S}^1$.
\\ \indent On note $\Phi$ le système de racines de $(G,T)$, et $\mathcal{W}$ son groupe de Weyl. On note de plus $\Phi^+ \subset \Phi$ un système de racines positives. Selon la parité de $n$, les ensembles $\Phi$, $\Phi^+$ et $\mathcal{W}$ sont donnés de la manière suivante :
\\ \indent - si $n$ est pair : $\Phi = \{ \pm e_i \pm e_j \}_{1\leq i<j\leq m} \subset \R^m$, $\Phi^+ = \{ e_i \pm e_j \}_{1\leq i<j\leq m} \subset \Phi$, et $\mathcal{W}=\mathcal{S}_m \ltimes (\{\pm 1\}^m)^0$ (où $\mathcal{S}_m$ désigne le groupe des permutations d'un ensemble à $m$ éléments, et $(\{\pm 1\}^m)^0$ le sous-groupe des éléments $(\varepsilon_i) \in \{\pm 1\}^m$ tels que $\prod \varepsilon_i =1$).
\\ \indent - si $n$ est impair : $\Phi = \{ \pm e_i \}_{1\leq i\leq m} \cup \{ \pm e_i \pm e_j \}_{1\leq i<j\leq m} \subset \R^m$, $\Phi^+ = \{ e_i \}_{1\leq i\leq m} \cup \{ e_i \pm e_j \}_{1\leq i<j\leq m} \subset \Phi$, et $\mathcal{W}=\mathcal{S}_m \ltimes \{\pm 1\}^m$.
\\ \indent On note $X=\mathrm{Hom}(T,\mathbb{S}^1)$ le groupe des caractères (qui satisfait $X\simeq \Z^m$). On fixe un produit scalaire $(\ ,\ )$ sur $X\otimes \R$ invariant par $\mathcal{W}$. Un poids dominant de $G$ est un élément $\lambda \in X$ tel que $(\lambda,\alpha) \geq 0$ pour tout élément $\alpha \in \Phi^+$. La théorie de Cartan-Weyl définit une bijection canonique $\lambda \mapsto V_\lambda$ entre les poids dominants de $G$ et les représentations irréductibles de $G$ à isomorphisme près. Le poids dominant $\lambda$ est appelé le plus haut poids de $V_\lambda$. Si $\lambda\in X$ et $i\in T$, on pose $t^\lambda = \lambda (t)$
\\ \indent Enfin, pour $t\in T$, on note $M=C_G(t)^0$ la composante neutre du centralisateur de $t$ dans $G$ (en particulier, $T$ est un tore maximal dans $M$, et $t\in M$). On associe à $M$ les sous-ensembles de $\Phi^+$ et de $\mathcal{W}$ suivants :
$$\Phi_M^+ = \{ \alpha \in \Phi^+ \vert t^\alpha =1 \}$$
$$\mathcal{W}^M = \{w\in \mathcal{W} \vert w^{-1} (\Phi_M^+) \subset \Phi^+ \}$$
\indent On note $\rho$ et $\rho_M$ la demi-somme respectivement des éléments de $\Phi^+$ et de $\Phi_M^+$. Pour $w\in \mathcal{W}^M$, on pose $\lambda_w = w(\lambda +\rho )-\rho_M$. On définit enfin pour $v\in X\otimes \R$ :
$$P_M(v) = \prod_{\alpha \in \Phi_M^+} \dfrac{(\alpha,v+\rho_M)}{(\alpha,\rho_M)}.$$
\indent On a la proposition suivante d'après \cite[ch. 1, proposition 1.9]{CC} :
\begin{prop}[Version dégénérée de la formule des caractères de Weyl] Soient $\lambda \in X$ un poids dominant, $t\in T$ et $M=C_G(t)^0$. Alors le caractère de $t$ sur la représentation $V_\lambda$ de plus haut poids $\lambda$ est donné par :
$$\chi_{V_\lambda} (t) = \dfrac{\sum_{w\in \mathcal{W}^M} \varepsilon(w) \cdot t^{w(\lambda +\rho)-\rho} \cdot P_M(w(\lambda + \rho )-\rho_M )}{\prod_{\alpha \in \Phi^+ \setminus \Phi_M^+}(1-t^{-\alpha})}$$
où $\varepsilon:\mathcal{W}\rightarrow \{\pm 1\}$ désigne la signature sur le groupe $\mathcal{W}$.
\end{prop}
\indent Pour terminer ce paragraphe, expliquons comment utiliser concrètement cette formule pour l'application au calcul de la trace de $\mathrm{T}_A$. On rappelle que dans cette formule interviennent des éléments de la forme $\gamma g_y \in \mathrm{SO}_n(\R)$, et on doit évaluer la quantité $\mathrm{tr}(\gamma g_y \vert W)$.
\\ \indent Au chapitre \ref{4}, on explique comment déterminer explicitement les $\gamma g_y$. La question est donc la suivante : si on se donne $g\in \mathrm{SO}_n(\R)$, comment trouver $t\in T$ conjugué à $g$ dans $\mathrm{SO}_n(\R)$.
\\ \indent On calcule le déterminant $\mathrm{det}(X\mathrm{Id}-g)$, qui est de la forme :
$$(X-1)^{n-2m} \cdot \prod_{i=1}^m (X-t_i)(X-\overline{t_i}),$$
et on pose $t=(t_1,\dots ,t_m) \in T$.
\\ \indent Si $n$ est impair, $g$ est conjugué à $t$, et le problème est résolu : il suffit de déterminer des $t_i$ satisfaisant l'égalité précédente. C'est encore vrai lorsque $n$ est pair et que l'un des $t_i$ vaut $\pm 1$.
\\ \indent En revanche, dans le cas général, on a seulement que $g$ est conjugué dans $\mathrm{SO}_n(\R)$ à $t$ ou à $\overline{t}=(t_1,\dots ,t_{m-1},\overline{t_m})$. On a la relation : $\mathrm{tr}(\overline{t}\vert W)=\mathrm{tr}(t\vert W^c)$, où la représentation $(\rho^c, W^c)$ est définie par $(g\mapsto \rho (cgc^{-1}),W)$, pour $c\in \mathrm{O}_n(\R) \setminus \mathrm{SO}_n(\R)$. En terme de plus hauts poids, on a : $V_\lambda^c = V_{\overline{\lambda}}$, où $\overline{\lambda}=(\lambda_1,\dots ,-\lambda_m)$ si $\lambda=(\lambda_1,\dots ,\lambda_m)$.
\\ \indent On peut donc calculer pour tout $g\in \mathrm{SO}_n(\R)$ la quantité : $\mathrm{tr}(g\vert V_\lambda)+\mathrm{tr}(g\vert V_\lambda^c)$, et donc $\mathrm{tr}(\mathrm{T}_A\vert \mathcal{M}_W(\mathrm{SO}_n))+\mathrm{tr}(\mathrm{T}_A\vert \mathcal{M}_{W^c}(\mathrm{SO}_n))$.
\\ \indent Le lemme suivant montre qu'on a en fait la relation : $\mathrm{tr}(\mathrm{T}_A\vert \mathcal{M}_W(\mathrm{SO}_n))=\mathrm{tr}(\mathrm{T}_A\vert \mathcal{M}_{W^c}(\mathrm{SO}_n))$, et donc $\mathrm{tr}(\mathrm{T}_A\vert \mathcal{M}_W(\mathrm{SO}_n))=\frac{\mathrm{tr}(\mathrm{T}_A\vert \mathcal{M}_W(\mathrm{SO}_n))+\mathrm{tr}(\mathrm{T}_A\vert \mathcal{M}_{W^c}(\mathrm{SO}_n))}{2}$.
\begin{lem} Soient $s\in \mathrm{O}_n(\R)\setminus \mathrm{SO}_n(\R)$, et $W$ une représentation de $\mathrm{SO}_n(\R)$. L'application $f\mapsto (L\mapsto f(s(L)))$ induit une bijection $\C$-linéaire $\mathcal{M}_W (\mathrm{SO}_n) \rightarrow \mathcal{M}_{W^s}(\mathrm{SO}_n)$. Cette bijection commute à l'opérateur de Hecke $\mathrm{T}_A$ pour tout $A$.
\end{lem}
\begin{proof}
Pour simplifier, on note $\varphi$ l'application définie par $\varphi (L) = f(s(L))$. Pour tout $g\in \mathrm{SO}_n(\R)$, on a :
$$\varphi (gL) = f(sgL)=f(sgs^{-1}\cdot sL) = \rho^s (g) \varphi(L).$$
L'application introduite dans le lemme est donc bien une application $\C$-linéaire de $\mathcal{M}_W (\mathrm{SO}_n)$ dans $\mathcal{M}_{W^s}(\mathrm{SO}_n)$. Elle est évidemment bijective, de réciproque $f\mapsto (L\mapsto f(s^{-1} L)))$.
\\ \indent On voit dans la définition des $A$-voisins que : $s\left( \mathrm{vois}_A(L) \right) = \mathrm{vois}_A \left( s(L) \right)$, ce qui conclut que l'application précédente commute bien à l'opérateur $\mathrm{T}_A$. 
\end{proof}
\section{Algorithmes de calculs et résultats pour $p$ impair.}\label{4}
\indent Soient $L$ un réseau pair de $\R^n$, $R=R(L)$ son ensemble de racines, et $W$ son groupe de Weyl. Pour les réseaux étudiés dans cette partie (à savoir $\mathrm{E}_7$, $\mathrm{E}_8$ et $\mathrm{E}_8\oplus \mathrm{A}_1$), $R$ est un système de racines qui engendre $\Z$-linéairement le réseau $L$, et le groupe $W$ est égal au groupe $\mathrm{O}(L)$ (d'après la proposition \ref{Weyl et SO}). On note $W^+=\mathrm{SO}(L)$ le sous-groupe de $W$ des éléments de déterminant $1$.
\\ \indent On désigne par $q$ la puissance d'un nombre premier $p$ (dont la parité sera précisée lorsque cela sera nécessaire).
\subsection{Présentation des algorithmes de calculs.}\label{4.1}
\indent Les algorithmes que l'on construit ci-dessous ont trois objectifs :
\\ \begin{tabular}{rp{14cm}}
$(i)$ & Donner une description facile à manipuler du groupe de Weyl $W$. \\
$(ii)$ & Regrouper les orbites de $C_L(\Z/q\Z)$ pour l'action du groupe $W^+$, et pour chaque orbite donner un représentant et le cardinal de l'orbite.\\
$(iii)$ & Pour chaque représentant d'une orbite trouvée à l'étape $(ii)$, expliciter le $q$-voisin associé $L'$ et une transformation $g\in \mathrm{SO}_n(\R)$ telle que : $L'=g ( L)$.\\
\end{tabular}
\indent Les parties $(i)$ et $(iii)$ sont utilisées aussi bien pour le cas où $p=2$ que pour le cas où $p$ est impair. La partie $(ii)$ n'est utilisée que pour le cas où $p$ est impair (et on détaille au chapitre \ref{5} le cas où $p=2$).
\subsection{La création du groupe de Weyl.}\label{4.2}
\subsubsection{Principe utilisé.}
\indent On souhaite pouvoir utiliser dans nos algorithmes le groupe de Weyl de $R$. On cherche un moyen de parcourir tout le groupe $W$ rapidement et sans avoir à allouer une mémoire trop conséquente.
\\ \indent On suppose donné l'ensemble $R$ des racines du réseau $L$, ainsi qu'un système de racines simples $(v_1,\dots ,v_n)$. On associe à ces racines simples les réflexions $(s_1,\dots ,s_n)$, la chambre de Weyl $C$ dont les murs sont les hyperplans laissés stables par un des $s_i$, $\rho$ un élément de $C$, et $l:W\rightarrow \N$ la longueur associée à ces racines simples. Le groupe $W$ est engendré par les réflexions $s_i$, et on suppose connu le cardinal de $W$.
\\ \indent On suppose aussi que l'on possède un sous-groupe $W'$ de $W$, dont les éléments sont facile à expliciter, et dont on connaît le cardinal. Pour tout élément $w\in W$, on définit son image dans le ``quotient" $W/W'$ comme l'ensemble : $\overline{w} = \{ \gamma \circ w \vert \gamma \in W' \}$. On suppose qu'il existe une fonction $\Phi_{n}$ suffisamment simple au sens algorithmique, définie sur $W$ et à valeurs dans un ensemble que l'on précisera, telle que :
$$(\forall w_1,w_2 \in W) \ \overline{w_1}=\overline{w_2} \Leftrightarrow \Phi_{n}(w_1) = \Phi_{n} (w_2).$$
\indent L'algorithme expliqué ci-dessous nous donne un ensemble $H_{n}$ de représentants du quotient $W/W'$ (le cardinal de $H_{n}$ étant connu, avec : $\vert H_{n} \vert = \vert W \vert /\vert W' \vert $).
\\ \indent L'objectif final est de pouvoir écrire les éléments de $W$ sous la forme $w\circ h$, pour $w\in W'$ et $h\in H_n$.
\\ \\ \textbf{Description de l'algorithme :}
\\ \indent Pour créer notre ensemble $H_{n}$, on va construire une suite d'ensembles $(S_i)$ avec $\emptyset = S_0 \subset S_1 \subset \dots \subset S_r =H_n$, avec $\vert S_i \vert =i$, et donc $r=\vert H_n\vert = \vert W\vert / \vert W' \vert$. La construction des $S_i$ se fait récursivement. On suppose que l'on a déjà construit l'ensemble $S_i$ (avec $0\leq i\leq r-1$), et on parcourt le groupe $W$ en commençant par les éléments de plus petites longueurs. Pour chaque élément $w$ parcouru par notre algorithme, on regarde son image $\overline{w} \in W/W'$ :
\\ \indent - si $\overline{w} \in \{ \overline{h} \vert h\in S_i \}$ (c'est-à-dire si $\Phi_{n}(w) \in  \{ \Phi_{n}(h) \vert h\in S \}$) : on passe à l'élément suivant dans $W$.
\\ \indent - si $\overline{w} \notin \{ \overline{h} \vert h\in S \}$ (c'est-à-dire si $\Phi_{n}(w) \notin  \{ \Phi_{n}(h) \vert h\in S \}$) : on pose $S_{i+1} = S_i\cup \{ w\}$ .
\\ \indent On arrête notre algorithme lorsque l'on a construit l'ensemble $S_r$, et on pose $H_n = S_r$. L'ensemble $H_n$ ainsi créé correspond à l'ensemble des représentants de $W/W'$ de plus petites longueurs.
\\ \\ \indent Notre seul problème est donc de parcourir $W$. Pour cela, on pose pour tout $j\in \N$ : $W_j=\{ w\in W \vert l(w) =j\}$. On connaît déjà les ensembles $W_0$ et $W_1$ (qui correspondent respectivement à $\{ \mathrm{id} \}$ et à $\{ s_i \vert i=1,\dots ,n\}$). On détermine récursivement tous les ensembles $W_i$ grâce à l'application :
$$
	\begin{array}{rcl}
		W_j\times W_1 & \rightarrow & W_{j-1}\cup W_{j+1} \\
		(w,s_i) & \mapsto & s_i \circ w \\
	\end{array}
$$
dont l'image contient $W_{j+1}$. Pour savoir si l'image d'un élément est bien dans $W_{j+1}$, on utilise le lemme suivant (conséquence immédiate du $(ii)$ de la proposition \ref{longueur}) :
\begin{lem} Soient $j\in \N$, $w\in W_j$ et $i\in \{1,\dots ,n\}$. Alors on a l'équivalence :
$$l(s_i\circ w) =j+1 \Leftrightarrow \left( v_i \cdot w(\rho) \right) \cdot \left( v_i \cdot \rho \right) \geq 0.$$
\end{lem}
\indent Un élément de $W_{j+1}$ peut s'écrire de différentes manières sous la forme $s\circ w$ (avec $s\in W_1$ et $w\in W_j$). Afin de ne pas parcourir plusieurs fois le même élément de $W_{j+1}$, on utilise le corollaire \ref{transitif} :
$$(\forall w_1,w_2\in W_j)(\forall s_{i_1},s_{i_2}\in W_1)\  s_{i_1}\circ w_1=s_{i_2}\circ w_2 \Leftrightarrow s_{i_1}\circ w_1 (\rho) = s_{i_2}\circ w_2 (\rho).$$
\indent Concrètement, si on écrit les éléments $w$ de $W_j$ sous la forme $w=s_{a_j} \circ \dots \circ s_{a_1}$ pour $a=(a_1,\dots,a_j)\in \{1,\dots ,n\}^j$, alors on gardera seulement l'écriture où $a$ est le plus petit pour l'ordre lexicographique.
\\ \indent Afin de ne pas surcharger la mémoire allouée, et de limiter les calculs, on représentera tout élément $w\in W_j$ sous la forme d'un couple $(a,v)$, où $a$ est de forme énoncée ci-dessus, et $v=w(\rho)\in \R^n$.
\\ \indent On détaille cette méthode dans le cas de $\mathrm{E}_7$ et de $\mathrm{E}_8$ pour aider à mieux comprendre.
\subsubsection{Le cas de $\mathrm{E}_7$.}
\indent Dans la description faite de $\mathrm{E}_7$ au paragraphe \ref{2.1}, on constate que l'on a l'inclusion $\mathrm{A}_7 \subset \mathrm{E}_7$. Le groupe des permutations $\mathcal{S}_8$ agit sur $\mathrm{A}_7$ et sur $\mathrm{E}_7$ par permutation des coordonnées. On a en fait $\mathrm{O}(\mathrm{A}_7) = \{ \pm \mathrm{id}\} \times \mathcal{S}_8$.
\\ \indent On note $W'=\{ \pm \mathrm{id} \} \times \mathcal{S}_8 = \mathrm{O}(\mathrm{A}_7)\subset W$. L'ensemble $W/W'$ possède $\vert W \vert / (2\cdot 8!) = 36$ éléments (comme $\vert W\vert = 2^{10}\cdot 3^4\cdot  5\cdot 7$, d'après \cite[Planche VI]{Bo} par exemple). Le système de racines simples que l'on choisit est donné par les $v_i$ avec :
$$ \left\{\begin{aligned}
v_1 =1/2 \cdot(-1, 1, 1, 1, -1, -1, -1, 1)\\
v_2 =1/2 \cdot(-1, 1, 1, -1, -1, 1, 1, -1)\\
v_3 =1/2 \cdot(1, -1, -1, 1, -1, 1, 1, -1)\\
v_4 =1/2 \cdot(1, -1, 1, -1, 1, -1, -1, 1)\\
v_5 =1/2 \cdot(-1, 1, -1, 1, 1, -1, 1, -1)\\
v_6 =1/2 \cdot(1, 1, -1, -1, -1, 1, -1, 1)\\
v_7 =1/2 \cdot(-1, -1, 1, 1, 1, 1, -1, -1)\\
\end{aligned} \right.$$
\indent L'élément $\rho$ associé à ce système de racines simples est  :
$$\rho = (29,21,13,5,-3,-11,-19,-35) .$$
\indent La fonction $\Phi_7$ que l'on va utiliser est donnée par la proposition-définition évidente suivante :
\begin{propdef} Soit $E$ l'ensemble des multi ensembles de $\R$. On définit la fonction $\phi_7 : \R^8 \rightarrow E$ par :
$$\phi_7\left( (v_1,\dots,v_8)\right)=\{ \{ v_1,\dots ,v_8 \}\} .$$
\indent Si on se donne  $v_1,v_2\in \R^8$, on a l'équivalence :
$$\{ w(v_1) \vert w\in W' \} = \{ w(v_2) \vert w\in W' \} \Leftrightarrow \phi_7(\pm v_1) = \phi_7(\pm v_2) .$$
\indent Ainsi, la fonction $\Phi_7 : W\rightarrow \{ \{e_1,e_2\} \vert e_1,e_2\in E\}$ donnée par :
$$\Phi_7 (w) = \{ \phi_7 (w(\rho)) , \phi_7 (-w(\rho)) \}$$
vérifie l'équivalence :
$$(\forall w_1,w_2\in W) \ \overline{w_1}=\overline{w_2} \Leftrightarrow \Phi_7 (w_1) =  \Phi_7 (w_2).$$ 
\end{propdef}
\indent Notons $s_1,\dots,s_7$ les symétries orthogonales associées à $v_1,\dots,v_7$. On note $[i_1,\dots ,i_j]$ l'élément $s_{i_j}\circ \dots \circ s_{i_1}$ (avec la notation $[\ ]$ pour $\mathrm{id}$). Alors notre algorithme donne pour $H_7$ l'ensemble :
$$\begin{aligned}
&\{ [\ ],[1],[2],[3],[4],[5],[6],[7],[1,2],[1,4],[1,5],[1,6],[1,7],[2,3],[2,6],[3,5],[3,6],[3,7],[4,6],\\
&[4,7],[1,2,6],[1,4,6],[1,4,7],[2,3,6],[3,1,4],[4,2,3],[4,3,5],[5,4,6],[1,5,4,6],[3,1,4,6],\\
&[3,1,4,7],[4,2,3,6],[3,1,5,4,6],[5,4,2,3,6],[4,3,1,5,4,6],[2,4,3,1,5,4,6] \} \\
	\end{aligned}
$$
\\ \indent Dans la suite, lorsque l'on utilisera le groupe $W(\mathrm{E}_7)$, on écrira ses éléments grâce à la bijection :
$$
	\begin{array}{rcl}
		\mathcal{S}_8 \times \{\pm \mathrm{id}\} \times H_7 & \overset{\sim}{\longrightarrow} & W \\
		(\sigma,\varepsilon ,h) & \mapsto & \sigma\circ \varepsilon \circ h.\\
	\end{array}
$$

\subsubsection{Le cas de $\mathrm{E}_8$.}
\indent Dans la description de $\mathrm{E}_8$ faite au paragraphe \ref{2.1}, on constate que l'on a l'inclusion $\mathrm{D}_8 \subset \mathrm{E}_8$. Le groupe $\mathcal{S}_8$ agit sur $\mathrm{D}_8$ et sur $\mathrm{E}_8$ par permutation des coordonnées. On a : $W(\mathrm{D}_8)=\mathcal{S}_8 \ltimes \left( \{ \pm 1\}^8 \right) ^0$, où $\left( \{ \pm 1\}^8\right) ^0$ désigne le groupe de matrices diagonales de $\mathrm{M}_8(\Z)$ de déterminant $1$.
\\ \indent On note $W'=\mathcal{S}_8 \ltimes \left( \{ \pm 1\}^8\right) ^0=W(\mathrm{D}_8)\subset W$. L'ensemble $W/W'$ possède $\vert W \vert /(2^7\cdot 8!) = 135$ éléments (comme $\vert W \vert = 2^{14}\cdot 3^5\cdot 5^2\cdot 7$, d'après \cite[Planche VII]{Bo} par exemple). On pose enfin $W''=\mathcal{S}_8 \ltimes \{ \pm 1\}^8 = \mathrm{O}(\mathrm{D}_8)\supset W'$. Le système de racines simples que l'on choisit est donné par les $v_i$ avec :
\\ \indent - les $v_i$ pour $i=1,\dots ,7$ sont les mêmes que pour $\mathrm{E}_7$,
\\ \indent - le vecteur $v_8$ est : $(0, 0, 0, 0, 0, 0, 1, 1)$.
\\ \indent L'élément $\rho$ associé à ce système de racines simples est :
$$\rho = (29,25,21,17,13,9,5,-3).$$
\indent On prendra garde au fait que : $W'' \not\subset W$. Si on voit $W$ comme un sous-ensemble de $\mathrm{O}(V)$, alors les groupes $W'$ et $W''$ ont une action bien définie sur $W$ par composition à gauche. Le lemme suivant donne la relation entre les orbites pour ces deux actions :
\begin{lem}\label{lemmed} Soient $w_1, w_2 \in W$. On a l'équivalence :
$$W'\circ w_1 = W'\circ w_2 \Leftrightarrow W'' \circ w_1=  W''.\circ w_2 .$$
\end{lem}
\begin{proof}
On procède directement par équivalences.
\\ \indent Soit $d\in \mathrm{O}(V)$ dont la matrice est la matrice diagonale ayant $-1$ comme premier terme, et $1$ partout ailleurs. Alors on a : $W'' = W'\sqcup (W'\circ d)$. Ainsi, on déduit :
$$	W''\circ w_1 = W'' \circ w_2 \Leftrightarrow  w_1\circ w_2^{-1}\in W'' \Leftrightarrow \left\{
		\begin{aligned}
			& w_1\circ w_2^{-1}\in W' \\
			& \text{ou }w_1\circ w_2^{-1}\in  W' \circ d\\
		\end{aligned}
		\right.		.$$
\indent Le cas $w_2^{-1}\cdot w_1\in  W'\circ d$ est impossible, car $ (W' \circ d) \cap W =\emptyset$ et on a l'équivalence cherchée. 
\end{proof}
\indent La fonction $\Phi_8$ que l'on va utiliser est donnée par la proposition-définition évidente suivante :
\begin{propdef} Soit $E$ l'ensemble des multi ensembles de $\R$. On définit la fonction $\phi_8 : \R^8 \rightarrow E$ par :
$$\phi_8\left( (v_1,\dots,v_8)\right)=\{ \{ v_1^2,\dots ,v_8^2 \}\} .$$
\indent Si on se donne  $v_1,v_2\in \R^8$, on a l'équivalence :
$$\{ \gamma(v_1) \vert \gamma\in W'' \} = \{ \gamma(v_2) \vert \gamma\in W'' \} \Leftrightarrow \phi_8(v_1) = \phi_8(v_2).$$
\indent Ainsi, la fonction $\Phi_8 : W\rightarrow  E$ donnée par :
$$\Phi_8 (w) = \phi_8 (w(\rho))$$
vérifie l'équivalence :
$$(\forall w_1,w_2\in W) \ \overline{w_1}=\overline{w_2} \Leftrightarrow \Phi_8 (w_1) =  \Phi_8 (w_2).$$ 
\end{propdef}
\indent Avec les mêmes notations que pour $\mathrm{E}_7$, notre algorithme donne l'ensemble $H_8$ suivant :
\\ \\ \begin{tabular}{l}
\small $ \{  [\ ], [1], [2], [3], [4], [5], [6], [7], [1, 2], [1, 4], [1, 5], [1, 6], [1, 7], [2, 3], [2, 5], [2, 6], [2, 7], [3, 5], [3, 6],[3, 7], [4, 6],$\\
\small $ [4, 7], [5, 7], [8, 7], [1, 2, 5], [1, 2, 6], [1, 2, 7], [1, 4, 6], [1, 4, 7], [1, 5, 7], [1, 8, 7], [2, 3, 5], [2, 3, 6], [2, 3, 7], [2, 5, 7],$\\
\small $ [2, 8, 7], [3, 1, 4], [3, 5, 7], [3, 8, 7], [4, 2, 3], [4, 2, 5], [4, 3, 5], [4, 8, 7], [5, 4, 6],[5, 8, 7], [6, 5, 7], [1, 2, 5, 7],$\\
\small $ [1, 2, 8, 7], [1, 4, 2, 5], [1, 4, 8, 7], [1, 5, 4, 6], [1, 5, 8, 7], [1, 6, 5, 7],[2, 3, 5, 7],[2, 3, 8, 7], [2, 5, 8, 7], [2, 6, 5, 7],$\\
\small $ [3, 1, 4, 6], [3, 1, 4, 7], [3, 5, 8, 7], [3, 6, 5, 7], [4, 2, 3, 6], [4, 2, 3, 7],[4, 2, 5, 7],[4, 3, 5, 7], [6, 5, 8, 7], [1, 2, 5, 8, 7],$\\
\small $ [1, 2, 6, 5, 7], [1, 4, 2, 5, 7], [1, 6, 5, 8, 7], [2, 3, 5, 8, 7], [2, 3, 6, 5, 7],[2, 6, 5, 8, 7], [3, 1, 4, 2, 5], [3, 1, 4, 8, 7],$\\
\small $ [3, 1, 5, 4, 6], [3, 6, 5, 8, 7], [4, 2, 3, 8, 7], [4, 2, 5, 8, 7], [4, 2, 6, 5, 7],[4, 3, 5, 8, 7], [4, 3, 6, 5, 7], [5, 4, 2, 3, 6],$\\
\small $ [7, 6, 5, 8, 7], [1, 2, 6, 5, 8, 7], [1, 4, 2, 5, 8, 7], [1, 4, 2, 6, 5, 7],[1, 7, 6, 5, 8, 7], [2, 3, 6, 5, 8, 7], [2, 7, 6, 5, 8, 7],$\\
\small $ [3, 1, 4, 2, 5, 7], [3, 7, 6, 5, 8, 7], [4, 2, 6, 5, 8, 7], [4, 3, 1, 5, 4, 6],[4, 3, 6, 5, 8, 7], [5, 4, 2, 6, 5, 7], [5, 4, 3, 6, 5, 7],$\\
\small $ [1, 2, 7, 6, 5, 8, 7], [1, 4, 2, 6, 5, 8, 7], [1, 5, 4, 2, 6, 5, 7],[2, 3, 7, 6, 5, 8, 7], [2, 4, 3, 1, 5, 4, 6], [3, 1, 4, 2, 5, 8, 7],$\\
\small $ [3, 1, 4, 2, 6, 5, 7], [4, 2, 7, 6, 5, 8, 7], [4, 3, 7, 6, 5, 8, 7],[5, 4, 2, 6, 5, 8, 7], [5, 4, 3, 6, 5, 8, 7], [1, 4, 2, 7, 6, 5, 8, 7],$\\
\small $ [1, 5, 4, 2, 6, 5, 8, 7], [3, 1, 4, 2, 6, 5, 8, 7],[3, 1, 5, 4, 2, 6, 5, 7], [5, 4, 2, 7, 6, 5, 8, 7], [5, 4, 3, 7, 6, 5, 8, 7],$\\
\small $ [1, 5, 4, 2, 7, 6, 5, 8, 7], [3, 1, 4, 2, 7, 6, 5, 8, 7],[3, 1, 5, 4, 2, 6, 5, 8, 7], [4, 3, 1, 5, 4, 2, 6, 5, 7], [6, 5, 4, 2, 7, 6, 5, 8, 7],$\\
\small $ [6, 5, 4, 3, 7, 6, 5, 8, 7],[1, 6, 5, 4, 2, 7, 6, 5, 8, 7], [2, 4, 3, 1, 5, 4, 2, 6, 5, 7], [3, 1, 5, 4, 2, 7, 6, 5, 8, 7],$\\
\small $ [4, 3, 1, 5, 4, 2, 6, 5, 8, 7],[2, 4, 3, 1, 5, 4, 2, 6, 5, 8, 7], [3, 1, 6, 5, 4, 2, 7, 6, 5, 8, 7], [4, 3, 1, 5, 4, 2, 7, 6, 5, 8, 7],$\\
\small $ [2, 4, 3, 1, 5, 4, 2, 7, 6, 5, 8, 7],[4, 3, 1, 6, 5, 4, 2, 7, 6, 5, 8, 7], [2, 4, 3, 1, 6, 5, 4, 2, 7, 6, 5, 8, 7],$\\
\small $ [5, 4, 3, 1, 6, 5, 4, 2, 7, 6, 5, 8, 7],[2, 5, 4, 3, 1, 6, 5, 4, 2, 7, 6, 5, 8, 7], [4, 2, 5, 4, 3, 1, 6, 5, 4, 2, 7, 6, 5, 8, 7],$\\
\small $[3, 4, 2, 5, 4, 3, 1, 6, 5, 4, 2, 7, 6, 5, 8, 7],[1, 3, 4, 2, 5, 4, 3, 1, 6, 5, 4, 2, 7, 6, 5, 8, 7] \} .$\\
\end{tabular}
\\ \\ \indent Dans la suite, lorsque l'on utilisera le groupe $W(\mathrm{E}_8)$, on écrira ses éléments grâce à la bijection :
$$
	\begin{array}{rcl}
		\mathcal{S}_8 \times \left(\{\pm 1\}^8\right)^0 \times H_8 & \overset{\sim}{\longrightarrow} & W \\
		(\sigma,\varepsilon,h) & \mapsto & \sigma\circ\varepsilon\circ h\\
	\end{array} .
$$

\subsection[La détermination des orbites de $C_L(\Z/q\Z)$ par l'action du groupe de Weyl $W$ \\ et par $W^+$ pour $q$ une puissance de $p$ premier impair.]{La détermination des orbites de $C_L(\Z/q\Z)$ par l'action du groupe de Weyl $W$ et par $W^+$ pour $q$ une puissance de $p$ premier impair.}\label{4.3}
\subsubsection{Méthode utilisée.}
\indent On reprend les notations du corollaire \ref{formuletrace}. Soient $n=7,8$ ou $9$, $L\in \mathcal{L}_n$, $q$ une puissance d'un nombre premier $p$ impair, $C=C_{L}(\Z/q\Z)$, et $R$ l'ensemble des racines de $L$ (qui est un système de racines de $\R^n$). On note $W$ le groupe de Weyl de $R$, et $W'$ un sous-groupe bien choisi de $W$ (que l'on précisera à chaque cas).
\\ \indent Notre algorithme effectue les étapes suivantes :
\\ \indent \textbf{Première étape :} on cherche les orbites de $C$ par l'action de $W'$. On donne pour chaque orbite un représentant ainsi que le cardinal de l'orbite.
\\ \indent \textbf{Deuxième étape :} on rappelle que l'on connaît le sous-ensemble $H_n\subset W$ des représentants de $W/W'$ de plus petites longueurs. On cherche les orbites de $C$ pour l'action de $W$ en faisant agir $H_n$ sur les représentants des orbites de $C$ pour l'action de $W'$ choisis à la première étape. On donne pour chaque orbite un représentant et son cardinal.
\\ \indent \textbf{Troisième étape :} on cherche les orbites de $C$ pour l'action de $W^+$. Ces orbites sont obtenues grâce aux orbites trouvées à la deuxième étape. Par exemple, lorsque $n=7$ ou $9$, les orbites de $C$ pour les actions de $W$ et de $W^+$ sont égales.
\\ \indent On présente en détail ces étapes dans les cas où $n=7,8,9$. Dans ces cas, le réseau $L$ désignera respectivement les réseaux $\mathrm{E}_7$, $\mathrm{E}_8$ et $\mathrm{E}_8\oplus \mathrm{A}_1$.
\\ \indent Dans la suite, si on se donne $\Gamma$ un groupe agissant sur un ensemble $X$, on note pour $x\in X$ l'orbite $\mathcal{O}_\Gamma (x)$ de l'élément $x$ sous l'action de $\Gamma$.
\\ \indent On fixe enfin $p$ un nombre premier impair, et on note par $q$ une puissance de $p$.
\subsubsection{Le cas de $\mathrm{E}_7$.}$  $
\\ \textbf{Première étape :}
\\ \indent On reprend la définition de $\mathrm{E}_7$ donnée au paragraphe \ref{2.1}.
\\ \indent Le nombre $q$ est impair, donc premier à $\vert \mathrm{E}_7 /\mathrm{A}_7 \vert =2$. L'inclusion naturelle $\mathrm{A}_7 \rightarrow \mathrm{E}_7$ induit une bijection : $\mathrm{A}_7 /q \mathrm{A}_7 \overset{\sim}{\rightarrow} \mathrm{E}_7 /q\mathrm{E}_7$.
\\ \indent D'autre part, la surjection $\Z^8 \overset{\sum x_i}{\longrightarrow} \Z$ est scindée de noyau $\mathrm{A}_7$, et on a l'isomorphisme naturel :
$$\mathrm{A}_7 /q\mathrm{A}_7 \simeq \{ (x_1,\dots ,x_8)\in (\Z/q)^8 \vert \sum x_i =0 \}.$$
\indent Au final, la quadrique $C_{\mathrm{E}_7} (\Z/q\Z)$ s'identifie à l'ensemble des $\Z/q$-droites $\Z/q \cdot (x_1,\dots ,x_8)$ de $(\Z/q)^8$ vérifiant : $\sum x_i = \sum x_i^2 =0$.
\\ \indent Par construction, cette identification commute aux actions des groupes orthogonaux considérés, celle de $W'$ étant particulièrement transparente. On a notamment les lemmes suivants, dont la vérification est évidente :
\begin{lem} Soit $x\in C_{\mathrm{E}_7}(\Z/q\Z)$. Alors l'orbite de $x$ par l'action de $W'$ possède un élément $x'=\mathrm{vect}_{\Z/q\Z} (v)$ pour $v\in \mathrm{E}_7$, où $v=(v_1,\dots ,v_8)$ vérifie :
$$\begin{array}{rl}
	(i) &v_i\in \{0,\dots ,q-1\} \\
	(ii) & v_1\leq \dots \leq v_8\\
	(iii) & \mathrm{inf}\{v_i \, \vert \, v_i>0 \} = 1 \\
\end{array}$$
\end{lem}
\begin{lem} Soient $x,x'\in C_{\mathrm{E}_7}(\Z/q\Z)$, et $v,v'\in (\Z/q\Z)^8$ des générateurs de $x$ et $x'$. Alors $x$ et $x'$ sont dans la même orbite de $C_{\mathrm{E}_7}(\Z/q\Z)$ pour l'action de $W'$ si, et seulement si :
$$(\exists i\in \Z/q\Z)\ \phi_7 (\overline{i\cdot v}) = \phi_7(\overline{v'})$$ 
où $\phi_7$ désigne la fonction introduite au paragraphe précédent, et où on a noté par $\overline{v}$ l'unique élément de $\{0,\dots ,q-1 \}^8$ ayant même image que $v$ par le passage au quotient : $\Z^8 \rightarrow (\Z/q\Z)^8$.
\end{lem}
\indent On aura ainsi pour chaque orbite cherchée un représentant. Le cardinal de l'orbite est donné par le lemme suivant dont la vérification est laissée au lecteur :
\begin{lem} Soit $x\in C_{\mathrm{E}_7}(\Z/q\Z)$ une droite isotrope, et $v=(v_1,\dots ,v_8)$ un générateur de $x$. Alors le cardinal de l'orbite de $x$ par l'action de $W'$ est donné par :
\begin{center}
\scalebox{1}{$\mathrm{card}\left( \mathcal{O}_{W'}(x)\right) =\dfrac{8!}{\vert \{ i\in \Z/q\Z \vert \phi_7 (\overline{i\cdot v}) = \phi_7(\overline{v}) \} \vert \cdot \displaystyle \prod_{i\in \Z/q\Z} \left(\vert \{j\in \{1,\dots 8\} \vert v_j = i \} \vert !\right)}.$}
\end{center}
\end{lem}
\textbf{Deuxième étape :}
\\ \indent On possède ainsi pour chaque orbite de $C_{\mathrm{E}_7}(\Z/q\Z)$ sous l'action de $W'$ un représentant et le cardinal de l'orbite. Il suffit d'utiliser l'ensemble $H_7$ construit au paragraphe \ref{4.2} pour savoir si des orbites différentes pour l'action de $W'$ sont dans la même orbite pour l'action de $W$. Le lemme suivant se déduit facilement de la définition de l'ensemble $H_7$ :
\begin{lem} Soient $x,x'\in C_{\mathrm{E}_7}(\Z/q\Z)$, et $v,v'\in (\Z/q\Z)^8$ des générateurs de $x$ et $x'$. Alors $x$ et $x'$ sont sur la même orbite de $C_{\mathrm{E}_7}(\Z/q\Z)$ pour l'action de $W$ si, et seulement si :
$$(\exists i\in \Z/q\Z) (\exists h \in H_7)\  \phi_7 (\overline{i\cdot v}) = \phi_7(\overline{h(v')})$$
où les notations sont les mêmes qu'aux lemmes précédents.
\end{lem}
\textbf{Troisième étape :} Il n'y a rien à faire dans la troisième étape du fait du lemme suivant :
\begin{lem} Les orbites de $C_{\mathrm{E}_7}(\Z/q\Z)$ pour l'action de $W$ sont exactement celles pour l'action de $W^+$.
\end{lem}
\begin{proof}
il suffit de constater que : $-1\in W\setminus W^+$, et que $-1$ a une action triviale sur les éléments de $C_{\mathrm{E}_7}(\Z/q\Z)$.  
\end{proof}
\subsubsection{Le cas de $\mathrm{E}_8$.}$  $
\\ \textbf{Première étape :}
\\ \indent On reprend la définition de $\mathrm{E}_8$ donnée au paragraphe \ref{2.1}. On fait les mêmes constatations que pour le cas de $\mathrm{E}_7$ : le point clef est que l'on a l'inclusion naturelle $\mathrm{D}_8\rightarrow \mathrm{E}_8$, et que le nombre $q$ est premier à $\vert \mathrm{E}_8/\mathrm{D}_8\vert =2$. Au final, la quadrique $C_{\mathrm{E}_8}(\Z/q\Z)$ s'identifie à l'ensemble des $\Z/q$-droites $\Z/q\cdot (x_1,\cdots,x_8)$ de $(\Z/q)^8)$ vérifiant : $\sum x_i^2 = 0$.
\\ \indent Par construction, cette identification commute aux actions des groupes orthogonaux considérés. Cependant, l'action de $W' = \mathcal{S}_8\ltimes \left( \{\pm 1\}\right) ^0$ n'est pas agréable. On considère plutôt le groupe $W''=\mathcal{S}_8 \ltimes \{ \pm 1\}^8$, dont l'action est bien définie. On a notamment les lemmes suivants, dont la vérification est évidente :
\begin{lem} Soit $x\in C_{\mathrm{E}_8}(\Z/q\Z)$. Alors l'orbite de $x$ par l'action de $W''$ possède un élément $x'=\mathrm{vect}_{\Z/q\Z} (v)$ pour $v\in \mathrm{E}_8$, où $v=(v_1,\dots ,v_8)$ vérifie :
$$\begin{array}{rl}
	(i) &v_i\in \{0,\dots ,(q-1)/2\} \\
	(ii) & v_1\leq \dots \leq v_8\\
	(iii) & \mathrm{inf}\{v_i\, \vert \, v_i>0 \} = 1 \\
\end{array}$$
\end{lem}
\indent Pour savoir si deux droites de $C_{\mathrm{E}_8}(\Z/q\Z)$ correspondent à une même orbite pour l'action de $W''$, on utilise le lemme suivant dont la vérification est laissée au lecteur :
\begin{lem} Soient $x,x'\in C_{\mathrm{E}_8}(\Z/q\Z)$, et $v,v'\in (\Z/q\Z)^8$ des générateurs de $x$ et $x'$. Alors $x$ et $x'$ sont dans la même orbite de $C_{\mathrm{E}_8}(\Z/q\Z)$ pour l'action de $W''$ si, et seulement si :
$$(\exists i\in \Z/q\Z)\ \phi_8 (\overline{i\cdot v}) = \phi_8(\overline{v'})$$ 
où $\phi_8$ est la fonction introduite au paragraphe précédent, et où on a noté par $\overline{v}$ l'unique élément de $\{0,\dots ,q-1 \}^8$ ayant même image que $v$ par le passage au quotient : $\Z^8 \rightarrow (\Z/q\Z)^8$.
\end{lem}
\indent On aura ainsi pour chaque orbite pour l'action de $W''$ un représentant. Le cardinal de l'orbite est donné par le lemme suivant :
\begin{lem} Soit $x\in C_{\mathrm{E}_8}(\Z/q\Z)$ une droite isotrope, et $v=(v_1,\dots ,v_8)$ un générateur de $x$. Alors le cardinal de l'orbite de $x$ par l'action de $W''$ est donné par :
\begin{center}
\scalebox{1}{$\mathrm{card}\left( \mathcal{O}_{W''}(x)\right) =\dfrac{2^{8-\vert \{ j\in \{1,\dots ,8\} \vert v_j = 0\} \vert } \cdot 8!}{\vert \{ i\in \Z/q\Z \vert \phi_8 (\overline{i\cdot v}) = \phi_8(\overline{v}) \} \vert \cdot \displaystyle \prod_{i\in \Z/q\Z} \left(\vert \{j\in \{1,\dots 8\} \vert v_j = i \} \vert !\right)}$}
\end{center}
\end{lem}
\indent Enfin, le lien entre les orbites de $C_{\mathrm{E}_8}(\Z/q\Z)$ pour les actions de $W''$ et de $W'$ est donné par le lemme suivant :
\begin{lem} On définit la fonction $\mathrm{signe}_q:(\Z/q\Z)^8 \rightarrow \{-1,0,1\}$ par :
$$\mathrm{signe}_q(v_1,\dots ,v_8) = \prod_{i=1}^8 \varepsilon_i ,$$
avec :
$$\varepsilon_i = \left\{
	\begin{aligned}
 		&-1\text{ si }v_i\in \{ (q+1)/2,\dots ,q-1 \} .\\
 		&0\text{ si }v_i =0 .\\
 		&1\text{ si }v_i\in \{1,\dots ,(q-1)/2 \} .\\
 	\end{aligned}
 	\right.
 $$
\indent Alors si $v,v'\in (\Z/q\Z)^8$ vérifient $\mathcal{O}_{W''}(v) = \mathcal{O}_{W''}(v')$, on a l'équivalence :
$$\mathcal{O}_{W'}(v) = \mathcal{O}_{W'}(v') \Leftrightarrow \mathrm{signe}_q(v) = \mathrm{signe}_q (v') .$$
\end{lem}
\begin{proof}
Soient $v,v'\in (\Z/q\Z)^8$. On suppose pour simplifier que les images des coordonnées de $v$ et de $v'$ dans $\{0,\dots ,q-1\}^8$ sont rangées par ordre croissant (ce qui ne change rien à la généralité du problème, comme $\mathcal{S}_8 \subset W' \subset W''$).
\\ \indent On constate que la fonction $\mathrm{signe}_q$ est invariante par l'action de $W'$ ce qui donne déjà l'implication :
$$\mathcal{O}_{W'}(v) = \mathcal{O}_{W'}(v') \Rightarrow \mathrm{signe}_q(v) = \mathrm{signe}_q (v') .$$
\indent En reprenant l'élément $d$ introduit à la démonstration du lemme \ref{lemmed}, on utilise l'écriture $W''=W'\cup (W'\circ d)$ et l'égalité $\mathrm{signe}_q (d(v)) = -\mathrm{signe}_q (v)$. On a ainsi deux possibilités :
\\ \indent - si $v_1=0$ : alors on a $d( v) =v$, donc $\mathcal{O}_{W''}(v) = \mathcal{O}_{W'}(v) =\mathcal{O}_{W'} (v')$, et $\mathrm{signe}_q(v) = \mathrm{signe}_q(v')=0$.
\\ \indent - si $v_1 \neq 0$ : alors on a $\mathrm{signe}_q ({d( v)}) \neq \mathrm{signe}_q (v)$. Ainsi on peut écrire l'union disjointe : $\mathcal{O}_{W''}(v) = \mathcal{O}_{W'}(v) \sqcup \mathcal{O}_{W'}(d( v))$. Comme $v'\in \mathcal{O}_{W''}(v)$, alors on a $v'\in \mathcal{O}_{W'}( v)$ ou $\mathcal{O}_{W'}(d ( v))$. Comme la fonction $\mathrm{signe}_q$ est constante sur les orbites de $W'$, on déduit que :
$$\mathcal{O}_{W'}(v')=\mathcal{O}_{W'}(v') \Leftrightarrow v' \in \mathcal{O}_{W'}(v) \Leftrightarrow \mathrm{signe}_q (v) = \mathrm{signe}_q (v').$$
\indent D'où le résultat cherché. 
\end{proof}
\indent Le lemme suivant se déduit alors facilement :
\begin{lem} Soit $O$ une orbite de $C_{\mathrm{E}_8}(\Z/q\Z)$ pour l'action de $W''$. Alors $O$ correspond à une ou deux orbites pour l'action de $W'$.
\\ \indent S'il existe $(v_1,\dots ,v_8) \in (\Z/q\Z)^8$ avec $x=\mathrm{vect}_{\Z/q\Z}(v_1,\dots ,v_8)\in O$ tel que $\mathrm{signe}_q(\overline{v}) =0$, alors $O$ ne correspond qu'à une seule orbite pour l'action de $W'$.
\\ \indent S'il n'existe pas de tel élément, alors $O$ sera la réunion de deux orbites $O^+$ et $O^-$ pour l'action de $W'$. Si l'on se donne $x=\mathrm{vect}_{\Z/q\Z}(v_1,\dots ,v_8)\in O$ quelconque, alors $O^+$ et $O^-$ peuvent être définies par :
$$\left\{
	\begin{aligned}
		&\mathrm{vect}_{\Z/q\Z}(v_1,\dots,v_8)\in O^+ \\
		&\mathrm{vect}_{\Z/q\Z}(-v_1,\dots,v_8)\in O^- \\
		&\mathrm{card}(O^+)=\mathrm{card}(O^-) =\dfrac{\mathrm{card}(O)}{2} \\	
	\end{aligned}
\right.$$
\end{lem}
\textbf{Deuxième étape :}
\\ \indent On possède ainsi pour chaque orbite de $C_{\mathrm{E}_8}(\Z/q\Z)$ sous l'action de $W'$ un représentant et le cardinal de l'orbite. Il suffit d'utiliser l'ensemble $H_8$ construit au paragraphe \ref{4.2} pour savoir si des orbites différentes pour l'action de $W'$ sont dans la même orbite pour l'action de $W$. Le lemme suivant se déduit facilement de la définition de l'ensemble $H_8$ :
\begin{lem} Soient $x,x'\in C_{\mathrm{E}_8}(\Z/q\Z)$, et $v,v'\in (\Z/q\Z)^8$ des générateurs de $x$ et $x'$. Alors $x$ et $x'$ sont sur la même orbite de $C_{\mathrm{E}_8}(\Z/q\Z)$ pour l'action de $W$ si, et seulement si :
$$(\exists i\in \Z/q\Z) (\exists h \in H_8)\  \left \{
	\begin{aligned}
		&\phi_8 (\overline{i\cdot v}) = \phi_8(\overline{h(v')})\\
		&\mathrm{signe}_q( \overline{i\cdot v}) = \mathrm{signe}_q (\overline{h(v')}) \\
	\end{aligned}
\right.
$$
où les notations sont les mêmes qu'aux lemmes précédents.
\end{lem}
\indent \textbf{Troisième étape :} Les orbites de $C_{\mathrm{E}_8}(\Z/q\Z)$ pour $W$ et pour $W^+$ sont en général différentes. Soit $\tau_{1,2} \in \mathcal{S}_8 \cap (W\setminus W^+)$ la transposition qui échange les deux premières coordonnées : l'application définie sur $W$ par composition à gauche par $\tau_{1,2}$ est une involution qui échange $W^+$ et $W\setminus W^+$. Le lien entre les orbites de $C_{\mathrm{E}_8}(\Z/q\Z)$ pour les actions de $W$ et de $W^+$ se comprend par le lemme facile suivant :
\begin{lem} Soit $O$ une orbite de $C_{\mathrm{E}_8}(\Z/q\Z)$ pour l'action de $W$. On pose $v=(v_1,\dots ,v_8)$ tel que  : $x=\mathrm{vect}_{\Z/q\Z}(v) \in O$.
\\ \indent Alors $O$ correspond à une seule orbite pour l'action de $W^+$ si, et seulement si, un des $h(v)$ pour $h\in H_8$ possède deux coordonnées égales ou opposées.
\\ \indent Dans le cas contraire, $O$ est la réunion de deux orbites $O^+$ et $O^-$ pour l'action de $W^+$. Les orbites $O^+$ et $O^-$ sont entièrement déterminées par les relations :
$$\left\{
	\begin{aligned}
		&\mathrm{vect}_{\Z/q\Z}(v_1,v_2,\dots,v_8)\in O^+ \\
		&\mathrm{vect}_{\Z/q\Z}(v_2,v_1,\dots,v_8)\in O^- \\
		&\mathrm{card}(O^+)=\mathrm{card}(O^-) =\dfrac{\mathrm{card}(O)}{2} \\	
	\end{aligned}
\right.$$
\end{lem}
\subsubsection{Le cas de $\mathrm{E}_8\oplus \mathrm{A}_1$.}$  $
\\ \textbf{Première étape :}
\\\indent D'après la définition introduite au paragraphe \ref{2.1}, on a $\mathrm{A}_1=\Z \cdot (1,-1)$. On identifie dans la suite $\mathrm{A}_1$ à $\Z$ par l'isomorphisme : $(a,-a)\in \mathrm{A}_1 \mapsto a\in \Z$. On note dans la suite par $(v,w)$ avec $v\in \mathrm{E}_8$ et $w\in \Z$ les éléments de $\mathrm{E}_8 \oplus \mathrm{A}_1$. 
\\ \indent Les résultats pour $\mathrm{E}_8\oplus \mathrm{A}_1$ se déduisent directement de ceux de $\mathrm{E}_8$, grâce l'isomorphisme $W(\mathrm{E}_8)\times W(\mathrm{A}_1) \simeq W(\mathrm{E}_8 \oplus \mathrm{A}_1)$ suivant :
$$
	\begin{array}{rcl}
	W(\mathrm{E}_8)\times W(\mathrm{A}_1) & \tilde{\rightarrow} & W(\mathrm{E}_8\oplus \mathrm{A}_1) \\
	(\gamma _1,\gamma _2)& \mapsto & \left( (v,w) \mapsto (\gamma_1 (v),\gamma_2 (w))\right) .\\
	\end{array}
$$
\indent La forme bilinéaire symétrique $(v,w)\cdot (v',w')$ utilisée sur $(\mathrm{E}_8\oplus \mathrm{A}_1)\times (\mathrm{E}_8\oplus \mathrm{A}_1)$ et la forme quadratique $q$ associée sont données par :
$$(v,w)\cdot (v',w') = \left( \sum_i v_i\cdot v'_i\right)  +2\cdot w\cdot w'$$
$$q(v,w) = \dfrac{1}{2}\left( \sum_i v_i^2 \right) + w^2 $$
où on a pris $v=(v_1,\dots ,v_8)$ et $v'=(v'_1,\dots ,v'_8)$ dans $\mathrm{E}_8$ et $w,w'\in \Z$. Notons enfin que $W(\mathrm{A}_1)=\mathrm{O}(\mathrm{A}_1)=\{ \pm \mathrm{id} \}$.
\\ \indent On définit les groupes : $W' = \left( \mathcal{S}_8\ltimes \left( \{\pm 1\}\right) ^0\right) \times \{ \pm \mathrm{id}\}$ et $W''=\left( \mathcal{S}_8 \ltimes \{ \pm 1\}^8\right) \times \{\pm 1\}$. On a les inclusions : $W'\subset W$ et $W'\subset W''$.
\\ \indent Les lemmes suivants se déduisent directement des résultats obtenus pour $\mathrm{E}_8$ et sont donc évidents :
\begin{lem} Soit $x\in C_{\mathrm{E}_8\oplus \mathrm{A}_1}(\Z/q\Z)$. Alors l'orbite de $x$ par l'action de $W''$ possède un élément $x'=\mathrm{vect}_{\Z/q\Z} (v,w)$ pour $(v,w)\in \mathrm{E}_8\oplus \mathrm{A}_1$, où $v=(v_1,\dots ,v_8)$ vérifie :
$$\begin{array}{rl}
	(i) & v_i\in \{0,\dots ,(q-1)/2\} \\
	(ii) & w\in \{0,\dots ,(q-1)/2\} \\
	(iii) & v_1\leq \dots \leq v_8\\
	(iv) & \mathrm{inf}\{v_i\, \vert \, v_i>0 \} = 1\\
\end{array}$$
\end{lem}
\begin{lem} Soient $x,x'\in C_{\mathrm{E}_8}(\Z/q\Z)$, et $(v,w),(v',w')\in (\Z/q\Z)^8\times \Z/q\Z$ des générateurs de $x$ et $x'$. Alors $x$ et $x'$ sont dans la même orbite de $C_{\mathrm{E}_8\oplus \mathrm{A}_1}(\Z/q\Z)$ pour l'action de $W''$ si, et seulement si :
$$(\exists i\in \Z/q\Z)\ 
	\left\{
		\begin{aligned}
			&\phi_8 (\overline{i\cdot v}) = \phi_8(\overline{v'}) \\
			&(i\cdot w)^2 =w'^2 \\
		\end{aligned}
	\right.
 $$ 
où $\phi_8$ est la fonction définie au paragraphe précédent, et où on a noté par $\overline{v}$ l'unique élément de $\{0,\dots q-1 \}^8$ ayant même image que $v$ par le passage au quotient : $\Z^8 \rightarrow (\Z/q\Z)^8$.
\end{lem}
\begin{lem} Soit $x\in C_{\mathrm{E}_8\oplus \mathrm{A}_1}(\Z/q\Z)$ une droite isotrope, et $(v,w)=((v_1,\dots ,v_8),w)$ un générateur de $x$. Alors le cardinal de l'orbite de $x$ par l'action de $W''$ est donné par :
\begin{center}
\scalebox{1}{$\begin{aligned}
	&\mathrm{card} \left( \mathcal{O}_{W''}(x)\right) = \\
	& \left\{ \begin{aligned}
		&\dfrac{2^{8-\vert \{ j\in \{1,\dots ,8\} \vert v_j = 0\} \vert } \cdot 8!}{\vert \{ i\in \Z/q\Z \vert \phi_8 (\overline{i\cdot v}) =\phi_8 (\overline{v}) \} \vert \cdot \displaystyle \prod_{i\in \Z/q\Z} \left(\vert \{j\in \{1,\dots 8\} \vert v_j = i \} \vert !\right)}\text{ si }w=0\\
		&\dfrac{2^{8-\vert \{ j\in \{1,\dots ,8\} \vert v_j = 0\} \vert } \cdot 8! \cdot 2}{\vert \{ i\in \Z/q\Z \vert  (\phi_8 (\overline{i\cdot v}),(iw)^2) =(\phi_8 (\overline{v}),w^2) \} \vert \cdot \displaystyle \prod_{i\in \Z/q\Z} \left(\vert \{j\in \{1,\dots 8\} \vert v_j = i \} \vert !\right)}\text{ si }w\neq 0\\
		\end{aligned}
		\right. \\
	\end{aligned}$}
\end{center}
\end{lem}
\indent Le lien entre les orbites de $C_{\mathrm{E}_8\oplus \mathrm{A}_1}(\Z/q\Z)$ pour l'action de $W''$ et pour $W'$ se comprend exactement comme dans le cas de $\mathrm{E}_8$. On trouve de la même manière que pour $\mathrm{E}_8$ les lemmes faciles suivants :
\begin{lem} Si $(v,w),(v',w')\in (\Z/q\Z)^8\times \Z/q\Z$ vérifient $\mathcal{O}_{W''}(v,w) = \mathcal{O}_{W''}(v',w')$, on a l'équivalence :
$$\mathcal{O}_{W'}((v,w)) = \mathcal{O}_{W'}((v',w')) \Leftrightarrow \mathrm{signe}_q(v) = \mathrm{signe}_q (v') .$$
\end{lem}
\begin{lem} Soit $O$ une orbite de $C_{\mathrm{E}_8\oplus \mathrm{A}_1}(\Z/q\Z)$ pour l'action de $W''$. Alors $O$ correspond à une ou deux orbites pour l'action de $W'$.
\\ \indent S'il existe $v=(v_1,\dots ,v_8) \in (\Z/q\Z)^8$ et $w\in \Z/q\Z$ avec $x=\mathrm{vect}_{\Z/q\Z}(v,w)\in O$ tel que $v_1 =0$, alors $O$ ne correspond qu'à une seule orbite pour l'action de $W'$.
\\ \indent S'il n'existe pas de tel élément, alors $O$ sera la réunion de deux orbites $O^+$ et $O^-$ pour l'action de $W'$. On note $x=\mathrm{vect}_{\F_p}((v_1,\dots ,v_8),w)\in O$, alors $O^+$ et $O^-$ peuvent être définies par :
$$\left\{
	\begin{aligned}
		&\mathrm{vect}_{\Z/q\Z}(v_1,\dots,v_8)\in O^+ .\\
		&\mathrm{vect}_{\Z/q\Z}(-v_1,\dots,v_8)\in O^- .\\
		&\mathrm{card}(O^+)=\mathrm{card}(O^-) =\dfrac{\mathrm{card}(O)}{2} .\\	
	\end{aligned}
\right.$$
\end{lem}
\textbf{Deuxième étape :}
\\ \indent On possède ainsi pour chaque orbite de $C_{\mathrm{E}_8\oplus \mathrm{A}_1}(\Z/q\Z)$ sous l'action de $W'$ un représentant et le cardinal de l'orbite. On utilise la même méthode que pour $\mathrm{E}_8$ pour savoir si des orbites différentes pour l'action de $W'$ correspondent à la même orbite pour l'action de $W$. Concrètement, on utilise le lemme suivant qui se déduit immédiatement de la définition de $H_8$ :
\begin{lem} Soient $x,x'\in C_{\mathrm{E}_8\oplus \mathrm{A}_1}(\Z/q\Z)$, et $(v,w),(v',w')\in (\Z/q\Z)^8\times \Z/q\Z$ des générateurs de $x$ et $x'$. Alors $x$ et $x'$ sont sur la même orbite de $C_{\mathrm{E}_8}(\Z/q\Z)$ pour l'action de $W$ si, et seulement si :
$$(\exists i\in \Z/q\Z) (\exists h \in H_8)\  \left \{
	\begin{aligned}
		&\phi_8 (\overline{i\cdot v}) = \phi_8(\overline{h(v')})\\
		& (i\cdot w)^2 = w'^2 \\
		&\mathrm{signe}_q( \overline{i\cdot v}) = \mathrm{signe}_q (\overline{h(v')}) \\
	\end{aligned}
\right.
$$
où les notations sont les mêmes qu'aux lemmes précédents.
\end{lem}
\textbf{Troisième étape :} On retrouve la même situation que pour $\mathrm{E}_7$ grâce au lemme suivant :
\begin{lem} Les orbites de $C_{\mathrm{E}_8\oplus \mathrm{A}_1}(\Z/q\Z)$ pour l'action de $W$ sont exactement celles pour l'action de $W^+$.
\end{lem}
\begin{proof}
il suffit de constater que : $-1\in W\setminus W^+$, et que $-1$ a une action triviale sur $C_{\mathrm{E}_8\oplus \mathrm{A}_1}(\Z/q\Z)$. 
\end{proof}
\subsection[La détermination d'une transformation de $\mathrm{SO}_n(\R)$ transformant $L$ \\ en un $q$-voisin donné.]{La détermination d'une transformation de $\mathrm{SO}_n(\R)$ transformant $L$ en un $q$-voisin donné.}\label{4.4}
\indent Au paragraphe précédent, on a montré comment trouver, pour chaque orbite de $C_{L}(\Z/q\Z)$ pour l'action de $W^+=\mathrm{SO}(L)$, un représentant ainsi que le cardinal de l'orbite. La formule trouvée au corollaire \ref{formuletrace} fait intervenir, pour chacun des représentants $x$ trouvé au paragraphe précédent, un élément $g_x \in \mathrm{SO}_n(\Q)$ tel que le $q$-voisin $L_x$ de $L$ associé à $x$ par la proposition-définition \ref{p-voisin} s'écrive : $L_x = g_x L$.
\\ \indent La méthode qui suit explique comment, à partir d'une droite isotrope $x$ engendrée par un vecteur $v$, construire le réseau $L_x$ associé (en exhibant une famille génératrice), puis comment trouver un élément $g_x\in \mathrm{SO}_n(\Q)$ tel que $L_x = g_x L$.
\subsubsection{La construction du $q$-voisin à partir d'une droite isotrope.}
\indent On reprend la construction de la proposition-définition \ref{p-voisin}. Soient $L\in \mathcal{L}_n$, $x$ la droite isotrope de $C_L(\Z/q\Z)$ engendrée par le vecteur $v\in L$, et $L_x$ le $q$-voisin de $L$ associé à $x$ d'après la proposition-définition \ref{p-voisin}. On souhaite déterminer une famille génératrice de $L_x$.
\\ \indent La première étape consiste à créer le réseau $M$ que l'on avait défini comme l'image inverse par l'homomorphisme $L \rightarrow L/qL$ de $x^\perp$, ce que l'on décrit dans le lemme suivant :
\begin{lem} Soient $(\alpha_1,\dots ,\alpha_n)$ une $\Z$-base de $L$, et $(\overline{\alpha_1},\dots ,\overline{\alpha_n})$ son image dans $L/qL$. Alors le $\Z/q\Z$-module $ x^\perp$ est libre de rang $n-1$. De plus, si l'on se donne $i_0 \in \{ 1,\dots ,n\}$ tel que $(\overline{\alpha_{i_0}} \cdot v) \in (\Z/q\Z)^*$, la famille :
$$ \left\{ (\alpha_{i_0} \cdot v)\,  \overline{\alpha_i} -(\alpha_i \cdot v)\, \overline{\alpha_{i_0}} \ \vert \ i\in \{1,\dots ,n\}\setminus {i_0} \right\}$$
est une base de $x^\perp$. En tant qu'image inverse de $x^\perp$ par la projection $L\rightarrow L/qL$, le réseau $M$ admet pour famille génératrice la famille suivante :
$$ \left\{ (\alpha_{i_0} \cdot v)\,  \alpha_i -(\alpha_i \cdot v)\, \alpha_{i_0} \ \vert \ i\in \{1,\dots ,n\}\setminus {i_0} \right\} \bigcup \left\{ q\, \alpha_i \ \vert \ i\in \{1,\dots ,n\} \right\}.$$
\end{lem}
\begin{proof}
le rang de $x^\perp$ et l'existence de l'indice $i_0$ sont assurés par la non dégénérescence du produit scalaire sur $L$. Le reste est évident.  
\end{proof}
\indent Il est alors facile d'exhiber une famille génératrice de $L_x$ à partir de la famille génératrice de $M$ précédente. Concrètement, on a la proposition suivante :
\begin{prop} On reprend les mêmes notations que précédemment. Le réseau $L_x$ possède une famille $\Z$-génératrice à $2n$ éléments, à savoir :
$$
	\begin{array}{c}
	 \left\{ (\alpha_{i_0} \cdot v)\,  \alpha_i -(\alpha_i \cdot v)\, \alpha_{i_0} \vert i\in \{1,\dots ,n\}\setminus {i_0} \right\} \\
	 \bigcup \left\{ q\, \alpha_i \vert i\in \{1,\dots ,n\} \right\} \bigcup \left\{\dfrac{1}{q} \left( v - \dfrac{v\cdot v}{2}\, m\, \alpha_{i_0}\right) \right \}, \\
	\end{array}
$$
où $m$ est un entier tel que $m \, (v\cdot \alpha_{i_0}) \equiv 1 \mathrm{\ mod\ q}$ (qui existe bien comme $(\overline{\alpha_{i_0}}\cdot v)\in (\Z/q\Z)^*$).
\end{prop}
\begin{proof}
On reprend la construction faite à la proposition-définition \ref{p-voisin}. Le réseau $L_x$ est donné par :
$$L_x = M + \Z\,\dfrac{v'}{q},$$
où $v' = v - \dfrac{v\cdot v}{2}\, m\, \alpha_{i_0}$, et $i_0$ et $m$ sont définis au-dessus. Comme $v$ est isotrope, alors $v$ et $v'$ ont bien même image dans $L/qL$, et $v'$ vérifie $(v'\cdot v')\equiv 0\mathrm{\ mod\ }2q^2$. 
\end{proof}
\subsubsection{Détermination d'une transformation entre $L$ et $L_x$.}
\indent On suppose que $n\leq 9$. En particulier, on a les deux propriétés suivantes :
\\ \indent - deux réseaux $L,L'\in \mathcal{L}_n$ sont nécessairement isomorphes ;
\\ \indent - soit $\{\alpha_1,\dots ,\alpha_n \}$ un système de racines simples de $L\in \mathcal{L}_n$ : c'est une $\Z$-base de $L$.
\\ \indent De ces deux constatations, on déduit le lemme facile suivant : 
\begin{lem} Soient $n\leq 9$, et $L,L'\in \mathcal{L}_n$ deux réseaux. On note $\{\alpha_1,\dots ,\alpha_n \}$ un système de racines simples de $L$, et $\{ \beta_1, \dots ,\beta_n \}$ un système de racines simples de $L'$. Ce sont respectivement des bases de $L$ et $L'$. On suppose que les $\alpha_i$ et les $\beta_j$ sont numérotés de manière à donner le même diagramme de Dynkin, c'est-à-dire que les matrices $(\alpha_i\cdot \alpha_j)_{i,j}$ et $(\beta_i\cdot \beta_j)_{i,j}$ sont égales (ce qui est bien possible comme les réseaux $L$ et $L'$ sont isomorphes).
\\ \indent Alors l'unique application linéaire $g$ donnée par :
$$\begin{array}{rcl}
	g: L & \rightarrow & L' \\
	\alpha_i & \mapsto & \beta_i \\
	\end{array}
$$
est un élément de $\mathrm{O}_n(\R)$ tel que $g(L)=L'$.
\\ \indent De plus, si $g\notin \mathrm{SO}_n(\R)$, alors $g' = g\circ s_{\alpha_1} \in \mathrm{SO}_n(\R)$ et vérifie $L'=g'(L)$.
\end{lem}
\indent On souhaite construire un élément $g\in \mathrm{SO}_n(\R)$ tel que $L_x=g(L)$ (où $L_x$ est le $q$-voisin de $L$ associé à l'élément $x\in C_L(\Z/q\Z)$ d'après la proposition-définition \ref{p-voisin}). Grâce à l'étude faite au paragraphe précédent, on possède déjà une famille génératrice pour $L_x$. D'après le lemme précédent, il suffit de déterminer un système de racines simples pour $L_x$, puis de l'ordonner correctement, et enfin de créer la transformation $g$ associée. L'algorithme qui fait cela se fait selon les étapes suivantes :
\\ \indent \textbf{Première étape :} à l'aide d'une famille $\Z$-génératrice de $L_x$ (par exemple celle trouvée au paragraphe précédent), on utilise l'algorithme LLL pour avoir une $\Z$-base de $L_x$.
\\ \indent \textbf{Deuxième étape :} la forme bilinéaire sur $L_x$ est entièrement déterminée par sa valeur sur la base trouvée à la première étape. On en déduit l'ensemble des racines de $L_x$. La fonction $\mathrm{qfminim}$ de Pari-GP nous donne directement un ensemble $R^+$ de racines positives.
\\ \indent \textbf{Troisième étape :} en posant $\rho = (1/2)\cdot \sum_{\beta \in R^+} \beta$, on déduit le système de racines simples $\{ \beta_1,\dots ,\beta_n\}$ associé à la chambre de Weyl contenant $\rho$.
\\ \indent \textbf{Quatrième étape :} on réordonne ensuite les indices des $\beta_i$ pour qu'ils donnent le même diagramme de Dynkin que les $\alpha_i$ (où les $\alpha_i$ sont un système de racines simples de $L$, ordonnées comme dans \cite[Planches VI et VII]{Bo} par exemple).
\\ \indent \textbf{Cinquième étape :} on construit l'application $g$ définie au lemme précédent.
\\ \\ \indent Seules les quatrième et cinquième étapes diffèrent selon les réseaux considérés. On les détaille ci-dessous dans le cas où $L=\mathrm{E}_7$ :
\\ \\ \indent \textbf{L'exemple de $\mathrm{E}_7$ :}
\\ \indent On reprend les notations des paragraphes \ref{2.1} et \ref{4.2}. On pose $\alpha_i=v_i$ pour $i=1,\dots ,7$ (qui est bien un système de racines simples de $\mathrm{E}_7$). Le diagramme de Dynkin et la matrice de Gram associés sont donnés par :
$$\begin{tikzpicture}
\begin{scope}[start chain]
\dnode{1}
\foreach \dyni in {3,...,7} {
\dnode{\dyni}
}
\end{scope}
\begin{scope}[start chain=br going above]
\chainin (chain-3);
\dnodebr{2}
\end{scope}
\end{tikzpicture}\indent \left( \begin{matrix}
2&0&-1&0&0&0&0\\
0&2&0&-1&0&0&0\\
-1&0&2&-1&0&0&0\\
0&-1&-1&2&-1&0&0\\
0&0&0&-1&2&-1&0\\
0&0&0&0&-1&2&-1\\
0&0&0&0&0&-1&2\\
\end{matrix} \right)$$
\indent Donnons-nous un système de racines simples $\{ \beta_1,\dots ,\beta_7 \}$ de $L_x$ (celui qu'on a trouvé à la troisième étape grâce à Pari-GP). On pose $B=(b_{i,j})= (\beta_i \cdot \beta_j )$ la matrice de Gram associée. On explique ci-dessous comment trouver les entier $i_1,\dots,i_7$ tels que :
\\ \begin{tabular}{rp{14cm}}
$(i)$ & $\{i_1,\dots,i_7\} = \{1,\dots,7\}$.\\
$(ii)$ & les matrices $B' = (\beta_{i_l}\cdot \beta_{i_m})_{l,m}$ et $A = (\alpha_{l}\cdot \alpha_{m})_{l,m}$ sont égales.
\end{tabular}
\\ \indent Et une telle numérotation est unique.
\\ \\ \indent On procède comme suit :
\\ \indent - l'entier $i_4$ est l'indice de l'unique ligne de $B$ dont la somme des coefficients est $-1$.
\\ \indent - on définit les ensembles $I_1,I_2,I_3,I_4 \subset \{ 1,\dots ,7 \}$ par :
$$I_1=\{ j\in \{1,\dots ,7 \} \vert b_{i_4,j} =-1 \},$$
$$I_2=\{ j\in \{1,\dots ,7 \} \vert \sum_i b_{i,j} =1 \},$$
$$I_3=\{ j\in \{1,\dots ,7 \} \vert \sum_i b_{i,j} =0 \},$$
$$I_4=\{ j\in \{1,\dots ,7 \} \vert (\exists i\in I_1) b_{i,j} =-1 \}.$$
\indent - d'après la numérotation du diagramme de Dynkin de $\mathrm{E}_7$, on a les égalités :
$$\{ i_2,i_3,i_5 \} = I_1,$$
$$\{ i_1,i_2,i_7 \} = I_2,$$
$$\{ i_3,i_5,i_6 \} = I_3,$$
$$\{ i_1,i_4,i_6 \} = I_4.$$
\indent - dans un premier temps, on déduit : $i_1=I_2 \cap I_4$, $i_2=I_1 \cap I_2$ et $i_6 = I_3\cap I_4$.
\\ \indent - on déduit ensuite : $i_7=I_2 \setminus \{i_1,i_2 \}.$
\\ \indent - enfin on trouve $i_3=\{ j\in \{1,\dots ,7 \} \vert b_{i_1,j} =-1 \}$, et $i_5$ est le dernier indice restant.
\section{L'étude de $\mathrm{T}_A$ pour $A$ un $2$-groupe.}\label{5}
\indent Dans cette partie, on considère $A$ un $2$-groupe de la forme $A=(\Z/2\Z)^i$ (avec $i\leq n/2$) ou $A=\Z/4\Z$. Soient $n=7,8$ ou $9$, et $L\in \mathcal{L}_n$. On note $R=R(L)$ l'ensemble des racines de $L$, $W$ son groupe de Weyl, et $W^+=\mathrm{SO}(L)$ le sous-groupe de $W$ des éléments de déterminant $1$. On souhaite déterminer les orbites de $A$-voisins du réseau $L$ sous l'action du groupe $W^+$.
\\ \indent Une particularité du cas $p=2$ repose sur le lemme suivant :
\begin{lem} \label{bourbaki} Soit $L$ un réseau du type $\mathrm{E}_7$ ou $\mathrm{E}_8$. On pose $\tilde{q}$ la forme quadratique obtenue à partir de $q$ par passage au quotient dans $L/2L$. L'homomorphisme naturel suivant :
$$W\rightarrow \mathrm{O}(\tilde{q})=\mathrm{O}(L/2L)$$
est surjectif.
\end{lem}
\begin{proof}
voir \cite[ch.VI, \S 4, exercices 1 et 3]{Bo}.  
\end{proof}
\subsection{Détermination des orbites de $A$-voisins pour les actions des groupes $W$ et $W^+$.}\label{5.1}
\subsubsection{Les orbites des $2$-voisins et des $4$-voisins pour l'action de $W^+$.}
\indent D'après la proposition \ref{paramètre A-vois} et la proposition-définition \ref{p-voisin}, les orbites des $2$-voisins ou des $4$-voisins pour l'action de $W^+$ se comprennent par l'étude des orbites des droites isotropes de $L/2L$ ou de $L/4L$ pour l'action de $W^+$.
\\ \indent Les orbites de droites isotropes de $L
/2L$ pour l'action de $W^+$ sont expliquées par les corollaires suivants du lemme \ref{bourbaki} :
\begin{cor} Soit $L$ le réseau $\mathrm{E}_7$ ou $\mathrm{E}_8$, et $W$ son groupe de Weyl. Soit $\tilde{q}$ la forme quadratique obtenue à partir de $q$ par passage au quotient dans $L/2L$. Alors $L\otimes \F_2 -\{ 0\}$ ne possède que deux orbites pour l'action de $W^+$, à savoir : $\tilde{q}^{-1}(0) -\{0 \}$ et $\tilde{q}^{-1}(1)$.
\end{cor}
\begin{proof}
découle du lemme \ref{bourbaki} et du théorème de Witt. En effet, si on se donne deux éléments non nuls $u,v\in L\otimes \F_2$ avec $\tilde{q}(u) = \tilde{q}(v)$, le théorème de Witt nous donne un élément $s\in \mathrm{O}(\tilde{q})$ tel que $s(u)=v$. Et un tel élément $s$ provient d'un élément $w\in W$ d'après le lemme \ref{bourbaki}. Ainsi, $L\otimes \F_2 -\{ 0\}$ ne possède que deux orbites pour l'action de $W$, à savoir : $\tilde{q}^{-1}(0) -\{0 \}$ et $\tilde{q}^{-1}(1)$.
\\ \indent Si $u\in \tilde{q}^{-1}(0)-\{0\}$ ou si $u\in \tilde{q}^{-1}(1)$, on peut toujours trouver un vecteur $u'\in L$ laissé stable par un élément $\sigma \in W\setminus W^+$, dont l'image par $L\rightarrow L/2L$ est $u$. Par exemple, si $\sigma$ désigne la permutation des deux dernières coordonnées dans $\R^8$, on peut voir que les vecteurs $u'=(1,-1,0,\dots ,0)$ et $u'(1,1,-1,-1,0,\dots ,0)$ conviennent. Ainsi, les orbites de $L\otimes \F_2 -\{ 0\}$ pour les actions de $W$ et de $W^+$ coïncident. D'où le résultat. 
\end{proof}
\indent Les corollaires suivants se déduisent du fait que les $2$-voisins du réseau $L$ sont en bijection avec les éléments de $C_L(\Z/2\Z)$, et que cette bijection commute aux actions de $W^+$ :
\begin{cor} Soient $L=\mathrm{E}_7$ ou $\mathrm{E}_8$, $W$ son groupe de Weyl, et $W^+=\mathrm{SO}(L)$. Le groupe $W^+$ agit transitivement sur $C_L(\F_2)$ ainsi que sur l'ensemble des $2$-voisins de $L$.
\end{cor}
\begin{cor} Soient $L=\mathrm{E}_8\oplus \mathrm{A}_1$, $W$ son groupe de Weyl, et $W^+=\mathrm{SO}(L)$. Notons $q_1$ et $q_2$ les formes quadratiques associées à la projection de $L$ sur $\mathrm{E}_8$ et $\mathrm{A}_1$ respectivement. Alors $C_L(\F_2)$ possède exactement deux orbites sous l'action de $W^+$, qui sont :
$$\begin{aligned}
	\left( q_1^{-1}(0) -\{0\} \right) & \times \{0 \}\\
	\left( q_1^{-1}(1) \right) &\times \left( q_2^{-1}(1)\right) \\
	\end{aligned}
$$
\indent Il y a donc deux orbites de $2$-voisins de $L$ sous l'action de $W^+$, qui sont associées aux deux orbites de $C_L(\F_2)$ par la bijection de la proposition-définition \ref{p-voisin}.
\end{cor}
\indent De même, nous allons voir que, pour $L=\mathrm{E}_7$ ou $L=\mathrm{E}_8$, le groupe $W^+=\mathrm{SO}(L)$ agit transitivement sur $C_L(\Z/4\Z)$, et donc sur l'ensemble des $4$-voisins de $L$.
\\ \indent Pour cela, posons $X=\{ x\in \mathrm{E}_7 \setminus 2\mathrm{E}_7 \, \vert \, (x\cdot x) =8 \}$. On a les lemmes suivants :
\begin{lem} On a les propriétés suivantes :
\\ \begin{tabular}{rp{10cm}}
$(i)$ & $\vert X \vert =4032 = 2\cdot 2^5\cdot 63$. \\
$(ii)$ & l'action de $W^+$ sur $X$ est transitive.
\end{tabular}
\end{lem}
\begin{proof}
Le point $(i)$ se vérifie facilement à la main, car il est élémentaire d'énumérer tous les éléments de $X$. On pourrait aussi utiliser un argument de série thêta, en utilisant le développement suivant :
$$\theta_{\mathrm{E}_7}(q) = 1 + 126\cdot q + 756\cdot q^2 + 2072\cdot q^3 + 4158\cdot q^4 + 7560\cdot q^5 + \dots .$$
\indent Pour le $(ii)$, on utilise la description faite de $W$ au paragraphe \ref{4.2}. Il est alors facile de vérifier qu'il n'y a qu'une seule orbite d'éléments de $X$ pour l'action de $W$. Comme le vecteur $(2,1,-1,-1,-1,0,0,0)\in X$ est invariant par la permutation $\sigma\in W\setminus W^+$ qui échange les deux dernières coordonnées, on déduit qu'il n'y a qu'une seule orbite d'éléments de $X$ pour l'action de $W^+$. 
\end{proof}
\indent On veut montrer que toute droite de la quadrique $C_{\mathrm{E}_7}(\Z/4\Z)$ est engendrée par un vecteur de $X$, ce qui découle du lemme suivant : 
\begin{lem} L'application de réduction modulo $4$ : $X\rightarrow \mathrm{E}_7/4\mathrm{E}_7$ est injective. De plus, elle induit une bijection :
$$
	\begin{array}{rcl}
		\left\{ \{x,-x\} \vert x\in X \right\} &\rightarrow & C_{\mathrm{E}_7}(\Z/4\Z) .\\
		\{x,-x\} & \mapsto & \mathrm{vect}_{\Z/4\Z} (x) .\\
	\end{array}
$$ 
\end{lem}
\begin{proof}
Montrons d'abord l'injectivité de l'application de réduction modulo $4$ : $X\rightarrow \mathrm{E}_7/4\mathrm{E}_7$. Soient $x\in X$, et $u$ non nul dans $\mathrm{E}_7$ tels que $(x-4\cdot u) \in X$. En particulier, on a les égalités : $(x-4u)\cdot (x-4u) =8$ et $x\cdot x=8$, donc $4 u\cdot u =x\cdot u$.
\\ \indent Le théorème de Cauchy-Schwarz nous dit que : $4 (u\cdot u)^2 \leq 8 (u\cdot u)$, puis $u\cdot u\leq 2$. Comme $u\in \mathrm{E}_7$ et que $u\neq 0$, on déduit que $u\cdot u=2$. Ainsi, toutes les inégalités précédentes sont des égalités, et les vecteurs $x$ et $u$ sont $\R$-proportionnels. En cherchant $x$ de la forme $\lambda \cdot u$ pour $\lambda \in \R$, on trouve facilement que $x=2 u$, ce qui est impossible par définition de $X$. D'où l'injectivité cherchée.
\\ \indent Il reste à déterminer le caractère bijectif de l'application $\{ x,-x \} \mapsto \mathrm{vect}_{\Z/4\Z}(x)$. Cette application est déjà injective d'après le premier point du lemme. Pour voir qu'elle est surjective, il suffit de voir que $\vert X \vert =2\cdot \vert C_{\mathrm{E}_7}(\Z/4\Z) \vert$. C'est un exercice de vérifier que $\vert C_{\mathrm{E}_7}(\Z/4\Z) \vert = 2^5 \cdot \vert C_{\mathrm{E}_7}(\Z/2\Z) \vert =2^5 \cdot (2^6-1) = 2^5 \cdot 63$. 
\end{proof}
\begin{cor} L'action naturelle du groupe $W^+=\mathrm{SO}(\mathrm{E}_7)$ sur l'ensemble des $4$-voisins de $\mathrm{E}_7$ est transitive.
\end{cor} 
\begin{proof}
c'est une conséquence immédiate des lemmes précédents. On sait déjà que $W^+$ agit transitivement sur l'ensemble $X$, et on voit facilement que la bijection $\{x,-x\} \mapsto \mathrm{vect}_{Z/4\Z}(x)$ commute avec les actions de $W^+$ sur $X$ et sur $C_{\mathrm{E}_7}(\Z/4\Z)$. Ainsi, $W^+$ agit transitivement sur $C_{\mathrm{E}_7}(\Z/4\Z)$, donc sur l'ensemble des $4$-voisins de $\mathrm{E}_7$. 
\end{proof}
\indent La même méthode s'applique à l'étude des $4$-voisins de $\mathrm{E}_8$. On se contente d'exposer les résultats analogues au cas de $\mathrm{E}_7$, en laissant les démonstrations au lecteur.
\\ \indent Posons $Y=\{ y\in \mathrm{E}_8 \setminus 2\mathrm{E}_8 \, \vert \, (y\cdot y) =8 \}$. On a les lemmes suivants :
\begin{lem} On a les propriétés suivantes :
\\ \begin{tabular}{rp{10cm}}
$(i)$ & $\vert Y \vert =17280 = 2\cdot 2^6\cdot 135$. \\
$(ii)$ & l'action de $W^+$ sur $Y$ est transitive.
\end{tabular}
\end{lem}
\begin{lem} L'application de réduction modulo $4$ : $Y\rightarrow \mathrm{E}_8/4\mathrm{E}_8$ est injective. De plus, elle induit une bijection :
$$
	\begin{array}{rcl}
		\left\{ \{y,-y\} \vert y\in Y \right\} &\rightarrow & C_{\mathrm{E}_8}(\Z/4\Z) .\\
		\{y,-y\} & \mapsto & \mathrm{vect}_{\Z/4\Z} (y) .\\
	\end{array}
$$ 
\end{lem}
\begin{cor} L'action naturelle du groupe $W^+=\mathrm{SO}(\mathrm{E}_8)$ sur l'ensemble des $4$-voisins de $\mathrm{E}_8$ est transitive.
\end{cor}
\subsubsection{Les orbites des $2,\dots ,2$-voisins pour l'action de $W^+$.}
\indent Dans cette partie, on ne considère que les cas où $L=\mathrm{E}_7$ ou $L=\mathrm{E}_8$. On pose $A=(\Z/2\Z)^i$, et on s'intéresse aux $A$-voisins de $L$. Notre but est d'étudier les orbites de $A$-voisins de $L$ pour l'action de $W^+ = \mathrm{SO}(L)$, ce qui repose sur les lemmes suivants :
\begin{lem} \label{unique M} Soient $n=7$ ou $8$, et $1\leq i\leq n/2$. \`A isomorphisme près, il existe un unique réseau $M\subset \R^n$ pair et tel que :
$$\mathrm{r\acute{e}s}\, M \simeq \mathrm{H}\left( (\Z/2\Z)^i \right) \oplus \mathrm{r\acute{e}s}\, L_0,$$
où $L_0$ est un élément quelconque de $\mathcal{L}_n$.
\end{lem}
\begin{proof}
Soient $I,I'$ deux espaces isotropes isomorphes à $(\Z/2\Z)^i$ et en somme directe dans $\mathrm{r\acute{e}s}\, M$ (qui existent bien par hypothèse). Considérons $L$ et $L'$ les images réciproques respectives de $I$ et $I'$ par l'application naturelle $M^\sharp \rightarrow \mathrm{r\acute{e}s}\, M$ : ce sont deux éléments de $\mathcal{L}_n$, qui sont des $(\Z/2\Z)^i$-voisins.
\\ \indent \`A isométrie près, $\mathcal{L}_n$ ne possède qu'un seul élément, et on peut donc supposer que $L=\mathrm{E}_7$ ou $L=\mathrm{E}_8$ (selon le choix de $n$). D'après la proposition \ref{paramètre A-vois}, $M$ est l'image réciproque via l'application naturelle $L\rightarrow L/2L$ de l'orthogonal d'un sous-espace isotrope de dimension $i$. Comme $\mathrm{O}(L)$ permute transitivement ces sous-espaces (par le théorème de Witt et le lemme \ref{bourbaki}), $M$ est bien unique à isomorphisme près.
\\ \indent Il reste à montrer l'existence de tels réseaux $M$ : celle-ci est claire grâce à l'existence de $(\Z/2\Z)^i$-voisins. On donnera dans la suite des exemples de $M$ satisfaisant les propriétés de l'énoncé du lemme. 
\end{proof}
\begin{lem}\label{transitivité I} Soient $n$, $i$ et $M$ comme dans le lemme \ref{unique M}. Alors le groupe $\mathrm{O}(M)$ agit transitivement sur l'ensemble des couples $(I,I')$, où $I$ et $I'$ sont des sous-espaces isotropes de dimension $i$ en somme directe dans $\mathrm{r\acute{e}s}\, M$.
\end{lem}
\begin{proof}
On procède cas par cas. Comme le réseau $M$ est unique à isomorphisme près, on comprend qu'il suffit de montrer le lemme pour les différentes valeurs possibles pour $i$ et $n$, mais à chaque fois pour un seul réseau $M$ tel que $\mathrm{r\acute{e}s}\, M \simeq \mathrm{H}\left( (\Z/2\Z)^i \right) + \mathrm{r\acute{e}s}\, L_0$.
\begin{proof}[Démonstration du cas $n=7$ et $i=1$ ] posons $M=\mathrm{A}_1 \oplus \mathrm{D}_6$ (d'après les notations du paragraphe \ref{2.1}). Notons $\varepsilon =(1/2,-1/2) \in \mathrm{A}_1^\sharp$, qui est un générateur de $\mathrm{A}_1^\sharp /\mathrm{A}_1$, et $e^\pm = (\pm 1/2,1/2,\dots ,1/2) \in \mathrm{D}_6^\sharp$ qui engendrent $\mathrm{D}_6^\sharp /\mathrm{D}_6$. On a les égalités :
$$\begin{aligned}
\mathrm{r\acute{e}s}\, M =& \Z/2\Z \varepsilon \oplus \Z/2\Z e^+ \oplus \Z/2\Z e^- \\
=& \left( \Z/2\Z (\varepsilon+e^+) \oplus \Z/2\Z (\varepsilon + e^-) \right) \oplus  \Z/2\Z (\varepsilon +e^+ + e^-) \\
\simeq & \mathrm{H}(\Z/2\Z) \oplus \mathrm{r\acute{e}s}\, \mathrm{E}_7 .\\
\end{aligned}$$
\indent Il existe seulement deux droites isotropes de dimension $1$ et elles sont en somme directe dans $\mathrm{r\acute{e}s}\, M$. Ce sont les droites engendrées par $\varepsilon +e^+$ et par $\varepsilon +e^-$.
\\ \indent Soit $\sigma \in \mathrm{O}(\mathrm{D}_6)$ la symétrie orthogonale par rapport au vecteur $(1,0,\dots ,0)$ : elle échange $e^+$ et $e^-$. L'élément $\mathrm{id}\oplus \sigma$ est donc un élément de $\mathrm{O}(M)$ qui échange les droites engendrées par $\varepsilon +e^+$ et $\varepsilon + e^-$ : le groupe $\mathrm{O}(M)$ agit transitivement sur l'ensemble des couples de droites isotropes en somme directe dans $\mathrm{r\acute{e}s}\, M$.  
\end{proof}
\begin{proof}[Démonstration du cas $n=7$ et $i=2$] posons $M=(\mathrm{A}_1)^3\oplus \mathrm{D}_4$ (d'après les notations du paragraphe \ref{2.1}). Notons $\varepsilon_k = (1/2,-1/2)\in \mathrm{A}_1^\sharp$, qui est un générateur du $k$-ème terme $\mathrm{A}_1^\sharp /\mathrm{A}_1$ apparaissant dans l'écriture de $\mathrm{r\acute{e}s}\, M$, et $e^\pm = (\pm 1/2,1/2,1/2,1/2) \in \mathrm{D}_4^\sharp$ qui engendrent $\mathrm{D}_4^\sharp /\mathrm{D}_4$. Enfin, on note $P$ le plan $(\F_2)^2$ muni de la forme quadratique $(x,y) \mapsto (1/2)(x^2+xy+y^2)$ à valeurs dans $\Q/\Z$. On a les égalités :
$$\begin{aligned}
\mathrm{r\acute{e}s}\, M =& \Z/2\Z \varepsilon_1 \oplus \Z/2\Z \varepsilon_2 \oplus \Z/2\Z \varepsilon_3 \oplus \Z/2\Z e^+ \oplus \Z/2\Z e^- \\
=& \left( \Z/2\Z e^+ \oplus \Z/2\Z e^- \right) \oplus \left( \Z/2\Z (\varepsilon_1 +\varepsilon_2) \oplus \Z/2\Z (\varepsilon_1 +\varepsilon_3) \right) \oplus  \Z/2\Z (\varepsilon_1 +\varepsilon_2 + \varepsilon_3) \\
\simeq & P\oplus P \oplus \mathrm{r\acute{e}s}\, \mathrm{E}_7 \simeq \mathrm{H}\left( (\Z/2\Z)^2\right) \oplus \mathrm{r\acute{e}s}\, \mathrm{E}_7 .\\
\end{aligned}$$
\indent Seule la dernière égalité n'est pas évidente : elle provient de \cite[ch. II, proposition 2.1]{CL}. Il suffit de constater que le sous-ensemble $\{ x\oplus x \vert x\in P \} \subset P\oplus P$ est un lagrangien de $P\oplus P$, et on a donc un isomorphisme $P\oplus P \simeq H\left( (\Z/2\Z)^2 \right)$.
\\ \indent Reste donc à comprendre l'action de $\mathrm{O}(M)$ sur les couples de plans isotropes de $\mathrm{r\acute{e}s}\, M$. Pour cela, il suffit d'utiliser l'inclusion : $\mathcal{S}_3\times \mathcal{S}_3 \subset \mathrm{O}(M)$, où le premier $\mathcal{S}_3$ agit par permutation des $\varepsilon_i$, et où le second $\mathcal{S}_3$ agit par permutation des vecteurs non-nuls de $\mathrm{r\acute{e}s}\, \mathrm{D}_4 \simeq P$. On remarque que le groupe $\mathrm{O}(P)$ agit par permutation des vecteurs non-nuls de $P$ : fixons pour la suite un isomorphisme $\phi : \mathcal{S}_3 \simeq \mathrm{O}(P)$ (ce qui revient à fixer une numérotation des vecteurs non-nuls de $P$). L'inclusion $\mathcal{S}_3\times \mathcal{S}_3 \subset \mathrm{O}(M)$ induit l'inclusion $\mathrm{O}(P) \times \mathrm{O}(P) \subset \mathrm{O}(M)$.
\\ \indent Les plans isotropes de $\mathrm{r\acute{e}s}\, M$ sont obtenus à partir des plans isotropes de $P\oplus P$. Si l'on se donne $I$ un tel plan, comme $P$ ne contient pas d'élément isotrope non nul, on a les égalités : $I\cap (P \oplus \{0 \}) = I\cap (\{ 0\} \oplus P ) = \{ 0\}$. La dimension de $I$ impose qu'il existe un unique élément $\sigma$ de $\mathcal{S}_3$ tel que : $I=\{ u + \phi(\sigma) (u) \in P\oplus P \vert u\in P \}$. Ceci montre déjà que $\mathrm{O}(M)$ agit transitivement sur les plans isotropes de $\mathrm{r\acute{e}s}\, M$.
\\ \indent Soient $I$ et $I'$ deux plans isotropes en somme directe dans $\mathrm{r\acute{e}s}\, M$. Quitte à faire agir un élément de $\mathrm{O}(M)$ bien choisi, on peut supposer que $I=\{ u+u \in P \oplus P \vert u\in P  \}$. Les espaces $I$ et $I'$ sont en somme directe si, et seulement si, l'élément $\sigma \in \mathcal{S}_3$ associé à $I'$ ne fixe aucun point (c'est-à-dire que c'est un $3$-cycle).
\\ \indent Notons $\sigma_1=(123)$ et $\sigma_2=(132)$ les deux $3$-cycles de $\mathcal{S}_3$, et notons $I_1,I_2$ les plans isotropes associés. Soient $\tau\in \mathcal{S}_3$ la transposition $(23)$, et $\overline{\tau}=\tau \times \tau \in \mathcal{S}_3 \times \mathcal{S}_3 \subset\mathrm{O}(M)$. L'action de $\overline{\tau}$ sur $P\oplus P \subset \mathrm{r\acute{e}s}\, M$ a bien un sens, en considérant l'élément $\phi(\tau) \times \phi (\tau) \in \mathrm{O}(P) \times \mathrm{O}(P)$. De plus, cette action échange $I_1$ et $I_2$, et laisse stable $I$. Ainsi, le groupe $\mathrm{O}(M)$ agit bien transitivement sur les couples de plans isotropes de $\mathrm{r\acute{e}s}\, M$.
\end{proof}
\begin{proof}[Démonstration du cas $n=7$ et $i=3$] posons $M=(\mathrm{A}_1)^7$ (d'après les notations du paragraphe \ref{2.1}). Notons $\varepsilon_k = (1/2,-1/2)\in \mathrm{A}_1^\sharp$, qui est un générateur du $k$-ème terme $\mathrm{A}_1^\sharp /\mathrm{A}_1$ apparaissant dans l'écriture de $\mathrm{r\acute{e}s}\, M$. On a les égalités :
$$\begin{aligned}
\mathrm{r\acute{e}s}\, M & = \bigoplus_{k=1}^7 \Z/2\Z \varepsilon_k = \left( \bigoplus_{k=2}^7 \Z/2\Z (\varepsilon_1 + \varepsilon_k) \right) \oplus \Z/2\Z \left( \sum_{k=1}^7 \varepsilon_k \right)\\
&\simeq \mathrm{H}\left( (\Z/2\Z)^3\right) \oplus \mathrm{r\acute{e}s}\, \mathrm{E}_7\\
\end{aligned}
$$
\indent Comme dans le cas précédent, la dernière égalité provient de \cite[ch. II, proposition 2.1]{CL} : il suffit de constater que l'espace $I_0$ engendré par les vecteurs $\varepsilon_1+\varepsilon_2+\varepsilon_3+\varepsilon_4,\varepsilon_1+\varepsilon_2+\varepsilon_5+\varepsilon_6$ et $\varepsilon_1+\varepsilon_3+\varepsilon_5+\varepsilon_7$ est un lagrangien de $\bigoplus_{k=2}^7 \Z/2\Z (\varepsilon_1 + \varepsilon_k)$.
\\ \indent Par définition de $M$, on a : $\mathrm{O}(M) = \mathcal{S}_7 \ltimes \{ \pm 1\}^7$, où $\{ \pm 1\}^7$ agit sur chaque $\mathrm{A}_1$ par $\varepsilon_k \mapsto \pm \varepsilon_k$ (et a donc une action triviale sur $\mathrm{r\acute{e}s}\, M$), tandis que $\mathcal{S}_7$ agit par permutation des $\varepsilon_k$. On déduit déjà que l'action de $\mathrm{O}(M)$ sur $\mathrm{r\acute{e}s}\, \mathrm{E}_7 \subset \mathrm{r\acute{e}s}\, M$ est triviale. La compréhension des espaces isotropes de dimension $3$ de $\mathrm{r\acute{e}s}\, M$ se déduit de la compréhension des lagrangiens de $\mathrm{H}\left( (\Z/2\Z)^3 \right)$ donnée par le lemme suivant :
\begin{lem} Soit $\phi$ l'application définie sur $\mathcal{S}_7$ par :
$$\begin{array}{rrcl}
\phi :& \mathcal{S}_7 & \rightarrow & B_I \\ 
& \sigma & \mapsto & \left( \mathrm{vect}_{\Z/2\Z}(v_1,v_2,v_3), (v_1,v_2,v_3) \right) \\
\end{array}
$$
où $B_I$ est l'ensemble des couples $(I,b)$, où $I$ est un lagrangiens de $H\left( (\Z/2\Z)^3 \right)$, $b$ est une base de $I$, et où les vecteurs $v_1,v_2,v_3$ sont donnés par :
$$
	\left\{ \begin{aligned}
		v_1&= \varepsilon_{\sigma (1) }+\varepsilon_{\sigma (2) }+\varepsilon_{\sigma (3) }+\varepsilon_{\sigma (4) } \\
		v_2&= \varepsilon_{\sigma (1) }+\varepsilon_{\sigma (2) }+\varepsilon_{\sigma (5) }+\varepsilon_{\sigma (6) } \\
		v_3&= \varepsilon_{\sigma (1) }+\varepsilon_{\sigma (3) }+\varepsilon_{\sigma (5) }+\varepsilon_{\sigma (7) } \\
		\end{aligned}
	\right. .
$$
\indent Alors $\phi$ est une bijection.
\end{lem}
\begin{proof}[Démonstration du lemme] Il est immédiat que $\phi$ est bien à valeurs dans $B_I$. Il suffit en effet de voir que les vecteurs $v_1,v_2,v_3$ sont isotropes, orthogonaux, et forment une famille libre.
\\ \indent Réciproquement, donnons-nous $(v_1,v_2,v_3)$ une base d'un lagrangien $I$. Pour simplifier, on note $I_i$ les sous-ensembles de $\{1,\dots ,7\}$ tels que : $v_i = \sum_{j\in I_i} \varepsilon_j$. Les $I_i$ sont deux-à-deux distincts. Comme $I$ est un lagrangien, on a en particulier que :
$$(\forall \ i\in \{1,2,3\} ) \ \vert I_i \vert \equiv 0\ \mathrm{mod}\ 4,$$
$$(\forall \ i,j\in \{1,2,3\} ) \ \vert I_i\cap I_j \vert \equiv 0\ \mathrm{mod}\ 2.$$
\indent Comme les $I_i$ sont des sous-ensembles de $\{1,\dots ,7\}$, les congruences ci-dessus donnent les égalités :
$$(\forall \ i\in \{1,2,3\} ) \ \vert I_i \vert =4,$$
$$(\forall \ i \neq j\in \{1,2,3\} ) \ \vert I_i\cap I_j \vert =2.$$
\indent Il est alors facile de voir que l'on a l'égalité :
$$\vert I_1 \cap I_2 \cap I_3 \vert =1 .$$
\indent En effet, les seules autres valeurs possibles sont $0$ ou $2$. Mais ces deux cas sont impossibles :
\\ \indent - si $\vert I_1 \cap I_2 \cap I_3 \vert=0$ : alors on voit facilement que $I_3=I_1\cup I_2 \setminus I_1 \cap I_2$, ce qui veut dire que : $v_3=v_1+v_2$, ce qui contredit que $(v_1,v_2,v_3)$ forme une base de $I$.
\\ \indent - si $\vert I_1 \cap I_2 \cap I_3 \vert=2$, alors $I_1\cap I_2 =I_2\cap I_3 = I_3\cap I_1 =J$, et les ensembles $I_i \setminus J$ formeraient trois ensembles disjoints à deux éléments de $\{1,\dots ,7 \} \setminus J$, qui ne possède que $5$ éléments, ce qui est impossible.
\\ \indent On peut ainsi définir une permutation $\sigma \in \mathcal{S}_7$ avec :
$$\{ \sigma (1)\}= I_1 \cap I_2 \cap I_3,$$
$$\{ \sigma(2) \}=  I_1\cap I_2 \setminus \{ \sigma (1)\},$$
$$\{ \sigma (3)\}= I_1\cap I_3 \setminus \{ \sigma (1)\},$$
$$\{ \sigma (4)\}=I_1 \setminus \{ \sigma (1),\sigma(2),\sigma(3)\},$$
$$\{ \sigma (5)\} = I_2\cap I_3 \setminus \{ \sigma (1)\},$$
$$\{ \sigma (6)\} = I_2 \setminus \{ \sigma (1), \sigma(2), \sigma(5)\},$$
$$\{ \sigma (7)\} = I_3 \setminus \{ \sigma (1),\sigma(3),\sigma(5)\}.$$
\indent On vérifie facilement que pour notre élément $\sigma$ défini ci-dessus on a : $\phi(\sigma) = (I,(v_1,v_2,v_3))$. L'application $\phi$ est une bijection, et sa réciproque est donnée par le processus précédent qui nous a permis de trouver $\sigma$.  
\end{proof}
\indent Revenons à la démonstration du lemme \ref{transitivité I} dans le cas $n=7$ et $i=3$. Soit $I$ un espace isotrope de dimension $3$ de $\mathrm{r\acute{e}s}\, M$ : on peut le voir comme un lagrangien de $\bigoplus_{k=2}^7 \Z/2\Z (\varepsilon_1+\varepsilon_k) \simeq \mathrm{H}\left( (\Z/2\Z)^3 \right)$, et d'après le lemme précédent il existe un élément $\sigma \in \mathcal{S}_7$ associé à une de ses bases. Il est facile de voir que $I=\sigma (I_0)$ (avec l'inclusion $\mathcal{S}_7 \subset \mathrm{O}(M)$), donc $\mathrm{O}(M)$ agit transitivement sur l'ensemble des espaces isotropes de dimension $3$ de $\mathrm{r\acute{e}s}\, M$.
\\ \indent On souhaite vérifier que $\mathrm{O}(M)$ agit bien transitivement sur les couples d'espaces isotropes de dimension $3$ en somme directe dans $\mathrm{r\acute{e}s}\, M$. Soit $(I,J)$ un tel couple. On peut supposer que $I=I_0$ (comme $\mathrm{O}(M)$ agit transitivement sur les espaces isotropes de dimension $3$ de $\mathrm{r\acute{e}s}\, M$). L'espace isotrope $J_0$ associé par $\phi$ au $4$-cycle $\sigma=(4576)$ est bien en somme directe avec $I_0$. Soit $S$ le sous-groupe des éléments de $\mathcal{S}_7$ qui préserve $I$ : pour montrer notre résultat, il suffit de vérifier que $J$ est de la forme $s(J_0)$ pour $s\in S$. On vérifie facilement à la main qu'il existe $8$ éléments de la forme $s(J_0)$, c'est-à-dire autant que d'espaces isotropes de dimension $3$ en somme directe avec $I_0$ dans $\mathrm{r\acute{e}s}\, M$, d'où le résultat. 
\end{proof}
\begin{proof}[Démonstration du cas $n=8$] On se contente de donner les réseaux $M$ qui conviennent selon les valeurs de $I$. On laisse au lecteur la vérification du lemme \ref{transitivité I} dans ces cas (les méthodes nécessaires ayant déjà été utilisées et détaillées dans le cas $n=7$). Pour $i=1,2,3,4$, des réseaux $M$ qui vérifient les conditions du lemme \ref{unique M} sont respectivement : $\mathrm{D}_8$, $\mathrm{D}_4\oplus \mathrm{D}_4$, $\sqrt{2}\mathrm{D}_8^\sharp$ et $\sqrt{2}\mathrm{E}_8^\sharp$. 
\end{proof}
\end{proof}
\indent On déduit des lemmes \ref{unique M} et \ref{transitivité I} la proposition suivante :
\begin{prop} Soient $1\leq i\leq 3$ et $A=(\Z/2\Z)^i$. Pour $L=\mathrm{E}_7$ ou $\mathrm{E}_8$, le groupe $\mathrm{SO}(L)$ agit transitivement sur l'ensemble des $A$-voisins de $L$.
\\ \indent Pour $L=\mathrm{E}_8$, il y a deux orbites de $2,2,2,2$-voisins de $L$ pour l'action de $\mathrm{SO}(L)$, et une seule pour l'action de $\mathrm{O}(L)$.
\end{prop}
\begin{proof}
soient $L$ et $A$ comme dans l'énoncé de la proposition. D'après la proposition \ref{paramètre A-vois}, les orbites de $A$-voisins de $L$ pour l'action de $\mathrm{SO}(L)$ sont en bijection avec les orbites de couples de la forme $(X,I)$, où $X$ est un espace totalement isotrope de $L/pL$ de dimension $i$, et $I$ est un lagrangien de $\mathrm{H}(L/M)$ (avec $M$ l'image réciproque de $X^\perp$ par $L\rightarrow L/pL$) transverse à $L/M$.
\\ \indent Le théorème de Witt nous dit déjà que $\mathrm{O}(L)$ agit transitivement sur les espaces isotropes $X$ de dimension $i$ de $L/pL$. Les lemmes \ref{unique M} et \ref{transitivité I} nous disent que, une fois $X$ fixé (et donc une fois le réseau $M$ fixé), le groupe $\mathrm{O}(L)$ agit transitivement sur l'ensemble des espaces isotropes de dimension $i$ en somme directe avec $L/M$ dans $\mathrm{r\acute{e}s}\, M$. Ainsi, le groupe $\mathrm{O}(L)$ agit transitivement sur l'ensemble des $A$-voisins de $L$.
\\ \indent Pour passer aux orbites pour l'action de $\mathrm{SO}(L)$, on a deux cas à traiter :
\\ \indent - si $i\leq 3$ : alors il est facile de trouver pour un couple $(X,I)$ de la forme précédente un élément $\sigma \in \mathrm{O}(L) \setminus \mathrm{SO}(L)$ tel que $\sigma (X,I) = (X,I)$. Par exemple, pour $L=\mathrm{E}_7$, il suffit de prendre $\sigma = -\mathrm{id}$. On déduit dans ce cas que le groupe $\mathrm{SO}(L)$ agit transitivement sur l'ensemble des couples $(X,I)$ de la forme précédente, et donc sur l'ensemble des $A$-voisins.
\\ \indent - si $i=4$ (et donc $L=\mathrm{E}_8$) : il existe exactement deux orbites d'espaces isotropes de dimension $4$ dans $\mathrm{E}_8/2\mathrm{E}_8$. Ainsi, il y a exactement deux orbites de couples $(X,I)$ de la forme précédente pour l'action de $\mathrm{SO}(\mathrm{E}_8)$. Il y a donc exactement deux orbites de $2,2,2,2$-voisins de $\mathrm{E}_8$ pour l'action de $\mathrm{SO}(\mathrm{E}_8)$. En particulier, si on se donne $\sigma \in \mathrm{O}(\mathrm{E}_8) \setminus \mathrm{SO}(\mathrm{E}_8)$ et $L$ un $2,2,2,2$-voisin de $\mathrm{E}_8$, alors les réseaux $L$ et $\sigma(L)$ sont des $2,2,2,2$-voisins de $\mathrm{E}_8$ dans des orbites différentes pour l'action de $\mathrm{SO}(\mathrm{E}_8)$. 
\end{proof}
\subsection{La création de $A$-voisins dans des cas particuliers.}\label{5.2}
\indent Les calculs de la trace de $\mathrm{T}_A$ (pour $A$ de la forme $(\Z/2\Z)^i$ ou $\Z/4\Z$) sont ainsi faciles à réaliser dans les cas précédents : comme on connaît bien les orbites, il suffit de chercher un voisin correspondant à chacune des orbites (c'est-à-dire au plus deux voisins), et de prendre en considération le cardinal de chaque orbite ensuite.
\\ \indent Une autre subtilité du cas où $L=\mathrm{E}_7$, $\mathrm{E}_8$ ou $\mathrm{E}_8\oplus \mathrm{A}_1$ est que tout élément isotrope $x\in L/2L$ possède un relèvement $v\in L$ tel que $v\cdot v =4$.
\\ \indent On détaille dans la suite des cas particuliers où il est facile de construire une transformations $\sigma
\in \mathrm{O}_n(\Q)$ qui transforme $L$ en un de ses $A$-voisin. Dans les cas étudiés, chaque orbite de $A$-voisin de $L$ pour l'action de $\mathrm{SO}(L)$ possède un élément de la forme $\sigma (L)$ pour de tels éléments $\sigma$.
\begin{lem}[Création d'un $2$-voisin] Soit $x\in L$ tel que $x\cdot x =4$. On définit la symétrie $\sigma\in \mathrm{O}_n(\Q)$ par :
$$\sigma (y) = y- \dfrac{(x\cdot y)}{2}\, x.$$
\indent Alors $\sigma (L)$ est un $2$-voisin de $L$. C'est même le $2$-voisin associé à la droite isotrope $\mathrm{vect}(x)\subset L/2L$.
\end{lem}
\begin{proof}
Posons $L' =\sigma (L)$. On a alors :
$$\begin{aligned}
L\cap L' &= \left\{ y\in L \, \vert \, \sigma^{-1} (y)\in L \right\} = \left\{ y\in L \, \vert \, \sigma (y)\in L \right\} \\
&= \left\{ y\in L \, \vert \,  (x\cdot y)\in 2\Z \right\} \text{ comme $x \notin 2L$ }\\
\end{aligned}
$$
\indent On déduit ainsi facilement que $L/(L\cap L') \simeq \Z/2\Z$ et que $M=L\cap L'$ est bien l'image réciproque de $x^\perp$ par $L\rightarrow L/2L$. Si l'on se donne $y\in L$ tel que $(x\cdot y) \equiv 1 \mathrm{\ mod\ }2$, alors :
$$
	\begin{aligned}
		L' &= M+\Z \, \sigma(y) =M+\Z \, \left( y-\frac{x}{2} \right) =M+\Z \, \frac{x+2y}{2}\\
	\end{aligned}
$$
et on a bien la forme voulue, comme $(x+2y)$ a même image que $x$ dans $L/2L$ et vérifie $(x+2y)\cdot (x+2y) \equiv 0\mathrm{\ mod\ }8$.  
\end{proof}
\indent On peut étendre cette construction aux $2,\dots ,2$-voisins grâce au lemme suivant :
\begin{lem}[Création de $2,\dots ,2$ voisins] Soit $I=\{1,\dots,m\}$, et $(x_i)_{i\in I}$ une famille d'éléments de $L$ vérifiant les conditions suivantes :
$$
	\begin{array}{rl}
		(i) & (\forall i\in I ) \ x_i \cdot x_i =4 ,\\
		(ii) & (\forall i,j \in I) \ i\neq j \Rightarrow x_i\cdot x_j  =0 ,\\
		(iii) & \text{la famille $(x_i)_{i\in I}$ est libre dans $L/2L$} .\\
	\end{array}
$$
\indent Alors, en définissant les $\sigma_i$ par :
$\sigma_i (y)=y-\dfrac{(x_i\cdot y)}{2}\, x_i,$
on obtient que :
$$
\left\{
	\begin{aligned}
		& L\cap \sigma _m \circ \dots \circ \sigma_1 (L) = L\cap \sigma_m(L) \cap \dots \cap \sigma_1 (L) ,\\
		& \sigma _m \circ \dots \circ \sigma_1 (L) \text{ est un $2,\dots ,2$-voisin de }L.\\
	\end{aligned}
\right.
$$
\end{lem}
\begin{proof}
 On pose pour simplifier : $\sigma =\sigma _m \circ \dots \circ \sigma_1$, $L'=\sigma (L)$ et $M=L\cap L'$. La condition $(ii)$ permet d'écrire :
$$\sigma(y) = y-\displaystyle{\sum _{i\in I}} \dfrac{(x_i \cdot y)\, x_i}{2} .$$
\indent Ensuite, la condition $(iii)$ nous donne le premier résultat cherché. En effet, on aura :
$$\begin{aligned}
	\sigma (y) \in L &\Leftrightarrow \displaystyle{\sum _{i\in I}} \dfrac{(x_i \cdot y)\, x_i}{2} \in L \Leftrightarrow \displaystyle{\sum_{i\in I}} (x_i \cdot y)\, x_i \in 2L\\
	&\Leftrightarrow (\forall i\in I) (x_i \cdot y) \equiv 0 \mathrm{\ mod\ }2 \Leftrightarrow (\forall i\in I) \sigma_i (y) \in L .\\
\end{aligned}
$$
\indent On déduit ainsi que $L\cap L'=L\cap \sigma (L)= L\cap \sigma_m (L) \cap \dots \cap \sigma_1 (L)$, et que $L'$ est bien un $2,\dots ,2$-voisin de $L$.
\\ \indent On peut trouver la famille $(v_i)$ de la proposition \ref{ppp-voisin} qui est associée à $L'$. Pour cela, notons $\overline{x_i}$ l'image de la famille $x_i$ dans $L/2L$. On considère $u_i$ une famille d'éléments de $L$ telle que :
$$
	\begin{array}{rl}
	(i) & (\forall i)\ u_i\in \bigcap_{j\neq i} \overline{x_j}^\perp \\
	(ii) & (\forall i)\ u_i\cdot \overline{x_i} \equiv 1\mathrm{\ mod\ }2\\
	\end{array}
$$
qui existe bien comme le produit scalaire $(\ \cdot \ )$ est non dégénéré et comme la famille des $\overline{x_i}$ est libre. On peut ainsi écrire :
$$
	\begin{aligned}
		L' & = M + \sum_i \Z \, \sigma(u_i) = M+\sum_i \Z \, \sigma_i(u_i)= M+ \sum_i \Z \, \dfrac{x_i+2u_i}{2},\\
	\end{aligned}
$$
donc $L'$ est le $2,\dots ,2$-voisin associé à la famille $(x_i +2u_i)$, qui est bien un relèvement de $(\overline{x_i})$ vérifiant : $(x_i +2u_i)\cdot (x_i+2u_i) \equiv 0\mathrm{\ mod\ }8$ et $(x_i +2u_i)\cdot (x_j+2u_j) \equiv 0\mathrm{\ mod\ }4$. 
\end{proof}
\indent Lorsque $m=2$, on a le résultat plus général suivant dont la démonstration est évidente :
\begin{lem} Soient $x_1$ et $x_2$ deux éléments de $L$ tels que $(x_1\cdot x_1) = (x_2\cdot x_2) =4$. On définit comme précédemment les symétries :
$$\sigma_i (y) = y-\dfrac{(x_i \cdot y)}{2}\, x_i,$$
qui transforment chacune $L$ en un $2$-voisin.
\\ \indent Alors selon la valeur de $(x_1\cdot x_2) \mathrm{\ mod\ }2$, on déduit la nature de $L'=\sigma_1 \circ \sigma_2 (L)$ :
\\ \indent - si $(x_1\cdot x_2)\equiv 0\mathrm{\ mod\ }2$ avec $x_1 =\pm x_2$ : alors $L'=L$.
\\ \indent - si $(x_1\cdot x_2)\equiv 0\mathrm{\ mod\ }2$ avec $x_1 \neq \pm x_2$ : alors $L'$ est un $2,2$-voisin de $L$.
\\ \indent - si $(x_1\cdot x_2)\equiv 1\mathrm{\ mod\ }2$ : alors $L'$ est un $4$-voisin de $L$.
\end{lem}
\section{Résultats obtenus pour les calculs de traces d'opérateurs de Hecke.}\label{6}
\subsection{Tables des résultats obtenus.}\label{6.1}
\indent Si l'on se donne $A$ et $n$ un groupe abélien et un entier bien choisis, pour $\lambda$ un poids dominant arbitraire de $\mathrm{SO}_n(\R)$, les résultats précédents nous permettent de calculer : $\mathrm{tr}\left( \mathrm{T}_A \vert \mathcal{M}_{V_\lambda}(\mathrm{SO}_n) \right)$. On a regroupé certains de ces résultats sous forme de tables disponibles à \cite{Meg}. Les théorèmes suivants détaillent les résultats en question :
\begin{theo} \label{résultat SO7} Soient $A$ le groupe $(\Z/2\Z)^i$ ($i\leq 3$), $(\Z/p\Z)$ ($p\leq 53$ premier) ou $\Z/q\Z$ ($q\in \{4,9,25,27\}$), et $\lambda =(a,b,c)$ ($13\geq a\geq b\geq c\geq 0$) un poids dominant de $\mathrm{SO}_7(\R)$. Alors les quantités :
$$\big( 2a+5,2b+3,2c+1,\vert A \vert ^a \mathrm{tr}\left( \mathrm{T}_A \vert \mathcal{M}_{V_\lambda} (\mathrm{SO}_7) \right) \big) $$
sont données par les tables de \cite{Meg}.
\end{theo}
\begin{theo} \label{résultat SO8} Soient $A$ le groupe $(\Z/2\Z)^i$ ($i\leq 4$), $(\Z/p\Z)$ ($p\leq 13$ premier) ou $\Z/4\Z$, et $\lambda =(a,b,c,d)$ ($12\geq a\geq b\geq c\geq d\geq 0$) un poids dominant de $\mathrm{SO}_8(\R)$. Alors les quantités :
$$\big( 2a+6,2b+4,2c+2,2d,\vert A \vert ^a \mathrm{tr}\left( \mathrm{T}_A \vert \mathcal{M}_{V_\lambda} (\mathrm{SO}_8) \right) \big) $$
sont données par les tables de \cite{Meg}.
\end{theo}
\begin{theo} \label{résultat SO9} Soient $A$ le groupe $(\Z/p\Z)$ ($p\leq 7$ premier), et $\lambda =(a,b,c,d)$ ($12\geq a\geq b\geq c\geq d\geq 0$) un poids dominant de $\mathrm{SO}_9(\R)$. Alors les quantités :
$$\big( 2a+7,2b+5,2c+3,2d+1,\vert A \vert ^a \mathrm{tr}\left( \mathrm{T}_A \vert \mathcal{M}_{V_\lambda} (\mathrm{SO}_9) \right) \big) $$
sont données par les tables de \cite{Meg}.
\end{theo}
\indent Notons au passage que dans ces tables, de nombreux poids $\lambda$ ne sont pas représentés : ce sont ceux pour lesquels la dimension de l'espace $\mathcal{M}_{V_\lambda}(\mathrm{SO}_n)$ est nulle.
\subsection{Premières constatations autour de quelques exemples.}\label{6.2}
\indent Partons de l'exemple de $\mathrm{SO}_7$. D'après \cite[Table 12]{CR}, si l'on exclut la représentation triviale $W=\C$ (c'est-à-dire $W=V_\lambda$ pour $\lambda=(0,0,0)$), le ``premier'' $W=V_\lambda$ tel que $\mathcal{M}_W(\mathrm{SO}_7) \neq 0$ est obtenu pour $\lambda=(4,4,4)$, auquel cas on a d'après \cite{CR} : $\mathrm{dim}\mathcal{M}_W(\mathrm{SO}_7) =1$. Sur cet espace, la valeur propre de $p^4 \mathrm{T}_p$ est :
$$(1+p+p^2)\tau (p),$$
où $\tau$ désigne la fonction de Ramanujan. On constate que c'est bien le résultat que l'on a trouvé.
\\ \indent Toujours d'après \cite[Table 12]{CR}, le ``deuxième" $W=V_\lambda$ tel que $\mathcal{M}_W(\mathrm{SO}_7) \neq 0$ est obtenu pour $\lambda=(6,0,0)$, et on a là encore d'après \cite{CR} $\mathrm{dim}\mathcal{M}_W(\mathrm{SO}_7) =1$ dans ce cas. Sur cet espace, la valeur propre de $p^6 \mathrm{T}_p$ est :
$$ \tau_{18}(p)+(1+p+p^2+p^3)p^7 ,$$
où les coefficients $\tau_k (p)$ sont les coefficients en $q^p$ des formes modulaires normalisées pour $\mathrm{SL}_2(\Z)$ de poids $k\leq 22$ (et en particulier $\tau=\tau_{12}$). Là encore, on constate que c'est bien le résultat que l'on a trouvé.
\\ \indent De manière plus générale, donnons-nous $W$ une représentation irréductible de $\mathrm{SO}_7$, et notons $\mathcal{M}_W(\mathrm{SO}_7) = \bigoplus_{i=1}^{r_W} \C F_i$, où les $F_i$ sont des éléments de $\mathcal{M}_W(\mathrm{SO}_7)$ propres pour tous les opérateurs de Hecke $\mathrm{T}_A$ (ce qui a bien un sens car les opérateurs de Hecke sont tous co-diagonalisables d'après \cite[ch. IV, \S 4]{CL}).
\\ \indent Certaines de ces formes $F_i$ sont en fait associées à des formes automorphes pour des $\Z$-groupes ``plus petits'' que $\mathrm{SO}_7$, au sens de la théorie de l'endoscopie (nous reviendrons plus en détail sur ce point dans le chapitre suivant). C'est par exemple le cas des deux exemples précédents, où les valeurs propres cherchées se déduisent des valeurs propres de formes modulaires pour $\mathrm{SL}_2(\Z)$.
\\ \indent Dans d'autres cas, ce sont des valeurs propres de formes modulaires de Siegel pour $\mathrm{Sp}_4(\Z)$ qui interviennent. Contrairement aux cas des formes modulaires pour $\mathrm{SL}_2(\Z)$, leurs valeurs propres sont déjà difficiles à calculer, et ont fait l'objet de nombreux travaux récents (Skoruppa, Faber-van der Geer, Chenevier-Lannes).
\\ \indent Par exemple, d'après un résultat de Tsushima, si $(j,k)\in \{(6,8),(8,8),(4,10),(12,6)\}$, alors la dimension de l'espace des formes modulaires de Siegel de poids $\mathrm{Sym}^j \C^2\otimes \mathrm{det}^k$ est $1$. Suivant \cite[ch. IX]{CL}, notons $\tau_{j,k}(p)$ la valeur propre de l'opérateur $p^{\frac{j+2k-6}{2}} \mathrm{K}_p$ sur cet espace. Ces valeurs propres ont élé calculées par Faber - van der Geer pour $p\leq 37$, et par une méthode différente dans \cite{CL} pour $p\leq 113$.
\\ \indent D'autre part, d'après \cite[Table 12]{CR}, pour $\lambda = (j+2k-3,15,j+1)=(a,b,c)$ et $W=V_\lambda$, on a $\mathrm{dim}\mathcal{M}_W(\mathrm{SO}_7) =1$, et la valeur propre de $p^{a}\mathrm{T}_p$ sur cet espace est :
$$\tau_{j,k}(p) + p^{\frac{j+2k-18}{2}} \tau_{16}(p),$$
où $\tau_{16}$ a été introduit précédemment.
\\ \indent En particulier, nos calculs permettent de retrouver les valeurs de $\tau_{j,k}(p)$ pour $p\leq 53$ pour une méthode différente.
\\ \indent Enfin, l'intérêt principal de nos calcul est qu'ils permettent de calculer la trace de $\mathrm{T}_p$ sur l'espace $\mathcal{M}_W^{\mathrm{ne}}(\mathrm{SO}_7) \subset \mathcal{M}_W(\mathrm{SO}_7)$ engendré par les formes non-endoscopiques. Le ``premier" $\lambda$ pour lequel $\mathcal{M}_{V_\lambda}^{\mathrm{ne}}(\mathrm{SO}_7) \neq 0$ est $\lambda=(9,5,2)$. On a alors $\mathrm{dim}(\mathcal{M}_{V_\lambda}^{\mathrm{ne}}(\mathrm{SO}_7))=1$. Si on note $\lambda_p$ la valeur propre de $p^9\mathrm{T}_p$ sur $\mathcal{M}_{V_\lambda}^{\mathrm{ne}}(\mathrm{SO}_7)$, on a les résultats suivants :
\begin{center}\begin{tabular}{c||c|c|c|c|c|c}
$p$ & $2$ & $3$ & $5$ & $7$ & $11$ & $13$ \\
\hline
$\lambda_p$ & $0$ & $-304668$ & $874314$ & $452588136$ & $-1090903017204$ & $1624277793138$ \\ 
\end{tabular}
\end{center}

\indent Les autres valeurs propres sont données dans les tables \ref{tableau1SO7} à \ref{tableau5SO7}.
\\ \indent Ces calculs suggèrent que la représentation de Galois de dimension $6$ mise en évidence par Bergström, Faber et van der Geer est associée à cette forme automorphe pour $\mathrm{SO}_7$, ce qui répond à une question de ces auteurs (ce qui était une des motivations principales de ce travail).
\\ \indent Dans le chapitre suivant, nous rappelons plus en détail suivant \cite{CL} et \cite{CR} la contribution des formes endoscopiques dans $\mathcal{M}_W(\mathrm{SO}_n)$.
\\ \indent Au final, nous ne donnerons des tables (équivalentes aux tables des théorèmes \ref{résultat SO7}, \ref{résultat SO8} et \ref{résultat SO9}) que pour les contributions non-endoscopiques.
\\ \indent Suivant Arthur, ces contributions s'expriment mieux en terme de certaines représentations automorphes pour les groupes linéaires $\mathrm{GL}_m$. Nos calculs permettent de donner des informations sur les ``paramètres de Satake" de ces représentations automorphes.
\section{La paramétrisation de Langlands-Satake}\label{7}
\subsection{Les formules de Gross et la paramétrisation de Satake}\label{7.1}
\subsubsection{La formule de Gross dans le cas général}
\indent Nous renvoyons à Borel \cite{Bor77} et à Springer \cite{Spr79} pour les notions de groupe dual et de données radicielles des groupes réductifs. Nous renvoyons aussi aux études de l'isomorphisme de Satake faites par Gross \cite{Gr} et par Satake \cite{Sat63}. Enfin, on suivra les notations utilisées dans \cite[ch. 3]{CR} et \cite[ch. VI]{CL}.
\\ \indent On considère $G$ un $\Z_p$-groupe semi-simple, et $\widehat{G}$ son dual de Langlands (qui est un groupe semi-simple sur $\C$). On note $\Psi (G) = (X,\Phi,\Delta,X^\vee,\Phi^\vee,\Delta ^\vee)$ sa donnée radicielle basée, $\Psi (\widehat{G}) = (X^\vee,\Phi^\vee,\Delta^\vee,X,\Phi,\Delta)$ la donnée duale, et on note $X_+\subset X^\vee$ l'ensemble des poids dominants de $G$. On rappelle aussi que, si $T$ désigne un tore maximal de $G$, et $B$ un sous-groupe de Borel contenant $T$, on leur associe la donnée radicielle basée : $\Psi (G,T,B)=(X^*(T),\Phi(G,T),\Delta(G,T,B),X_*(T),\Phi^\vee (G,T),\Delta^\vee (G,T,B))$, où les groupes $X^*(T)=\mathrm{Hom}(T,\mathbb{G}_m)$ et  $X_*(T)=\mathrm{Hom}(\mathbb{G}_m,T)$ sont repectivement les groupes des caractères et des co-caractères de $T$, $\Phi(G,T)$ (respectivement $\Phi^\vee (G,T)$) désigne l'ensemble des racines (respectivement des coracines) de $G$ relativement à $T$, $\Delta (G,T,B)$ est la base de $\Phi(G,T)$ associée au système positif de $\Phi(G,T)$ intervenant dans $\mathrm{Lie}(B)$, et $\Delta^\vee (G,T,B)$ est la base duale associée.
\\ \indent On rappelle que l'isomorphisme de Satake, introduit dans \cite{Sat63} et revisité par Langlands dans \cite[\S 2]{Lan}, est un isomorphisme d'anneaux canonique :
$$ \mathrm{Sat}:\mathrm{H}_p (G)\otimes \Z [p^{-\frac{1}{2}}] \overset{\sim}{\rightarrow} \mathrm{Rep}(\widehat{G}) \otimes \Z [p^{-\frac{1}{2}}],$$
où $\mathrm{Rep}(\widehat{G})$ désigne l'anneau de Grothendieck des représentations polynomiales de dimension finie de $\widehat{G}$.
\\ \indent Si l'on désigne par $\widehat{G}(\C)_{\mathrm{ss}}$ l'ensemble des classes de conjugaison d'éléments semi-simples de $\widehat{G}(C)$, alors pour $c\in \widehat{G}(\C)_{\mathrm{ss}}$ on possède une application $V\mapsto \mathrm{trace}(c\vert V)$ qui associe à une $\C$-représentation de dimension finie $V$ de $\widehat{G}$ la trace de $c$ dans $V$. Cette application s'étend en un homomorphisme d'anneaux $\mathrm{tr}(c) : \mathrm{Rep}(\widehat{G}) \rightarrow \C$. D'après un résultat de Chevalley, l'application $\mathrm{tr}:\widehat{G}(\C)_{\mathrm{ss}} \rightarrow \mathrm{Hom}_{\mathrm{anneaux}}(\mathrm{Rep}(\widehat{G}),\C)$ est une bijection. On en déduit la proposition suivante :
\begin{prop} \label{tr o Sat} L'application $c\mapsto \mathrm{tr}(c) \circ \mathrm{Sat}$ définit une bijection :
$$\widehat{G}(\C)_{\mathrm{ss}} \overset{\sim}{\rightarrow} \mathrm{Hom}_{\mathrm{anneaux}} (\mathrm{H}_p (G),\C).$$
\end{prop}
\indent La théorie de Cartan-Weyl pour les représentations de plus haut poids nous donne une première $\Z$-base naturelle de $\mathrm{Rep}(\widehat{G})$ indexée par $X_+$, à savoir les représentations irréductibles de la forme $V_\lambda$.
\\ \indent Pour $\lambda\in X^+$, on note $\left[ V_\lambda \right]$ la classe de $V_\lambda$ dans $\mathrm{Rep}(\widehat{G})$. De plus, si l'on se donne $\widehat{T}$ un tore maximal de $\widehat{G}$, et $\widehat{B}$ un sous-groupe de Borel de $\widehat{G}$ contenant $\widehat{T}$ (de sorte que $\Psi(\widehat{G})$ s'identifie à $\Psi(\widehat{G},\widehat{T},\widehat{B})$), et si $\mu\in X$, on désigne par $V_\lambda
(\mu)\subset V_\lambda$ le sous-espace propre de $V_\lambda$ pour $\mu$ sous l'action de $\widehat{T}$.
\\ \indent Comme $G$ est réductif sur $\Z_p$, il existe une décomposition de Cartan de la forme :
$$G(\Q_p) = \coprod_{\lambda \in X_+} G(\Z_p) \lambda (p) G(\Z_p).$$
\indent Pour $\lambda \in X$, on note $c_\lambda \in \mathrm{H}_p(G)$ la fonction caractéristique de la double classe $G(\Z_p) \lambda (p) G(\Z_p)$. Les $c_\lambda$ ainsi définis forment une $\Z_p$-base de $\mathrm{H}_p(G)$ pour $\lambda$ décrivant $X^+$, avec des relations de la forme :
$$ (\forall \lambda,\mu \in X_+) (\exists (n_{\lambda,\mu,\nu}) \in \Z) \ c_\lambda \cdot c_\mu = c_{\lambda+\mu} + \sum_{\nu< \lambda+\mu} n_{\lambda,\mu,\nu} c_\nu .$$
\indent Le lien entre ces deux $\Z_p$-bases de $\mathrm{H}_p(G)$ et $\mathrm{Rep}(\widehat{G})$ est donné grâce à l'isomorphisme de Satake comme suit :
\begin{prop}[La formule de Gross] Soit $G$ un $\Z_p$-groupe semi simple déployé et $X_+$ l'ensemble ordonné des poids dominants de $\widehat{G}$. Soit $\lambda \in X_+$. On note $\rho$ la demi somme des racines positives de $G$. Alors on dispose d'une identité de la forme :
$$p^{\left< \rho ,\lambda \right> } \left[ V_\lambda \right] = \mathrm{Sat}(c_\lambda) + \sum_{\{\mu \in X_+, \mu<\lambda \} } d_\lambda(\mu) \mathrm{Sat}(c_\mu)$$
pour certains entiers $d_\lambda (\mu)$ dépendant de $p$, qu'on détaille plus loin.
\\ \indent En particulier, on a les résultats suivants :
\\ \begin{tabular}{rp{14cm}}
$(i)$ &  Si $\lambda$ est un élément minimal, alors : $p^{\left< \rho,\lambda \right> } \left[ V_\lambda \right] = \mathrm{Sat}(c_\lambda)$.\\
$(ii)$ & Si $\mu \in X_+$ tel que $\mathrm{dim}\left( V_\lambda (\mu)\right) =1$, alors $d_\lambda(\mu) =1$\\
$(iii)$ & Si $V_\lambda=\mathrm{Lie}(G)$ est la représentation adjointe, alors : $d_\lambda (0) = \sum_i p^{m_i-1}$, où les $m_i$ sont les exposants de $G$.\\
\end{tabular}
\end{prop}
\begin{proof}
le cas où $G$ est adjoint est traité par Gross dans \cite{Gr}. Le cas où $G$ est semi-simple est traité dans \cite[ch. VI, lemme 2.7]{CL}. 
\end{proof}
\indent Les coefficients $d_\lambda (\mu)$ qui interviennent dans la proposition précédente ont été calculés par Lusztig et Kato. On a le résultat suivant :
\begin{prop} \label{Gross} On reprend les mêmes notations qu'à la proposition précédente. On note $\rho$ la demi somme des racines positives de $G$, et $\rho^\vee$ celles des coracines positives de $G$. Pour $\mu\in X_+$, on définit le polynôme $\widehat{P}$ comme :
$$\widehat{P}(\mu) =\sum _{\substack{ \mu=\sum n(\alpha^\vee) \alpha^\vee \\ \alpha \in \Phi^+}} p^{-\sum n(\alpha ^\vee)},$$
qui est un polynôme en $p^{-1}$ qui considère le nombre d'expressions de $\mu$ comme somme à coefficients positifs de coracines de $G$. Si $\mu$ ne peut pas s'écrire sous cette forme (ce qui est le cas si $\mu \in -\Phi^{\vee +}\setminus \{0\}$ par exemple), alors $\widehat{P}(\mu)=0$. Comme on considère la somme vide dans les possibilités, alors $\widehat{P}(0) =1$.
\\ \indent Les coefficients $d_\lambda (\mu)$ sont alors donnés par la formule suiante :
$$d_\lambda (\mu) = p^{\left< \rho ,\lambda -\mu  \right>} \sum_{\sigma \in W} \varepsilon (\sigma) \cdot \widehat{P}\left( \sigma (\lambda +\rho^\vee ) - (\mu + \rho^\vee) \right) ,$$
où $\varepsilon (\sigma)$ est la signature sur le groupe de Weyl $W$ de $\widehat{G}$.
\end{prop}
\begin{proof}
voir Kato dans \cite{Kat}, qui reprend des résultats de Lusztig \cite{Lus}, Kostant \cite{FH} et Brylinski \cite{Bry}.
\end{proof}
\subsubsection{Le groupe spécial orthogonal de dimension paire}
\indent Soit $r\geq 2$ un entier, $U=\left( \Z_p \right)^r$, et $V=\mathrm{H}(U)$ le module hyperbolique sur $U$. On pose $G=\mathrm{SO}_V$ le sous-groupes des automorphismes de $V$ préservant la forme quadratique définie sur $V$ et de déterminant $1$. On désigne par $(e_i)_{1\leq i\leq r}$ une $\Z_p$-base de $U$, et $(e_i^*)$ sa base duale. Alors on peut définir une donnée radicielle basée de $G$ en se donnant :
\\ \indent - pour le tore $T$ : le sous-groupe de $G$ des éléments préservant les droites engendrées par les $e_i$ et celles engendrées par les $e_j^*$.
\\ \indent - pour le groupe de Borel $B$ : le sous-$\Z_p$-groupe de $G$ des éléments préservant le drapeau complet de $U$ associé à $\{e_1\},\{e_1,e_2\},\dots,\{e_1,\dots,e_r\}$.
\\ \indent On pose $\varepsilon_i\in X^*(T)$ le caractère de $T$ agissant sur la droite engendrée par $e_i$ (ce qui veut dire que $T$ agit sur la droite engendrée par $e_j^*$ par le caractère $-\varepsilon_j$). On pose enfin $(\varepsilon^*_j) \in X_* (T)$ la famille duale de $(\varepsilon_i) \in X^* (T)$. Alors :
\\ \indent - les ensembles $X^*(T)$ et $X_*(T)$ s'identifient respectivement à $\bigoplus_{1\leq i\leq r} \varepsilon_i \Z$ et à $\bigoplus_{1\leq i\leq r} \varepsilon_i^* \Z$ (et en particulier sont isomorphes).
\\ \indent - l'ensemble $\Phi(G,T)$ est constitué des $\pm \varepsilon_i \pm \varepsilon_j$, avec la dualité : $(\pm \varepsilon_i \pm \varepsilon_j)^\vee = \pm \varepsilon_i^* \pm \varepsilon_j^*$ (pour $1\leq i < j \leq r$).
\\ \indent - l'ensemble $\Delta (G,T,B)$ est constitué des $\varepsilon_i - \varepsilon_{i+1}$ (pour $1\leq i\leq r-1$) et de $\varepsilon_{r-1} + \varepsilon_r$.
\\ \indent - l'ensemble $\Phi^+$ est constitué des $\varepsilon_i \pm \varepsilon_j$ (pour $1\leq i < j \leq r$).
\\ \indent - les éléments $\rho$ et $\rho^\vee$ sont donnés par : $\rho=(r-1) \varepsilon_1 +(r-2)\varepsilon_2+\dots +\varepsilon_{r-1}$ et $\rho^\vee = (r-1) \varepsilon_1^* +(r-2)\varepsilon_2^*+\dots +\varepsilon_{r-1}^*$.
\\ \indent - le dual de Langlands de $G$ est donné par : $\widehat{G} = \mathrm{SO}_{2r}(\C)$.
\\ \indent - le groupe de Weyl $W$ de $\widehat{G}$ s'identifie à : $\mathcal{S}_r \ltimes \left( \{ \pm 1 \}^r\right) ^0$.
\\ \indent En particulier, l'ensemble des copoids dominants de $G$ est donné par :
$$X_+ = \left\{ \sum_i m_i  \varepsilon_i^* \bigg\vert m_1\geq m_2\geq \dots \geq m_{r-1} \geq \vert m_r \vert \right\} .$$
\indent Pour $1\leq i\leq r$, on définit les éléments $\lambda_i \in X_+$ par : $ \lambda_i = \varepsilon_1^*+\dots +\varepsilon_i^*$.
\\ \indent Enfin, on pose $\tau$ l'automorphisme de $\Psi (G)$ qui fixe les $\varepsilon_i$ pour $1\leq i\leq r-1$, et qui envoie $\varepsilon_r$ sur $-\varepsilon_r$ (qui est une involution sur $X_+$).
\\ \indent Les opérateurs de Hecke de $\mathrm{H}(G)$ sont donnés par la proposition suivante :
\begin{prop} Soient $\lambda =\sum_i m_i \varepsilon_i^* \in X_+$ et $A_\lambda = \prod_{i=1}^r \left( \Z/p^{\vert m_i \vert } \Z \right)$.
\\ \indent Alors l'opérateur $\mathrm{T}_{A_\lambda}\in \mathrm{H}(G)$ est donné par :
$$\mathrm{T}_{A_\lambda} = \sum_{\mu\in \{\lambda,\tau(\lambda) \} } c_\mu .$$
\indent En particulier, on a les égalités :
\begin{center}
\begin{tabular}{rl}
$(i)$ & $\left( \forall\ 1\leq i\leq r-1 \right) c_{\lambda_i}=\mathrm{T}_{\left( \Z/p\Z\right)^i} $ \\
$(ii)$ & $c_{\lambda_r} + c_{\tau (\lambda_r)} = \mathrm{T}_{\left( \Z/p\Z\right)^r}$ \\
$(iii)$ & $\left( \forall\ m\geq 0 \right) c_{m\lambda_1}=\mathrm{T}_{p^m} $ \\
\end{tabular}
\end{center}
\end{prop}
\begin{proof}
voir \cite[ch. VI, scholie 2.9]{CL}. 
\end{proof}
\begin{prop}[Le poids de la représentation standard]\label{standard1} On considère $V_{\mathrm{St}}$ la représentation standard de $\widehat{G}=\mathrm{SO}_{2r}(\C)$. C'est la représentation de plus haut poids $\lambda_1$.
\\ \indent De plus, pour $2\leq i \leq r-1$, la représentation $\Lambda^i V_{\mathrm{St}}$ est irréductible et correspond à la représentation de plus haut poids $\lambda_i$.
\\ \indent Enfin, la représentation $\Lambda^r V_{\mathrm{St}}$ correspond à la somme des représentations $V_{\lambda_r}$ et $V_{\tau (\lambda_r)}$.
\end{prop}
\begin{proof}
découle de \cite[théorème 19.2]{FH}. 
\end{proof}
\indent On souhaite exprimer, grâce aux formules de Gross, les représentation $V_{\mathrm{St}}$ et $\Lambda^i V_{\mathrm{St}}$ en fonction des opérateurs de Hecke $\mathrm{T}_A$. Pour $r=4$ (c'est-à-dire pour $\widehat{G}=\mathrm{SO}_8 (\C)$), on a la proposition suivante :
\begin{prop} \label{GrossSO8} On reprend les mêmes notations, avec $r=4$. On considère les éléments $\mathrm{T}_A\in \mathrm{H}(G)$. On a alors les égalités suivantes :
\begin{center}
\begin{tabular}{rl}
$(i)$ & $p^3 \left[ V_{\mathrm{St}}\right] = \mathrm{Sat}(\mathrm{T}_p)$\\
$(ii)$ & $p^5 \left[ \Lambda^2 V_{\mathrm{St}}\right] = \mathrm{Sat}(\mathrm{T}_{p,p})+\left( p^4+2\, p^2+1\right)$\\
$(iii)$ & $p^6 \left[ \Lambda^3 V_{\mathrm{St}}\right] = \mathrm{Sat}(\mathrm{T}_{p,p,p})+\left( p^2+p+1\right)\, \mathrm{Sat}(\mathrm{T}_p)$\\
$(iv)$ & $p^6 \left[ \Lambda^4 V_{\mathrm{St}}\right] = \mathrm{Sat}(\mathrm{T}_{p,p,p,p})+2\, \mathrm{Sat}(\mathrm{T}_{p,p}) + 2\, \left(p^4+p^2+1 \right) $\\
\end{tabular}
\end{center}
\end{prop}
\begin{proof}
On utilise la proposition \ref{standard1} et les formules de Gross. On a :
\\ \indent Pour le $(i)$ : il n'y a pas de poids $\mu <\lambda_1$, et ainsi on a les égalités :
$$
\begin{aligned}
	p^3\left[ V_{\mathrm{St}}\right] &= p^{\left< \rho, \lambda_1 \right>} \left[ V_{\lambda_1} \right] = \mathrm{Sat}(c_{\lambda_1})+\sum_{ \{\mu\in X_+,\mu <\lambda_1\} } d_{\lambda_1}(\mu) \mathrm{Sat}(c_\mu) \\
	& =\mathrm{Sat}(c_{\lambda_1}) =\mathrm{Sat}(\mathrm{T}_p) \\
\end{aligned}
$$
\indent Pour le $(ii)$ : il y a un seul poids $\mu <\lambda_2$, à savoir $\mu =0$. On déduit ainsi :
$$
\begin{aligned}
	p^5\left[ \Lambda^2 V_{\mathrm{St}}\right] &= p^{\left< \rho, \lambda_2 \right>} \left[ V_{\lambda_2} \right] = \mathrm{Sat}(c_{\lambda_2})+\sum_{ \{\mu\in X_+,\mu <\lambda_2\} } d_{\lambda_2}(\mu) \mathrm{Sat}(c_\mu) \\
	& =\mathrm{Sat}(c_{\lambda_2}) +d_{\lambda_2}(0)=\mathrm{Sat}(\mathrm{T}_{p,p})+\left( p^4+2\, p^2+1\right) \\
\end{aligned}
$$
\indent Pour le $(iii)$ : il y a un seul poids $\mu <\lambda_3$, à savoir $\mu =\lambda_1$. On déduit ainsi :
$$
\begin{aligned}
	p^6\left[ \Lambda^3 V_{\mathrm{St}}\right] &= p^{\left< \rho, \lambda_3 \right>} \left[ V_{\lambda_3} \right] = \mathrm{Sat}(c_{\lambda_3})+\sum_{ \{\mu\in X_+,\mu <\lambda_3\} } d_{\lambda_3}(\mu) \mathrm{Sat}(c_\mu) \\
	& =\mathrm{Sat}(c_{\lambda_3}) +d_{\lambda_3}(\lambda_1)\, \mathrm{Sat}(c_{\lambda_1})=\mathrm{Sat}(\mathrm{T}_{p,p,p})+\left( p^2+p+1\right)\, \mathrm{Sat}(\mathrm{T}_p) \\
\end{aligned}
$$
\indent Pour le $(iv)$ : il y a deux poids $\mu$ inférieurs à $\lambda_4$ ou à $\tau (\lambda_4)$, à savoir $\mu =\lambda_2$ ou $\mu =0$ dans les deux cas. On déduit ainsi :
$$
\begin{aligned}
	p^6 \left[ \Lambda^4 V_{\mathrm{St}}\right] &= p^6 \, \left(\left[V_{\lambda_4}\right] +\left[V_{\tau (\lambda_4)}\right] \right) = p^{\left< \rho, \lambda_4 \right>} \left[ V_{\lambda_4} \right]+ p^{\left< \rho, \tau(\lambda_4) \right>} \left[ V_{\tau (\lambda_4)} \right] \\
	&= \mathrm{Sat}(c_{\lambda_4})+\sum_{ \{\mu\in X_+,\mu <\lambda_4\} } d_{\lambda_4}(\mu) \mathrm{Sat}(c_\mu) \\
	& \indent + \mathrm{Sat}(c_{\tau(\lambda_4)})+\sum_{ \{\mu\in X_+,\mu <\tau(\lambda_4)\} } d_{\tau(\lambda_4)}(\mu) \mathrm{Sat}(c_\mu) \\
	& =\mathrm{Sat}(c_{\lambda_4})+\mathrm{Sat}(c_{\tau (\lambda_4)}) +\left( d_{\lambda_4}(\lambda_2) + d_{\tau (\lambda_4)}(\lambda_2) \right) \, \mathrm{Sat}(c_{\lambda_2}) \\
	&\indent  +\left( d_{\lambda_4}(0) + d_{\tau (\lambda_4)}(0) \right)\\
	&=\mathrm{Sat}(\mathrm{T}_{p,p,p,p})+2\, \mathrm{Sat}(\mathrm{T}_{p,p}) + 2\, \left(p^4+p^2+1 \right) \\
\end{aligned}
$$
\indent Les coefficients $d_\lambda (\mu)$ ont été calculés à l'ordinateur grâce à la proposition \ref{Gross}.  
\end{proof}
\subsubsection{Le groupe spécial orthogonal en dimension impaire}
\indent Soit $r\geq 1$ un entier, $U=\left( \Z_p \right)^r$, $\mathrm{H}(U)$ le module hyperbolique sur $U$ et $V=\mathrm{H}(U)\oplus \Z_p$ (où on a muni $\Z_p$ de la forme quadratique $x\mapsto x^2$). On pose $G=\mathrm{SO}_{V}$ le sous-groupes des automorphismes de $V$ préservant la forme quadratique définie sur $V$ et de déterminant $1$. On définit comme dans le cas du groupe spécial orthogonal de dimension paire les éléments $T$, $B$, les $\varepsilon_i$ et les $\varepsilon_i^*$. Les changements sur la donnée radicielle basée se font alors comme suit :
\\ \indent - l'ensemble $\Phi(G,T)$ est constitué des $\pm \varepsilon_i \pm \varepsilon_j$ ($1\leq i < j \leq r$) et des $\pm  \varepsilon_i$ ($1\leq i \leq r$), avec la dualité : $(\pm \varepsilon_i \pm \varepsilon_j)^\vee = \pm \varepsilon_i^* \pm \varepsilon_j^*$ et $(\pm \varepsilon_i)^\vee = \pm 2\cdot  \varepsilon_i^*$.
\\ \indent - l'ensemble $\Delta (G,T,B)$ est constitué des $\varepsilon_i - \varepsilon_{i+1}$ (pour $1\leq i\leq r-1$) et de $ \varepsilon_r$.
\\ \indent - l'ensemble $\Phi^+$ est constitué des $\varepsilon_i \pm \varepsilon_j$ (pour $1\leq i < j \leq r$) et des $ \varepsilon_i$ (pour $1\leq i\leq r$).
\\ \indent - les éléments $\rho$ et $\rho^\vee$ sont donnés par : $\rho=\frac{2r-1}{2} \varepsilon_1 +\frac{2r-3}{2}\varepsilon_2+\dots +\frac{1}{2}\varepsilon_{r}$ et $\rho^\vee = r \varepsilon_1^* +(r-1)\varepsilon_2^*+\dots +\varepsilon_{r}^*$.
\\ \indent - le dual de Langlands de $G$ est donné par : $\widehat{G} = \mathrm{Sp}_{2r}(\C)$.
\\ \indent - le groupe de Weyl $W$ de $\widehat{G}$ s'identifie à : $\mathcal{S}_r \ltimes \{ \pm 1 \}^r$.
\\ \indent En particulier, l'ensemble des copoids dominants de $G$ est donné par :
$$X_+ = \left\{ \sum_i m_i  \varepsilon_i^* \bigg\vert m_1\geq m_2\geq \dots \geq m_r \geq 0\right\} .$$
\indent Pour $1\leq i\leq r$, on définit les éléments $\lambda_i \in X_+$ par : $ \lambda_i = \varepsilon_1^*+\dots +\varepsilon_i^*$.
\\ \indent Les opérateurs de Hecke de $\mathrm{H}(G)$ sont donnés par la proposition suivante :
\begin{prop} Soient $\lambda =\sum_i m_i \varepsilon_i^* \in X_+$ et $A_\lambda = \prod_{i=1}^r \left( \Z/p^{ m_i } \Z \right)$.
\\ \indent Alors l'opérateur $\mathrm{T}_{A_\lambda} \in \mathrm{H}(G)$ est donné par :
$$\mathrm{T}_{A_\lambda} =c_\lambda .$$
\indent En particulier, on a les égalités :
\begin{center}
\begin{tabular}{rl}
$(i)$ & $\left( \forall\ 1\leq i\leq r \right) c_{\lambda_i}=\mathrm{T}_{\left( \Z/p\Z\right)^i} $ \\
$(ii)$ & $\left( \forall\ m\geq 0 \right) c_{m\lambda_1}=\mathrm{T}_{p^m} $ \\
\end{tabular}
\end{center}
\end{prop}
\begin{proof}
voir \cite[ch. VI, scholie 2.9]{CL}.  
\end{proof}
\begin{prop}[Le poids de la représentation standard]\label{standard3} On considère $V_{\mathrm{St}}$ la représentation standard de $\widehat{G}=\mathrm{Sp}_{2r}(\C)$. C'est la représentation de plus haut poids $\lambda_1$.
\\ \indent De plus, pour $2\leq i \leq r$, la représentation $\Lambda^i V_{\mathrm{St}}$ n'est pas irréductible, son plus haut poids est $\lambda_i$, et s'écrit sous la forme : $\Lambda^i V_\mathrm{St} = V_{\lambda_i} \oplus V_{\lambda_{i-2}} \oplus \dots \oplus V_{\lambda_{i-2\cdot [i/2]}}$.
\end{prop}
\begin{proof}
découle de \cite[théorème 17.5]{FH} (en faisant une récurrence sur $i$).  
\end{proof}
\indent On souhaite exprimer, grâce aux formules de Gross, les représentation $V_{\mathrm{St}}$ et $\Lambda^i V_{\mathrm{St}}$ en fonction des opérateurs de Hecke $\mathrm{T}_A$. Pour $r=3$ (c'est-à-dire pour $\widehat{G}=\mathrm{Sp}_6 (\C)$), on a la proposition suivante :
\begin{prop} \label{GrossSO7} On reprend les mêmes notations, avec $r=3$. On considère les éléments $\mathrm{T}_A\in \mathrm{H}(G)$. On a alors les égalités suivantes :
\begin{center}
\begin{tabular}{rl}
$(i)$ & $p^{5/2} \left[ V_{\mathrm{St}}\right] = \mathrm{Sat}(\mathrm{T}_p) $\\
$(ii)$ & $p^4 \left[ \Lambda^2 V_{\mathrm{St}}\right] = \mathrm{Sat}(\mathrm{T}_{p,p})+  \left( p^4+p^2+1\right)$\\
$(iii)$ & $p^{9/2} \left[ \Lambda^3 V_{\mathrm{St}}\right] = \mathrm{Sat}(\mathrm{T}_{p,p,p})+ \left( p^2+1\right) \cdot \mathrm{Sat}(\mathrm{T}_p)$\\
\end{tabular}
\end{center}
\end{prop}
\begin{proof}
L'expression des $\left[ V_{\lambda_i} \right]$ en fonction des $\mathrm{Sat}(\mathrm{T}_A)$ se fait comme précédemment. Reste donc à exprimer les $\left[ \Lambda^i V_\mathrm{St} \right]$ en fonction des $\left[ V_{\lambda_i} \right]$.
\\ \indent - $(i)$ : on a $\left[ V_{\mathrm{St}} \right] = \left[ V_{\lambda_1} \right]$ et il n'y a pas de poids $\mu <\lambda_1$.
\\ \indent - $(ii)$ : on a $\left[ \Lambda^2 V_{\mathrm{St}} \right] = \left[ V_{\lambda_2} \right]+1$ et il y a un seul poids $\mu <\lambda_2$, à savoir $0$.
\\ \indent - $(iii)$ : on a $\left[ \Lambda^3 V_{\mathrm{St}} \right] = \left[ V_{\lambda_3} \right]+\left[ V_{\lambda_1} \right]$ et il y a un seul poids $\mu <\lambda_3$, à savoir $\lambda_1$.  
\end{proof}
\indent Pour $r=4$ (c'est-à-dire pour $\widehat{G}=\mathrm{Sp}_8 (\C)$), on a la proposition suivante :
\begin{prop} \label{GrossSO9} On reprend les mêmes notations, avec $r=4$. On considère les éléments $\mathrm{T}_A\in \mathrm{H}(G)$. On a alors les égalités suivantes :
\begin{center}
\begin{tabular}{rl}
$(i)$ & $p^{7/2} \left[ V_{\mathrm{St}}\right] = \mathrm{Sat}(\mathrm{T}_p) $\\
$(ii)$ & $p^6 \left[ \Lambda^2 V_{\mathrm{St}}\right] = \mathrm{Sat}(\mathrm{T}_{p,p})+ \left( p^6+p^4+p^2+1\right)$\\
$(iii)$ & $p^{15/2} \left[ \Lambda^3 V_{\mathrm{St}}\right] = \mathrm{Sat}(\mathrm{T}_{p,p,p})+\left( p^4+ p^2+1 \right)\cdot \mathrm{Sat}(\mathrm{T}_p)$\\
$(iv)$ & $p^{8} \left[ \Lambda^4 V_{\mathrm{St}}\right] = \mathrm{Sat}(\mathrm{T}_{p,p,p,p})+ \left( p^2+1 \right) \cdot \mathrm{Sat}(\mathrm{T}_{p,p})+\left( p^8+p^6+2\cdot  p^4+p^2+1 \right)$\\
\end{tabular}
\end{center}
\end{prop}
\begin{proof}
La démonstration se fait comme précédemment. Il suffit de faire les constatations suivantes :
\\ \indent - $(i)$ : on a $\left[ V_{\mathrm{St}} \right] = \left[ V_{\lambda_1} \right]$ et il n'y a pas de poids $\mu <\lambda_1$.
\\ \indent - $(ii)$ : on a $\left[ \Lambda^2 V_{\mathrm{St}} \right] = \left[ V_{\lambda_2} \right]+1$ et il y a un seul poids $\mu <\lambda_2$, à savoir $0$.
\\ \indent - $(iii)$ : on a $\left[ \Lambda^3 V_{\mathrm{St}} \right] = \left[ V_{\lambda_3} \right]+\left[ V_{\lambda_1} \right]$ et il y a un seul poids $\mu <\lambda_3$, à savoir $\lambda_1$.
\\ \indent - $(iv)$ : on a $\left[ \Lambda^4 V_{\mathrm{St}} \right] = \left[ V_{\lambda_4} \right]+\left[ V_{\lambda_2} \right] +1 $ et il y a deux poids $\mu <\lambda_4$, à savoir $\lambda_2$ et $0$. 
\end{proof}
\subsection{Les formes automorphes et la paramétrisation de Langlands}\label{7.2}
\indent Soient $G$ un $\C$-groupe semi-simple et $\mathfrak{g}$ sa $\C$-algèbre de Lie. On note $\mathrm{U}(\mathfrak{g})$ son algèbre enveloppante, $\mathrm{Z}(\mathrm{U}(\mathfrak{g}))$ le centre de $\mathrm{U}(\mathfrak{g})$. Soient enfin $\widehat{G}$ le dual de Langlands de $G$, et $\widehat{\mathfrak{g}}$ sa $\C$-algèbre de Lie.
\\ \indent Suivant Harish-Chandra et Langlands, on rappelle dans la proposition suivante comment voir les caractères centraux de $\mathrm{U}(\mathfrak{g})$-modules comme des classes de conjugaison semi-simples dans $\widehat{\mathfrak{g}}$, grâce à l'isomorphisme de Harish-Chandra :
\begin{prop}[L'isomorphisme de Harish-Chandra] Avec les mêmes notations, on pose $\mathrm{Pol} (\widehat{\mathfrak{g}})$ la $\C$-algèbre des fonctions polynomiales sur $\widehat{\mathfrak{g}}$. Elle est munie d'une action naturelle de $\widehat{G}(\C)$, et on note $\mathrm{Pol}(\widehat{\mathfrak{g}})^{\widehat{G}}$ l'algèbre des invariants. L'isomorphisme de Harish-Chandra est un isomorphisme canonique :
$$\mathrm{HC}:\mathrm{Z}(\mathrm{U}(\mathfrak{g})) \overset{\sim}{\longrightarrow} \mathrm{Pol}(\widehat{\mathfrak{g}})^{\widehat{G}}.$$
\indent L'isomorphisme de Harish-Chandra induit une bijection canonique :
$$\mathrm{Hom}_{\C -\mathrm{cusp}} (\mathrm{Z}(\mathrm{U}(\mathfrak{g})), \C ) \overset{\sim}{\longrightarrow} \widehat{\mathfrak{g}}_{\mathrm{ss}},$$
où $\widehat{\mathfrak{g}}_{\mathrm{ss}}$ désigne l'ensemble des classes de conjugaisons d'éléments semi-simples de $\widehat{\mathfrak{g}}$.
\end{prop}
\begin{proof}
voir \cite[\S 2]{Lan}. 
\end{proof}
\noindent{Exemple :} Soient $\lambda \in X_+$ un poids dominant de $G$, et $V_\lambda$ la $\C$-représentation irréductible de $G$ de plus haut poids $\lambda$. Cette représentation munit $V_\lambda$ d'une structure de $\mathrm{U}(\mathfrak{g})$-module. Ce module est simple et admet un caractère central. La classe de conjugaison dans $\widehat{\mathfrak{g}}_{\mathrm{ss}}$ qui correspond à ce caractère est la classe de conjugaison de $\lambda + \rho$ (où $\rho$ est la demi-somme des racines positives de $G$).
\\ \\ \indent On définit comme suit les différents ensembles de représentations automorphes :
\begin{defi} Soit $G$ un $\Z$-groupe semi-simple. On définit une représentation du couple $(G(\R), \mathrm{H}(G))$ comme la donnée d'un espace de Hilbert muni d'une représentation unitaire de $G(\R)$, et d'une structure de module à droite sur $\mathrm{H}(G)$ commutant à l'action de $G(\R)$.
\\ \indent On note alors $\Pi (G)$ l'ensemble des classes d'isomorphisme de représentations du couple $(G(\R), \mathrm{H}(G))$ de la forme $\pi = \pi_\infty \otimes \pi_V$, où $\pi_\infty$ est une représentation unitaire irréductible de $G(\R)$, et $\pi_V$ est une représentation irréductible complexe de dimension finie de $\mathrm{H}(G)^\mathrm{opp}$ (donc de dimension $1$).
\\ \indent On note aussi $\Pi_{\mathrm{disc}}(G)$ l'ensemble des représentations automorphes $\pi$ discrètes de $G$ telles que $\pi_p^{G(\Z_p)} \neq 0$ (le sous-espace des invariants de $\pi_p$ sous l'action de $G(\Z_p)$) pour tout $p$ premier. L'application $\pi\in\Pi_{\mathrm{disc}}(G) \mapsto \pi^{G(\widehat{\Z})}$ (le sous-espace des invariants de $\pi$ sous l'action de $G(\widehat{\Z})$) réalise $\Pi_{\mathrm{disc}}(G)$ comme un sous-ensemble de $\Pi (G)$. On note enfin $\Pi_{\mathrm{cusp}}(G)$ le sous-ensemble de $\Pi_{\mathrm{disc}}(G)$ des représentations automorphes cuspidales.
\end{defi}
\indent Si l'on se donne $H$ un $\C$-groupe, on peut lui associer un ensemble $\mathcal{X}(H)$ de collections de classes de conjugaison semi-simples comme suit :
\begin{defi} Soit $H$ un $\C$-groupe semi-simple, et $\mathfrak{h}$ son algèbre de Lie complexe. On note $H(\C)_\mathrm{ss}$ et $\mathfrak{h}_\mathrm{ss}$ les classes de $H(\C)$-conjugaison d'éléments semi-simples respectivement de $H (\C)$ et de $\mathfrak{h}$.
\\ \indent On note alors $\mathcal{X}(H)$ l'ensemble des familles $(\mathrm{c}_v)_{v\in P\cup \{\infty \}}$, où $\mathrm{c}_\infty \in \mathfrak{h}_\mathrm{ss}$ et $\mathrm{c}_p\in H(\C)_\mathrm{ss}$ pour tout $p\in P$.
\\ \indent Si on possède un morphisme de $\C$-groupe $r:H\rightarrow H'$, on note encore par $r$ l'application définie de $\mathcal{X}(H)$ dans $\mathcal{X}(H')$ envoyant $(\mathrm{c}_v)$ sur $(r(\mathrm{c}_v))$.
\end{defi}
\indent Suivant Langlands dans \cite{Lan}, on possède une application $\Pi (G)\rightarrow \mathcal{X}(\widehat{G})$ : 
\begin{propdef}[Paramétrisation de Langlands] On dispose d'une application canonique :
$$\begin{array}{rrcl}
	c:&\Pi (G) & \rightarrow & \mathcal{X}(\widehat{G}) \\
	& \pi & \mapsto & \left( \mathrm{c}_v (\pi) \right) \\
\end{array}$$
définie comme suit. Si $\pi=\pi_\infty \otimes \pi_V \in \Pi (G)$, on pose $\mathrm{c}_\infty (\pi)$ le caractère infinitésimal de $\pi_\infty$ (d'après Harish-Chandra). De plus, l'isomorphisme de Satake entraîne que $\mathrm{H}(G)$ est commutative, donc $\pi_V$ est de dimension $1$ et peut être vue comme un homomorphisme d'anneaux de $\mathrm{H}(G)$ dans $\C$. Sa restriction à $\mathrm{H}_p(G)$ est associée par la proposition \ref{tr o Sat} à un unique élément $\mathrm{c}_p (\pi) \in \widehat{G}(\C)_\mathrm{ss}$.
\\ \indent L'application $c$ ainsi définie est à fibres finies.
\end{propdef}
\begin{proof}
voir \cite[ch. VI, \S 4.2]{CL}. 
\end{proof}
\indent En particulier, le caractère $\mathrm{c}_\infty (\pi)$ nous permet de définir la notion de représentation algébrique ou autoduale dans le cas où $G$ est le groupe $\mathrm{PGL}_n$ :
\begin{defi}[Représentations algébriques] Soit $\pi\in \Pi_{\mathrm{cusp}}(\mathrm{PGL}_n)$. Les valeurs propres du caractère $\mathrm{c}_\infty (\pi)$ sont appelés les poids de $\pi$.
\\ \indent Une représentation $\pi\in \Pi_{\mathrm{cusp}}(\mathrm{PGL}_n)$ est dite algébrique si ses poids sont des demi-entiers dont les différences deux à deux sont des entiers. On note $\Pi_{\mathrm{alg}}(\mathrm{PGL}_n)$ le sous-ensemble de $\Pi_{\mathrm{cusp}}(\mathrm{PGL}_n)$ des représentations algébriques.
\end{defi}
\begin{defi}[Représentations autoduales] Soit $\pi\in \Pi(\mathrm{PGL}_n)$. On dit que la représentation $\pi$ est autoduale si elle est isomorphe à sa contragrédiente.
\\ \indent On note respectivement $\Pi^\bot(\mathrm{PGL}_n)$, $\Pi^\bot_{\mathrm{disc}}(\mathrm{PGL}_n)$, $\Pi^\bot_{\mathrm{cusp}}(\mathrm{PGL}_n)$ et $\Pi^\bot_{\mathrm{alg}}(\mathrm{PGL}_n)$ les sous-ensembles de $\Pi(\mathrm{PGL}_n)$, $\Pi_{\mathrm{disc}}(\mathrm{PGL}_n)$, $\Pi_{\mathrm{cusp}}(\mathrm{PGL}_n)$ et $\Pi_{\mathrm{alg}}(\mathrm{PGL}_n)$ constitués des représentations autoduales.
\end{defi}
\indent Suivant \cite[ch. IV, \S 3.2]{CL} par exemple, si l'on se donne une forme automorphe propre pour tous les opérateurs de Hecke, on peut lui associer une représentation automorphe discrète : il s'agit de la représentation automorphe engendrée par la forme automorphe propre considérée. Les formules de Gross et l'isomorphisme d'Harish-Chandra nous permettent de relier les valeurs propres associés aux opérateurs de Hecke d'une forme automorphe propre avec les paramètres de Satake de la représentation automorphe qu'elle engendre. On détaille ci-dessous ce lien dans le cas des représentations automorphes pour $\mathrm{SO}_8$.
\\ \\ \noindent{Exemple :} Soient $G=\mathrm{SO}_8$, $\lambda =\sum_{i=1}^4 m_i \varepsilon_i$ (suivant les notations adoptées précédemment dans le cas du groupe spécial orthogonal en dimension paire, avec $r=4$), et $W=V_\lambda$ la représentation irréductible de $\mathrm{SO}_8(\C)$ de plus haut poids $\lambda$. Soient $f\in \mathcal{M}_W(\mathrm{SO}_8)$ une forme propre, et $\pi \in \Pi_{\mathrm{disc}} (\mathrm{SO}_8)$ la représentation engendrée.
\\ \indent Le caractère infinitésimal $\mathrm{c}_\infty(\pi)$ correspond à la classe de conjugaison de $\lambda+\rho$, et ses valeurs propres dans la représentation standard sont les 
$$\pm (m_i +4-i),\ i=1,\dots ,4.$$
\indent Posons de plus : $\mathrm{T}_p (f) = \lambda_p f,\dots,  \mathrm{T}_{p,\dots ,p} (f) =\lambda_{p,\dots ,p} f$. Alors on a les relations :
\begin{center}
\begin{tabular}{rl}
$(i)$ & $p^3 \mathrm{Trace}\left( \mathrm{c}_p(\pi) \vert V_{\mathrm{St}}\right) = \lambda_p$\\
$(ii)$ & $p^5 \mathrm{Trace}\left( \mathrm{c}_p(\pi) \vert  \Lambda^2 V_{\mathrm{St}}\right) = \lambda_{p,p}+\left( p^4+2\cdot p^2+1\right)$\\
$(iii)$ & $p^6 \mathrm{Trace}\left( \mathrm{c}_p(\pi) \vert  \Lambda^3 V_{\mathrm{St}}\right) = \lambda_{p,p,p}+\left( p^2+p+1\right)\cdot \lambda_p$\\
$(iv)$ & $p^6 \mathrm{Trace}\left( \mathrm{c}_p(\pi) \vert  \Lambda^4 V_{\mathrm{St}}\right) = \lambda_{p,p,p,p}+2\cdot \lambda_{p,p} + 2\cdot \left(p^4+p^2+1 \right) $\\
\end{tabular}
\end{center}
qui découlent directement de la proposition \ref{GrossSO8}.
\\ \indent 
\subsection{La conjecture d'Arthur-Langlands}\label{7.3}
\indent Afin de formuler facilement la conjecture d'Arthur-Langlands, commençons par définir la notion de paramètre de Langlands :
\begin{defi}Soit $G$ un $\Z$-groupe semi-simple, et $r:\widehat{G} \rightarrow \mathrm{SL}_n$ une $\C$ représentation. Cette représentation induit une application $\mathcal{X}(\widehat{G}) \rightarrow \mathcal{X}(\mathrm{SL}_n)$, $(\mathrm{c}_v)\mapsto \left( r(\mathrm{c}_v) \right)$. Si $\pi \in \Pi (G)$, on lui associe l'élément :
$$\psi (\pi ,r) = r\left( c\left( \pi \right) \right) \in \mathcal{X}(\mathrm{SL}_n)$$
appelé paramètre de Langlands-Satake du couple $(\pi ,r)$.
\end{defi}
\indent Dans la suite, on reprend les notations de \cite[\S IV.4]{CL}, qui sont les suivantes :
\\ \indent - On note $\mathrm{St}_m$ la $\C$-représentation tautologique de $\mathrm{SL}_m$ sur $\C^m$. Pour $a$ et $b$ deux entiers, la somme directe et le produit tensoriel des représentations $\mathrm{St}_a$ et $\mathrm{St}_b$ nous donnent les applications naturelles :
$$\mathcal{X}\left( \mathrm{SL}_a \right) \times \mathcal{X}\left( \mathrm{SL}_b \right) \rightarrow \mathcal{X}\left( \mathrm{SL}_{a+b} \right) \indent \text{et} \indent \mathcal{X}\left( \mathrm{SL}_a \right) \times \mathcal{X}\left( \mathrm{SL}_b \right) \rightarrow \mathcal{X}\left( \mathrm{SL}_{ab} \right)$$
que l'on note respectivement $(c,c')\mapsto c\oplus c'$ et $(c,c') \mapsto c\otimes c'$.
\\ \indent - On note $e\in \mathcal{X}\left( \mathrm{SL}_2 \right)$ l'élément défini par :
$$e_p = \left[ \begin{matrix}
p^{-1/2} & 0 \\
0 & p^{1/2} \\
\end{matrix} \right]
\forall p\in P,\indent \text{et} \indent e_\infty = \left[ \begin{matrix}
-\frac{1}{2} & 0 \\
0 & \frac{1}{2}\\
\end{matrix}
\right].$$
Pour tout entier $d\geq 1$, on note $\left[ d\right]$ l'élément $\mathrm{Sym}^{d-1} (e) \in \mathcal{X}\left( \mathrm{SL}_d \right)$, où $\mathrm{Sym}^{d-1}$ désigne la représentation $\mathrm{Sym}^{d-1} \mathrm{St}_2$. Pour $m,d \geq 1$ entiers, et $c\in \mathcal{X}\left( \mathrm{SL}_m \right)$, on pose pour simplifier :
$$c\left[ d\right] = c\otimes \left[ d\right].$$
\indent - Pour $\pi\in \Pi_{\mathrm{cusp}} \left( \mathrm{PGL}_m \right)$, on note simplement $\pi$ pour désigner l'élément $c(\pi)\in \mathcal{X}\left( \mathrm{SL}_m \right)$.
\\ \\ \indent Avec les notations précédentes, si l'on se donne $n_1,\dots,n_k, d_1,\dots ,d_k$ des entiers naturels non nuls, et $\pi_i \in \Pi_{\mathrm{cusp}} \left( \mathrm{PGL} _{n_i} \right)$ pour tout $i\in \{1,\dots ,k\}$, en posant $n=\sum_{i=1}^k n_i d_i$, on dispose d'un élément bien défini :
$$\pi_1 \left[ d_1 \right] \oplus \pi_2 \left[ d_2 \right] \oplus \dots \oplus \pi_k \left[ d_k \right] \in \mathcal{X}\left( \mathrm{SL}_n \right).$$
\indent On définit alors $\mathcal{X}_\mathrm{AL}\left( \mathrm{SL}_n \right)$ comme suit :
\begin{propdef} On pose $\mathcal{X}_\mathrm{AL}\left( \mathrm{SL}_n \right)$ le sous-ensemble de $\mathcal{X}\left( \mathrm{SL}_n \right)$ des éléments la forme $\pi_1 \left[ d_1 \right] \oplus \pi_2 \left[ d_2 \right] \oplus \dots \oplus \pi_k \left[ d_k \right]$, pour un quadruplet $(k,(n_i),(d_i),(\pi_i))$ tel que $n=\sum_{i=1}^k n_i d_i$.
\\ \indent Si on se donne deux écritures $\oplus _{i=1}^k \pi_i [d_i] = \oplus_{j=1}^l \pi'_j [d'_j]$ dans $\mathcal{X}\left( \mathrm{SL}_n \right)$, alors $k=l$ et il existe une permutation $\sigma \in \mathcal{S}_k$ telle que $(\pi'_i,d'_i) = (\pi_{\sigma (i)},d_{\sigma (i)})$.
\end{propdef}
\begin{proof}
voir \cite{JS81b} et \cite{Lan79}. 
\end{proof}
\indent Suivant ces notations, on dira qu'un élément $\pi$ de $\mathcal{X}_\mathrm{AL}\left( \mathrm{SL}_n \right)$ est non endoscopique si son quadruplet $(k,(n_i),(d_i),(\pi_i))$ vérifie $k=1$ et $d_1=1$.
\\ \\ \indent Avec ces notations, on a la conjecture suivante :
\begin{conj}[Conjecture d'Arthur-Langlands] Soient $G$ un $\Z$-groupe semi-simple et $r: \widehat{G}\rightarrow \mathrm{SL}_n$ une $\C$-représentation. Si $\pi \in \Pi_{\mathrm{disc}} (G)$, alors $\psi (\pi,r) \in \mathcal{X}_{\mathrm{AL}}\left( \mathrm{SL}_n \right)$.
\end{conj}
\begin{theo} La conjecture d'Arthur-Langlands est vraie pour $G=\mathrm{SO}_n$ et $r=V_{\mathrm{St}}$.
\end{theo}
\begin{proof}
ce théorème est le produit de travaux de nombreux auteurs (Arthur \cite{Art13}, Langlands, Kottwitz, Shelstad \cite{She12a} \cite{She12b} \cite{She14}, Waldspurger \cite{Wal06} \cite{Wal14} \cite{MW14} \cite{LW13}, Ngô \cite{Ngo10}, Laumon, Chaudouard \cite{ChaL10} \cite{ChaL12}, Moeglin, Mezo \cite{Mez11} \cite{Mez13}), culminant par les travaux récents d'Arthur et Waldspurger.
\\ \indent L'énoncé ci-dessus, concernant $\mathrm{SO}_n$, est dû à Taïbi \cite{Tai16} et repose sur \cite{Art89} et les travaux de Kaletha \cite{Kal} et Arancibia-Moeglin-Renard \cite{AMR}. 
\end{proof}
\indent Lorsque $G$ est classique et que $r$ est la représentation standard $V_{\mathrm{St}}$ de $\widehat{G}$, pour tout élément $\pi \in \Pi_{\mathrm{disc}}(G)$ on a une égalité de la forme :
$$\psi (\pi,\mathrm{St}) = \bigoplus \pi_i [d_i]$$
où les $\pi$ sont autoduales. Si $G=\mathrm{SO}_n$, on constate sur les caractère infinitésimaux qu'elles sont algébriques, de sortes que les $\pi_i$ sont des éléments de $\Pi_{\mathrm{alg}}^\bot(\mathrm{PGL}_{n_i})$. De plus, on appellera ``poids de $\pi$" les valeurs propres du caractère infinitésimal de $\psi(\pi,\mathrm{St})$.
\\ \indent L'étude faite dans \cite{CR} a pour but de déterminer pour $n=7,8$ ou $9$ comment s'exprime tout élément de  $\Pi_{\mathrm{disc}}(\mathrm{SO}_n)$ grâce aux éléments de $\Pi_{\mathrm{alg}}^\bot(\mathrm{PGL}_m)$. Plus précisément, donnons-nous un tel $n$ et posons $m=[n/2]$, et posons $\overline{w}=(w_1,\dots,w_m)$ avec $w_1>\dots >w_m\geq 0$ des entiers positifs ou nuls de même parité que $n$. On désigne par $\Pi_{\overline{w}}(\mathrm{SO}_n)$ l'ensemble des éléments de $\Pi_{\mathrm{disc}}(\mathrm{SO}_n)$ dont les poids sont les $\pm \frac{w_i}{2}$. Alors les résultats de \cite{CR} permettent de calculer le cardinal de $\Pi_{\overline{w}}(\mathrm{SO}_n)$, et d'exprimer pour chacun de ses éléments le quadruplet $(k,(n_i),(d_i),(\pi_i))$ associé par la conjecture d'Arthur-Langlands. Soulignons au passage que les résultats conditionnels de \cite{CR} (les énoncés Theorem$^*$ et Theorem$^{**}$) sont maintenant inconditionnels grâce aux résultats récents des auteurs cités ci-dessus (notamment \cite{Tai16}).
\\ \indent On renvoie à \cite[Tables 12 à 14]{CR} pour une liste des représentations ainsi décrites dans les cas où les $w_i$ sont impairs, avec $n=7$ ou $9$, pour certaines valeurs de $\overline{w}$. La méthode énoncée dans \cite[chapitres 5, 6 et 7]{CR} nous permet de trouver l'ensemble des représentations de poids $\pm \frac{w_i}{2}$ dans les autres cas. On donne dans la table \ref{SO9} ci-dessous la décomposition de tous les éléments de $\Pi_{(w_1,w_2,w_3,w_4)}(\mathrm{SO}_9)$, où les $w_i$ sont des entiers impairs tels que $25=w_1 >w_2>w_3>w_4>0$, choisis de telle sorte qu'il existe un élément de $\Pi_{\mathrm{alg}}^\bot(\mathrm{PGL}_8)$ dont les poids sont les $\pm w_i/2$.
\\ \indent Les tables de \cite{CR} susmentionnées, ainsi que les tables \ref{SO9} à \ref{tableau4SO8} du présent article, font intervenir les notations suivantes. Donnons-nous $\overline{w}=(w_1,\dots,w_m)$ où $w_1>\dots>w_m\geq 0$ sont des entiers positifs de même parité. On pose $\Pi$ l'ensemble des $\pi\in\Pi_{\mathrm{alg}}^\bot(\mathrm{PGL}_{2m})$ dont les poids sont les $\pm w_i/2$ (avec $0$ de multiplicité double lorsque $w_m=0$). Si $\vert \Pi \vert =1$, on note $\Delta_{w_1,\dots,w_m}$ son unique élément. Si $\vert \Pi \vert =k$, on note $\Delta^k_{w_1,\dots,w_m}$ n'importe lequel de ses élément. De la même manière, donnons-nous $\overline{w}=(w_1,\dots,w_m)$ où $w_1>\dots>w_m>0$ sont des entiers pairs strictement positifs, et posons $\Pi^*$ l'ensemble des $\pi\in\Pi_{\mathrm{alg}}^\bot(\mathrm{PGL}_{2m+1})$ dont les poids sont les $\pm w_i/2$ et $0$. Si $\vert \Pi^* \vert =1$, on note $\Delta^*_{w_1,\dots,w_m}$ son unique élément. Si $\vert \Pi^* \vert =k$, on note $\Delta^{*k}_{w_1,\dots,w_m}$ n'importe lequel de ses élément.
\subsection{Résultats obtenus}\label{7.4}
\indent Notre but est de déterminer un maximum de paramètres de Langlands-Satake pour des représentations automorphes cuspidales pour les groupes $\mathrm{GL}_n$. Grâce à la conjecture d'Arthur-Langlands, ces paramètres apparaissent comme les éléments fondamentaux pour comprendre les formes automorphes discrètes de groupes plus généraux.
\\ \indent Donnons-nous $\pi\in \Pi_{\mathrm{cusp}}(\mathrm{PGL}_m)$ dont on souhaite déterminer les $\mathrm{c}_p(\pi)$. On procède comme suit :
\\ \indent - on cherche $\pi'\in \Pi_{\mathrm{disc}}(\mathrm{SO}_n)$ pour $n\in\{7,8,9\}$ tel que $\pi$ apparaisse dans l'écriture de $\psi(\pi',\mathrm{St})$ donnée par la conjecture d'Arthur-Langlands. Idéalement, on espère obtenir une égalité de la forme $\psi(\pi',\mathrm{St})=\pi$ ou $\psi(\pi',\mathrm{St})=\pi\oplus [d]$.
\\ \indent - on détermine grâce à l'étude de \cite{CR} l'ensemble des $\pi''\in \Pi_{\mathrm{disc}}(\mathrm{SO}_n)$ ayant les mêmes poids que $\pi'$, et on détermine les paramètres $\psi(\pi'',\mathrm{St})$ associés.
\\ \indent - sous réserve que les paramètres $\psi(\pi'',\mathrm{St})$ font intervenir uniquement des éléments dont les paramètres de Langlands-Satake sont bien connus, on en déduit les $\mathrm{c}_p(\pi)$ (ou du moins les $\mathrm{Trace}(\mathrm{c}_p(\pi) \vert \Lambda^i \mathrm{St})$ pour certains $i$).
\\ \\ \indent En guise d'exemple, détaillons comment on a étudié l'élément $\Delta_{23,15,7} \in \Pi_{\mathrm{alg}}^\bot(\mathrm{PGL}_6)$, qui est un cas assez représentatif. La notation $\Delta_{w_1,\dots,w_m}$ a été présentée au paragraphe \ref{7.3} (et est notamment utilisée dans \cite{CR} ou \cite{CL}).
\\ \indent D'après \cite[Table 12]{CR}, on a l'égalité : $\Pi_{(23,15,7)}(\mathrm{SO}_7) = \{\Delta_{23,7}\oplus \Delta_{15} ,\ \Delta_{23,15,7} \}$. On pose pour simplifier $\pi_1 = \Delta_{23,7}\oplus \Delta_{15}$ et $\pi_2 = \Delta_{23,15,7}$. Soit $W$ la représentation de plus haut poids $(9,6,3)$ de $\mathrm{SO}_7$. On pose $f_1$ et $f_2$ des formes propres de $\mathcal{M}_W (\mathrm{SO}_7)$ qui engendrent respectivement $\pi_1$ et $\pi_2$. On pose de plus, pour $p$ un nombre premier quelconque, les valeurs propres de $f_1$ et $f_2$ pour les opérateurs de Hecke $\mathrm{T}_p,\dots ,\mathrm{T}_{p,\dots ,p}$ comme étant respectivement les $\mu_p,\dots ,\mu_{p,\dots,p}$ et les $\lambda_p,\dots ,\lambda_{p\dots ,p}$. Les formules de Gross nous donnent alors les égalités :
$$\begin{aligned}
\mathrm{Tr}(\mathrm{T}_p \vert \mathcal{M}_W (\mathrm{SO}_7)) & =  \mu_p + \lambda_p = p^{5/2}\cdot \mathrm{Trace}(\mathrm{c}_p (\pi_1) \vert V_{\mathrm{St}})+p^{5/2}\cdot \mathrm{Trace}(\mathrm{c}_p (\pi_2) \vert V_{\mathrm{St}})\\
\mathrm{Tr}(\mathrm{T}_{p,p} \vert \mathcal{M}_W (\mathrm{SO}_7)) & =  \mu_{p,p} + \lambda_{p,p} = p^4\cdot \mathrm{Trace}(\mathrm{c}_p (\pi_1) \vert \Lambda^2 V_{\mathrm{St}})\\
 & \indent \indent \indent \indent +p^4\cdot \mathrm{Trace}(\mathrm{c}_p (\pi_2) \vert \Lambda^2 V_{\mathrm{St}})-2\cdot (p^4+p^2+1)\\
\mathrm{Tr}(\mathrm{T}_{p,p,p} \vert \mathcal{M}_W (\mathrm{SO}_7))& =  \mu_{p,p,p} + \lambda_{p,p,p} = p^{9/2}\cdot \mathrm{Trace}(\mathrm{c}_p (\pi_1) \vert \Lambda^3 V_{\mathrm{St}})\\
& \indent \indent \indent +p^{9/2}\cdot \mathrm{Trace}(\mathrm{c}_p (\pi_2) \vert \Lambda^3 V_{\mathrm{St}}) -(p^2+1)\cdot \mathrm{Tr}(\mathrm{T}_p \vert \mathcal{M}_W (\mathrm{SO}_7)) \\
\end{aligned}$$
\indent Les quantités qui nous intéressent ici sont les $\mathrm{Trace}(\mathrm{c}_p (\pi_2) \vert \Lambda^i V_{\mathrm{St}})$ (pour $i=1,2,3$). Pour les déterminer, on a besoin des traces des opérateurs de Hecke $T_p,T_{p,p},T_{p,p,p}$ (qu'on a calculées au paragraphe \ref{6.1}), et des quantités $\mathrm{Trace}(\mathrm{c}_p (\pi_1) \vert \Lambda^i V_{\mathrm{St}})$ (qui sont calculable à l'aide des quantités $\mathrm{Trace}(\mathrm{c}_p (\Delta_{23,7}) \vert \Lambda^i V_{\mathrm{St}})$ et $\mathrm{Trace}(\mathrm{c}_p (\Delta_{15}) \vert \Lambda^i V_{\mathrm{St}})$, et des lemmes techniques \ref{formuleplus}, \ref{formulefois} et \ref{formuled} présentés ci-dessous).
\\ \indent Reste donc à calculer les quantités $\mathrm{Trace}(\mathrm{c}_p (\Delta_{23,7}) \vert \Lambda^i V_{\mathrm{St}})$ et $\mathrm{Trace}(\mathrm{c}_p (\Delta_{15}) \vert \Lambda^i V_{\mathrm{St}})$. Si on note $V_{(m_1,m_2,m_3)}$ la représentation de $\mathrm{SO}_7$ de plus haut poids $(m_1,m_2,m_3)$, alors les formules de Gross nous donnent les égalités suivantes :
 $$\begin{aligned}
\mathrm{Tr}(\mathrm{T}_p \vert \mathcal{M}_{V_{(4,4,4)}} (\mathrm{SO}_7)) & = p^{5/2}\cdot \mathrm{Trace}(\mathrm{c}_p (\Delta_{11} [3]) \vert V_{\mathrm{St}}) \\
&= p^{5/2} \mathrm{Trace}(\mathrm{c}_p (\Delta_{11}) \vert V_{\mathrm{St}}) \cdot \mathrm{Trace}(\mathrm{c}_p ([3]) \vert V_{\mathrm{St}})\\
\mathrm{Tr}(\mathrm{T}_p \vert \mathcal{M}_{V_{(6,6,6)}} (\mathrm{SO}_7)) &= p^{5/2}\cdot \mathrm{Trace}(\mathrm{c}_p (\Delta_{15} [3]) \vert V_{\mathrm{St}}) \\
&= p^{5/2} \mathrm{Trace}(\mathrm{c}_p (\Delta_{15}) \vert V_{\mathrm{St}}) \cdot \mathrm{Trace}(\mathrm{c}_p ([3]) \vert V_{\mathrm{St}})\\
\mathrm{Tr}(\mathrm{T}_p \vert \mathcal{M}_{V_{(9,4,3)}} (\mathrm{SO}_7)) &= p^{5/2}\cdot \mathrm{Trace}(\mathrm{c}_p (\Delta_{23,7} \oplus \Delta_{11}) \vert V_{\mathrm{St}}) \\
& = p^{5/2} \left( \mathrm{Trace}(\mathrm{c}_p (\Delta_{23,7}) \vert V_{\mathrm{St}}) + \mathrm{Trace}(\mathrm{c}_p (\Delta_{11}) \vert V_{\mathrm{St}}) \right)\\
\mathrm{Tr}(\mathrm{T}_{p,p} \vert \mathcal{M}_{V_{(9,4,3)}} (\mathrm{SO}_7)) &= p^{4}\cdot \mathrm{Trace}(\mathrm{c}_p (\Delta_{23,7} \oplus \Delta_{11}) \vert \Lambda^2 V_{\mathrm{St}}) +(p^4+p^2+1)\\
& = p^{4} ( \mathrm{Trace}(\mathrm{c}_p (\Delta_{23,7}) \vert \Lambda^2 V_{\mathrm{St}})+\mathrm{Trace}(\mathrm{c}_p (\Delta_{11}) \vert \Lambda^2 V_{\mathrm{St}}) \\
& \indent \indent + \mathrm{Trace}(\mathrm{c}_p (\Delta_{23,7}) \vert V_{\mathrm{St}}) \cdot \mathrm{Trace}(\mathrm{c}_p (\Delta_{11}) \vert V_{\mathrm{St}}) )\\
\end{aligned}$$
\indent On constate aussi que :
$$\begin{aligned}
\mathrm{Trace}(\mathrm{c}_p (\Delta_{11}) \vert \Lambda^2 V_{\mathrm{St}})&=\mathrm{Trace}(\mathrm{c}_p (\Delta_{15}) \vert \Lambda^2 V_{\mathrm{St}})=1 \\
\mathrm{Trace}(\mathrm{c}_p (\Delta_{11}) \vert \Lambda^i V_{\mathrm{St}})&=\mathrm{Trace}(\mathrm{c}_p (\Delta_{15}) \vert \Lambda^i V_{\mathrm{St}})=0 \text{ pour }i\geq 3\\
\mathrm{Trace}(\mathrm{c}_p (\Delta_{23,7}) \vert \Lambda^3 V_{\mathrm{St}})&=\mathrm{Trace}(\mathrm{c}_p (\Delta_{23,7}) \vert \Lambda^2 V_{\mathrm{St}}) \\
\end{aligned}
$$
\indent Il suffit enfin de réinjecter ces valeurs dans les égalités précédentes pour calculer les quantités $\mathrm{Trace}(\mathrm{c}_p (\Delta_{23,15,7}) \vert \Lambda^i V_{\mathrm{St}})$, données par les tables \ref{tableau1SO7}, \ref{tableau3SO7}, \ref{tableau4SO7} et \ref{tableau5SO7}.
\\ \indent Les résultats obtenus sont décrits par les théorèmes \ref{theo1}, \ref{theo2}, \ref{theo3} et \ref{theo4} présentés en introduction.
\\ \indent Enfin, les quelques lemmes techniques suivant sont particulièrement utiles dans nos calculs, et leur démonstration est immédiate :
\begin{lem} Soit $\pi \in \mathcal{X}(\mathrm{SL}_n)$. Alors on a les relations :
$$
\begin{aligned}
\mathrm{Trace}&\left( \mathrm{c}_p(\pi)^2 \vert V_{\mathrm{St}}\right) = \mathrm{Trace}\left( \mathrm{c}_p(\pi) \vert V_{\mathrm{St}}\right)^2 - 2\cdot \mathrm{Trace}\left( \mathrm{c}_p(\pi) \vert \Lambda^2 V_{\mathrm{St}}\right)\\
\\
\mathrm{Trace}&\left( \mathrm{c}_p(\pi)^3 \vert V_{\mathrm{St}}\right) = \mathrm{Trace}\left( \mathrm{c}_p(\pi) \vert V_{\mathrm{St}}\right)^3 + 3\cdot \mathrm{Trace}\left( \mathrm{c}_p(\pi) \vert \Lambda^3 V_{\mathrm{St}}\right)\\
&-3\cdot \mathrm{Trace}\left( \mathrm{c}_p(\pi) \vert \Lambda^2 V_{\mathrm{St}}\right)\cdot \mathrm{Trace}\left( \mathrm{c}_p(\pi) \vert V_{\mathrm{St}}\right)\\
\\
\mathrm{Trace}&\left( \mathrm{c}_p(\pi)^4 \vert V_{\mathrm{St}}\right) = \mathrm{Trace}\left( \mathrm{c}_p(\pi) \vert V_{\mathrm{St}}\right)^4 - 4\cdot \mathrm{Trace}\left( \mathrm{c}_p(\pi) \vert \Lambda^4 V_{\mathrm{St}}\right)\\
&-4\cdot \mathrm{Trace}\left( \mathrm{c}_p(\pi) \vert \Lambda^2 V_{\mathrm{St}}\right)\cdot \mathrm{Trace}\left( \mathrm{c}_p(\pi) \vert V_{\mathrm{St}}\right)^2\\
&+4\cdot \mathrm{Trace}\left( \mathrm{c}_p(\pi) \vert \Lambda^3 V_{\mathrm{St}}\right)\cdot \mathrm{Trace}\left( \mathrm{c}_p(\pi) \vert V_{\mathrm{St}}\right)+2 \cdot \mathrm{Trace}\left( \mathrm{c}_p(\pi) \vert \Lambda^2 V_{\mathrm{St}}\right)^2\\
\end{aligned}
$$
\end{lem}
\begin{lem} \label{formuleplus} Soient $\pi_i\in \mathcal{X}(\mathrm{SL}_{n_i})$ et $n=\sum_i n_i$. L'élément $\bigoplus _i \pi_i \in \mathcal{X}(\mathrm{SL}_n)$ est alors bien défini, et on a :
$$\mathrm{Trace}\left( \mathrm{c}_p \left(\bigoplus_i \pi_i \right) \vert \Lambda^k V_\mathrm{St} \right) = \sum_{\substack{0<j\leq k \\ n_1+\dots +n_j =k \\ i_1<\dots < i_j}} \prod_{l=1}^j \mathrm{Trace}\left( \mathrm{c}_p(\pi_{i_l}) \vert \Lambda^{n_l} V_{\mathrm{St}}\right) $$
\indent En particulier, les cas $k=1$ ou $k=2$ donnent :
$$
\begin{aligned}
&\mathrm{Trace}\left( \mathrm{c}_p\left(\bigoplus _i \pi_i\right) \vert V_{\mathrm{St}}\right) = \sum_i \mathrm{Trace}\left( \mathrm{c}_p(\pi_i) \vert V_{\mathrm{St}}\right)\\
\\
&\mathrm{Trace}\left( \mathrm{c}_p\left(\bigoplus _i \pi_i\right) \vert \Lambda^2 V_{\mathrm{St}}\right) = \sum_i \mathrm{Trace}\left( \mathrm{c}_p(\pi_i) \vert \Lambda^2 V_{\mathrm{St}}\right)+\sum_{i<j} \mathrm{Trace}\left( \mathrm{c}_p(\pi_i) \vert V_{\mathrm{St}}\right)\cdot \mathrm{Trace}\left( \mathrm{c}_p(\pi_j) \vert V_{\mathrm{St}}\right)\\
\end{aligned}
$$
\end{lem}
\begin{lem} \label{formulefois} Soit $\pi_1 \in \mathcal{X}(\mathrm{SL}_{n_1})$ et $\pi_2 \in \mathcal{X}(\mathrm{SL}_{n_2})$. L'élément $\pi_1 \otimes \pi_2 \in \mathcal{X}(\mathrm{SL}_{n_1\cdot n_2})$ est alors bien défini, et on a :
$$\begin{aligned}
\mathrm{Trace}&\left( \mathrm{c}_p (\pi_1\otimes \pi_2) \vert V_\mathrm{St} \right) = \mathrm{Trace}\left( \mathrm{c}_p (\pi_1 ) \vert V_\mathrm{St} \right) \cdot \mathrm{Trace}\left( \mathrm{c}_p (\pi_2) \vert V_\mathrm{St} \right)\\
\\
\mathrm{Trace}&\left( \mathrm{c}_p (\pi_1\otimes \pi_2) \vert \Lambda^2 V_\mathrm{St} \right) = \mathrm{Trace}\left( \mathrm{c}_p (\pi_1 ) \vert V_\mathrm{St} \right)^2 \cdot \mathrm{Trace}\left( \mathrm{c}_p (\pi_2) \vert \Lambda^2 V_\mathrm{St} \right) \\
& +\mathrm{Trace}\left( \mathrm{c}_p (\pi_1 ) \vert \Lambda^2 V_\mathrm{St} \right) \cdot \mathrm{Trace}\left( \mathrm{c}_p (\pi_2) \vert  V_\mathrm{St} \right) ^2 -2\cdot \mathrm{Trace}\left( \mathrm{c}_p (\pi_1 ) \vert \Lambda^2 V_\mathrm{St} \right) \cdot \mathrm{Trace}\left( \mathrm{c}_p (\pi_2) \vert \Lambda^2 V_\mathrm{St} \right) \\
\\
\mathrm{Trace}&\left( \mathrm{c}_p (\pi_1\otimes \pi_2) \vert \Lambda^3 V_\mathrm{St} \right) = \mathrm{Trace}\left( \mathrm{c}_p (\pi_1 ) \vert V_\mathrm{St} \right)^3 \cdot \mathrm{Trace}\left( \mathrm{c}_p (\pi_2) \vert \Lambda^3 V_\mathrm{St} \right) \\
& +\mathrm{Trace}\left( \mathrm{c}_p (\pi_1 ) \vert \Lambda^3 V_\mathrm{St} \right) \cdot \mathrm{Trace}\left( \mathrm{c}_p (\pi_2) \vert  V_\mathrm{St} \right) ^3 +3\cdot \mathrm{Trace}\left( \mathrm{c}_p (\pi_1 ) \vert \Lambda^3 V_\mathrm{St} \right) \cdot \mathrm{Trace}\left( \mathrm{c}_p (\pi_2) \vert \Lambda^3 V_\mathrm{St} \right) \\
& -3\cdot \mathrm{Trace}\left( \mathrm{c}_p (\pi_1 ) \vert \Lambda^2 V_\mathrm{St} \right) \cdot \mathrm{Trace}\left( \mathrm{c}_p (\pi_1 ) \vert V_\mathrm{St} \right)\cdot \mathrm{Trace}\left( \mathrm{c}_p (\pi_2) \vert  \Lambda^3 V_\mathrm{St} \right) \\
&-3\cdot \mathrm{Trace}\left( \mathrm{c}_p (\pi_1 ) \vert \Lambda^3 V_\mathrm{St} \right) \cdot \mathrm{Trace}\left( \mathrm{c}_p (\pi_2 ) \vert \Lambda^2 V_\mathrm{St} \right) \cdot \mathrm{Trace}\left( \mathrm{c}_p (\pi_2) \vert V_\mathrm{St} \right) \\
&-3\cdot  \mathrm{Trace}\left( \mathrm{c}_p (\pi_1 ) \vert \Lambda^2 V_\mathrm{St} \right) \cdot \mathrm{Trace}\left( \mathrm{c}_p (\pi_1 ) \vert V_\mathrm{St} \right) \cdot \mathrm{Trace}\left( \mathrm{c}_p (\pi_2 ) \vert \Lambda^2 V_\mathrm{St} \right) \cdot \mathrm{Trace}\left( \mathrm{c}_p (\pi_2) \vert V_\mathrm{St} \right) \\
\end{aligned}
$$
\end{lem}
\begin{lem} \label{formuled} Soit $d\geq 1$, et $[d]\in \mathcal{X}(\mathrm{
SL}_d)$. On a alors :
$$\begin{aligned}
& \mathrm{Trace}\left( \mathrm{c}_p ([d]) \vert V_\mathrm{St} \right) = p^{\frac{1-d}{2}} \cdot \frac{p^d-1}{p-1}\\
& \mathrm{Trace}\left( \mathrm{c}_p ([d]) \vert \Lambda^2 V_\mathrm{St} \right) = p^{2-d} \cdot \frac{(p^d-1)\cdot (p^{d-1}-1)}{(p-1)^2 \cdot (p+1)}\\
& \mathrm{Trace}\left( \mathrm{c}_p ([d]) \vert \Lambda^3 V_\mathrm{St} \right) = p^\frac{3\cdot (3-d)}{2} \cdot \frac{(p^d-1)\cdot (p^{d-1}-1) \cdot (p^{d-2}-1)}{(p-1)^3 \cdot (p+1) \cdot (p^2+p+1)}\\
& \mathrm{Trace}\left( \mathrm{c}_p ([d]) \vert \Lambda^4 V_\mathrm{St} \right) = p^{2\cdot (4-d)} \cdot \frac{(p^d-1)\cdot (p^{d-1}-1) \cdot (p^{d-2}-1)\cdot (p^{d-3}-1)}{(p-1)^4 \cdot (p+1)^2 \cdot (p^2+p+1)\cdot (p^2+1)}\\
\end{aligned}
$$
\end{lem}
\clearpage
\begin{table} \begin{center} \renewcommand\arraystretch{1.2} \begin{tabular}{|c|m{5cm}||c|m{5cm}|}
\hline
\rule[-0.2cm]{0mm}{0.8cm} \small $(w_1,w_2,w_3,w_4)$ & \centering $\Pi_{w_1,w_2,w_3,w_4}(\mathrm{SO}_9)$ & \small $(w_1,w_2,w_3,w_4)$ & \centering $\Pi_{w_1,w_2,w_3,w_4}(\mathrm{SO}_9)$ \tabularnewline
\hline
\centering $(25,17,9,5)$ & \centering $ \Delta_{25,17,9,5}$ & \centering $(25,21,17,7)$ & \centering $ \Delta_{25,21,7}^2\oplus \Delta_{17},\  \Delta_{25,21,17,7}$ \tabularnewline\hline
\centering $(25,17,13,5)$ & \centering $ \Delta_{25,17,13,5}$ & \centering $(25,21,17,9)$ & \centering $\Delta_{25}\oplus \Delta_{21,9} \oplus \Delta_{17},$\newline$\Delta_{25,17}\oplus \Delta_{21,9},\  \Delta_{25,21,17,9}$ \tabularnewline\hline
\centering $(25,19,9,3)$ & \centering $ \Delta_{25,19,9,3}$ & \centering $(25,23,9,3)$ & \centering $ \Delta_{25,23,9,3}$ \tabularnewline\hline
\centering $(25,19,11,5)$ & \centering $ \Delta_{25,19,5}^2\oplus \Delta_{11},\  \Delta_{25,19,11,5}$ & \centering $(25,23,11,1)$ & \centering $ \Delta_{25,23,11,1}$ \tabularnewline\hline
\centering $(25,19,13,3)$ & \centering $ \Delta_{25,19,13,3}$ & \centering $(25,23,11,5)$ & \centering $ \Delta_{25,23,11,5}^2$ \tabularnewline\hline
\centering $(25,19,13,5)$ & \centering $ \Delta_{25,19,13,5}$ & \centering $(25,23,13,3)$ & \centering $ \Delta_{25,23,13,3}$ \tabularnewline\hline
\centering $(25,19,13,7)$ & \centering $ \Delta_{25,13}^2\oplus \Delta_{19,7},\  \Delta_{25,19,13,7}$ & \centering $(25,23,13,7)$ & \centering $ \Delta_{25,13}^2\oplus \Delta_{23,7},\  \Delta_{25,23,13,7}$ \tabularnewline\hline
\centering $(25,19,13,9)$ & \centering $ \Delta_{25,19,13,9}$ & \centering $(25,23,15,1)$ & \centering $ \Delta_{25,23,15,1}$ \tabularnewline\hline
\centering $(25,19,15,5)$ & \centering $ \Delta_{25,19,5}^2\oplus \Delta_{15},\  \Delta_{25,19,15,5}$ & \centering $(25,23,15,5)$ & \centering $ \Delta_{25,23,15,5}^3$ \tabularnewline\hline
\centering $(25,21,11,7)$ & \centering $ \Delta_{25,21,7}^2\oplus \Delta_{11},\  \Delta_{25,21,11,7}$ & \centering $(25,23,15,9)$ & \centering $ \Delta_{25}\oplus \Delta_{15} \oplus \Delta_{23,9},$ \newline $\Delta_{25,15}\oplus \Delta_{23,9},\  \Delta_{25,23,15,9}$ \tabularnewline\hline
\centering $(25,21,13,5)$ & \centering $ \Delta_{25,13}^2\oplus \Delta_{21,5},\  \Delta_{25,21,13,5}$ & \centering $(25,23,15,11)$ & \centering $ \Delta_{25,23,15,11}$ \tabularnewline\hline
\centering $(25,21,13,7)$ & \centering $ \Delta_{25,21,13,7}$ & \centering $(25,23,17,3)$ & \centering $ \Delta_{25,23,17,3}$ \tabularnewline\hline
\centering $(25,21,15,3)$ & \centering $ \Delta_{25,21,3}^2\oplus \Delta_{15},\  \Delta_{25,21,15,3}$ & \centering $(25,23,17,5)$ & \centering $\Delta_{23,17,5}\oplus \Delta_{25},\  \Delta_{25,23,17,5}$ \tabularnewline\hline
\centering $(25,21,15,5)$ & \centering $ \Delta_{25}\oplus \Delta_{21}\oplus \Delta_{21,5},$ \newline $\Delta_{25,15}\oplus \Delta_{21,5},\  \Delta_{25,21,15,5}$ & \centering $(25,23,17,7)$ & \centering $\Delta_{25}\oplus \Delta_{17} \oplus \Delta_{23,7},$ \newline $\Delta_{25,17} \oplus \Delta_{23,7},\  \Delta_{25,23,17,7}$ \tabularnewline\hline
\centering $(25,21,15,7)$ & \centering $ \Delta_{25,21,7}^2\oplus \Delta_{15},\  \Delta_{25,21,15,7}^2$ & \centering $(25,23,17,11)$ & \centering $ \Delta_{25,23,17,11}$ \tabularnewline\hline
\centering $(25,21,15,9)$ & \centering $ \Delta_{25}\oplus \Delta_{15}\oplus \Delta_{21,9},$\newline$\Delta_{25,15}\oplus \Delta_{21,9},\  \Delta_{25,21,15,9}$ & \centering $(25,23,19,5)$ & \centering $ \Delta_{25,23,19,5}$ \tabularnewline\hline
\centering $(25,21,17,5)$ & \centering $ \Delta_{25}\oplus \Delta_{17}\oplus \Delta_{21,5},$ \newline $\Delta_{25,17}\oplus \Delta_{21,5},\  \Delta_{25,21,17,5}$ &  &  \tabularnewline\hline
\end{tabular} \caption{Décomposition des éléments de $\Pi_{(w_1,w_2,w_3,w_4)}(\mathrm{SO}_9)$ pour $w_1=25$ lorsqu'il existe un élément de $\Pi_{\mathrm{cusp}} (\mathrm{PGL}_8)$ dont les poids sont les $\pm w_i/2$.}\label{SO9} \end{center} \end{table}
\begin{table} \begin{center} \renewcommand\arraystretch{1.2} \begin{tabular}{|c|m{12cm}|}
\hline
\rule[-0.4cm]{0mm}{1cm} $(w_1,w_2,w_3)$ & \centering $\mathrm{det}\left( 2^{w_1/2} X\cdot \mathrm{Id} - \mathrm{c}_2\left( \Delta_{w_1,w_2 ,w_3} \right) \right)$ \tabularnewline
\hline
\small \centering $(23,13,5)$ & \small \centering $ 2^{ 69 }\cdot X^6 + \dots -14948499456\cdot X^3 - 4472832\cdot X^2 + 1$ \tabularnewline\hline
\small \centering $(23,15,3)$ & \small \centering $ 2^{ 69 }\cdot X^6 + \dots  + 17641242624\cdot X^3 + 7139328\cdot X^2 + 3360\cdot X + 1$ \tabularnewline\hline
\small \centering $(23,15,7)$ & \small \centering $ 2^{ 69 }\cdot X^6 + \dots  + 528482304\cdot X^3 - 4288512\cdot X^2 + 720\cdot X + 1$ \tabularnewline\hline
\small \centering $(23,17,5)$ & \small \centering $ 2^{ 69 }\cdot X^6 + \dots -22246588416\cdot X^3 + 7323648\cdot X^2 - 1920\cdot X + 1$ \tabularnewline\hline
\small \centering $(23,17,9)$ & \small \centering $ 2^{ 69 }\cdot X^6 + \dots  + 5190451200\cdot X^3 - 417792\cdot X^2 - 1584\cdot X + 1$ \tabularnewline\hline
\small \centering $(23,19,3)$ & \small \centering $ 2^{ 69 }\cdot X^6 + \dots -8241807360\cdot X^3 + 872448\cdot X^2 + 96\cdot X + 1$ \tabularnewline\hline
\small \centering $(23,19,11)$ & \small \centering $ 2^{ 69 }\cdot X^6 + \dots -6259998720\cdot X^3 - 4288512\cdot X^2 + 96\cdot X + 1$ \tabularnewline\hline
\small \centering $(25,13,3)$ & \small \centering $ 2^{ 75 }\cdot X^6 + \dots -56170119168\cdot X^3 + 16023552\cdot X^2 + 8640\cdot X + 1$ \tabularnewline\hline
\small \centering $(25,13,7)$ & \small \centering $ 2^{ 75 }\cdot X^6 + \dots  + 1962934272\cdot X^3 + 5332992\cdot X^2 - 5040\cdot X + 1$ \tabularnewline\hline
\small \centering $(25,15,5)$ & \small \centering $ 2^{ 75 }\cdot X^6 + \dots -119587995648\cdot X^3 + 21331968\cdot X^2 + 1$ \tabularnewline\hline
\small \centering $(25,15,9)$ & \small \centering $ 2^{ 75 }\cdot X^6 + \dots -335208775680\cdot X^3 - 23052288\cdot X^2 + 6048\cdot X + 1$ \tabularnewline\hline
\small \centering $(25,17,11)$ & \small \centering $ 2^{ 75 }\cdot X^6 + \dots  + 185377751040\cdot X^3 - 11071488\cdot X^2 - 6432\cdot X + 1$ \tabularnewline\hline
\small \centering $(25,19,1)$ & \small \centering $ 2^{ 75 }\cdot X^6 + \dots  + 443421818880\cdot X^3 + 72425472\cdot X^2 + 10752\cdot X + 1$ \tabularnewline\hline
\small \centering $(25,19,13)$ & \small \centering $ 2^{ 75 }\cdot X^6 + \dots -173801472000\cdot X^3 - 3053568\cdot X^2 - 672\cdot X + 1$ \tabularnewline\hline
\small \centering $(25,21,15)$ & \small \centering $ 2^{ 75 }\cdot X^6 + \dots -106419978240\cdot X^3 - 14020608\cdot X^2 - 672\cdot X + 1$ \tabularnewline\hline
\end{tabular} \caption{Polynômes caractéristiques des $\mathrm{c}_2 (\pi)$ pour $\pi \in \Pi^\bot_{\mathrm{alg}}\left( \mathrm{PGL}_6\right)$}  \label{tableau1SO7} \end{center} \end{table}
\begin{table} \begin{center} \renewcommand\arraystretch{1.2} \begin{tabular}{|c|m{8cm}|}
\hline
\rule[-0.2cm]{0mm}{0.8cm} $(w_1,w_2,w_3)$ & \centering $P_{w_1,w_2,w_3}(X)$ \tabularnewline
\hline
\small \centering $(25,17,3)$ & \small \centering $X^2 + 768\cdot X - 2764800$ \tabularnewline
\hline
\small \centering $(25,17,7)$ & \small \centering $X^2 - 5232\cdot X - 23063040$ \tabularnewline
\hline
\small \centering $(25,19,5)$ & \small \centering $X^2 - 6624\cdot X - 38854656$ \tabularnewline
\hline
\small \centering $(25,19,9)$ & \small \centering $X^2 + 1104\cdot X - 35306496$ \tabularnewline
\hline
\small \centering $(25,21,3)$ & \small \centering $X^2 - 2880\cdot X - 8193024$ \tabularnewline
\hline
\small \centering $(25,21,7)$ & \small \centering $X^2 + 240\cdot X - 28491264$ \tabularnewline
\hline
\small \centering $(25,21,11)$ & \small \centering $X^2 + 1824\cdot X - 42771456$ \tabularnewline
\hline
\end{tabular} \caption{Polynômes annulateurs des $2^{w_1/2}\cdot \mathrm{Trace}(\mathrm{c}_2 (\pi))$ pour $\pi \in \Pi^\bot_{\mathrm{alg}}\left( \mathrm{PGL}_6\right)$} \label{tableau2SO7} \end{center} \end{table}
\clearpage
\begin{sidewaystable} \begin{center}\renewcommand\arraystretch{1.2} \begin{tabular}{|c|c|c|c|c|c|c|c|}
\hline
\backslashbox{$\overline{w}$}{$p$} & \centering $3$ & \centering $5$ & \centering $7$  & \centering $11$ & \centering $13$  & \centering $17$ & \centering $19$  \tabularnewline
\hline
\small \centering $(23,13,5)$ & \footnotesize \centering $-304668$ & \footnotesize \centering $874314$ & \footnotesize \centering $452588136$ & \footnotesize \centering $-1090903017204$ & \footnotesize \centering $1624277793138$ & \footnotesize \centering $126454166788950$ & \footnotesize \centering $-119149415901516$ \tabularnewline
\hline
\small \centering $(23,15,3)$ & \footnotesize \centering $-47628$ & \footnotesize \centering $83069994$ & \footnotesize \centering $-4690439544$ & \footnotesize \centering $-412279403844$ & \footnotesize \centering $8898668260818$ & \footnotesize \centering $-106699425426090$ & \footnotesize \centering $-312437470082556$ \tabularnewline
\hline
\small \centering $(23,15,7)$ & \footnotesize \centering $425412$ & \footnotesize \centering $-124558326$ & \footnotesize \centering $-3040958424$ & \footnotesize \centering $352045171116$ & \footnotesize \centering $-4816260369102$ & \footnotesize \centering $99848197859670$ & \footnotesize \centering $129801738947604$ \tabularnewline
\hline
\small \centering $(23,17,5)$ & \footnotesize \centering $-37548$ & \footnotesize \centering $9957354$ & \footnotesize \centering $-3491256504$ & \footnotesize \centering $1417257011676$ & \footnotesize \centering $-5403644192622$ & \footnotesize \centering $-1644876121770$ & \footnotesize \centering $-824110968459036$ \tabularnewline
\hline
\small \centering $(23,17,9)$ & \footnotesize \centering $161028$ & \footnotesize \centering $118413450$ & \footnotesize \centering $-3221005656$ & \footnotesize \centering $-1654692256404$ & \footnotesize \centering $-5869020263502$ & \footnotesize \centering $-8093664534186$ & \footnotesize \centering $676095496191060$ \tabularnewline
\hline
\small \centering $(23,19,3)$ & \footnotesize \centering $-201852$ & \footnotesize \centering $-26872950$ & \footnotesize \centering $4686149544$ & \footnotesize \centering $465927593196$ & \footnotesize \centering $-7534226506062$ & \footnotesize \centering $-90400042234026$ & \footnotesize \centering $392917842132180$ \tabularnewline
\hline
\small \centering $(23,19,11)$ & \footnotesize \centering $-252252$ & \footnotesize \centering $26651850$ & \footnotesize \centering $6781882344$ & \footnotesize \centering $25215729996$ & \footnotesize \centering $2875236177138$ & \footnotesize \centering $-128845421894826$ & \footnotesize \centering $-41596411782540$ \tabularnewline
\hline
\small \centering $(25,13,3)$ & \footnotesize \centering $-19764$ & \footnotesize \centering $-391988430$ & \footnotesize \centering $9750417432$ & \footnotesize \centering $13078424975076$ & \footnotesize \centering $-96701634737526$ & \footnotesize \centering $2452876322679990$ & \footnotesize \centering $-2642714743857924$ \tabularnewline
\hline
\small \centering $(25,13,7)$ & \footnotesize \centering $-112644$ & \footnotesize \centering $-559352430$ & \footnotesize \centering $-1243505928$ & \footnotesize \centering $7826821995636$ & \footnotesize \centering $107438724171114$ & \footnotesize \centering $-2831213421327690$ & \footnotesize \centering $-9749582433259284$ \tabularnewline
\hline
\small \centering $(25,15,5)$ & \footnotesize \centering $867132$ & \footnotesize \centering $-613050606$ & \footnotesize \centering $5377223544$ & \footnotesize \centering $-3134062555596$ & \footnotesize \centering $51842671522026$ & \footnotesize \centering $814881989695158$ & \footnotesize \centering $-2965210972182228$ \tabularnewline
\hline
\small \centering $(25,15,9)$ & \footnotesize \centering $-278964$ & \footnotesize \centering $533148210$ & \footnotesize \centering $-7056168168$ & \footnotesize \centering $2683226030436$ & \footnotesize \centering $-15864469792374$ & \footnotesize \centering $968124970032822$ & \footnotesize \centering $-2966903818822020$ \tabularnewline
\hline
\small \centering $(25,17,3)$ & \footnotesize \centering $-1478088$ & \footnotesize \centering $884141220$ & \footnotesize \centering $-9475591056$ & \footnotesize \centering $1338439935912$ & \footnotesize \centering $-114003342180780$ & \footnotesize \centering $827431528322412$ & \footnotesize \centering $9018803395859736$ \tabularnewline
\hline
\small \centering $(25,17,7)$ & \footnotesize \centering $1265112$ & \footnotesize \centering $626270820$ & \footnotesize \centering $-13034888016$ & \footnotesize \centering $-3063060887928$ & \footnotesize \centering $-34174702764780$ & \footnotesize \centering $2038338006384492$ & \footnotesize \centering $1506984152124216$ \tabularnewline
\hline
\small \centering $(25,17,11)$ & \footnotesize \centering $872316$ & \footnotesize \centering $-474730350$ & \footnotesize \centering $-9663808008$ & \footnotesize \centering $6996289229556$ & \footnotesize \centering $-123888344826774$ & \footnotesize \centering $197426191828662$ & \footnotesize \centering $-8092805263108500$ \tabularnewline
\hline
\small \centering $(25,19,1)$ & \footnotesize \centering $-106596$ & \footnotesize \centering $353216850$ & \footnotesize \centering $-17012565192$ & \footnotesize \centering $10854722172756$ & \footnotesize \centering $24295975183914$ & \footnotesize \centering $-2237898756283722$ & \footnotesize \centering $-1116669445539060$ \tabularnewline
\hline
\small \centering $(25,19,5)$ & \footnotesize \centering $90072$ & \footnotesize \centering $-334979100$ & \footnotesize \centering $-31105966416$ & \footnotesize \centering $-7883875892088$ & \footnotesize \centering $-105638103433068$ & \footnotesize \centering $-2537945828699796$ & \footnotesize \centering $10159571243517240$ \tabularnewline
\hline
\small \centering $(25,19,9)$ & \footnotesize \centering $-573192$ & \footnotesize \centering $927204900$ & \footnotesize \centering $62961605616$ & \footnotesize \centering $-3096943985688$ & \footnotesize \centering $-15467475516972$ & \footnotesize \centering $-508393328631444$ & \footnotesize \centering $821432168707800$ \tabularnewline
\hline
\small \centering $(25,19,13)$ & \footnotesize \centering $-702324$ & \footnotesize \centering $9404850$ & \footnotesize \centering $-14719266408$ & \footnotesize \centering $23152557649956$ & \footnotesize \centering $-10567857144054$ & \footnotesize \centering $3351056484428982$ & \footnotesize \centering $-4267132336471620$ \tabularnewline
\hline
\small \centering $(25,21,3)$ & \footnotesize \centering $-170280$ & \footnotesize \centering $-823542300$ & \footnotesize \centering $4910286000$ & \footnotesize \centering $-1405391636088$ & \footnotesize \centering $145190225249940$ & \footnotesize \centering $-842678842445460$ & \footnotesize \centering $-6403875311384520$ \tabularnewline
\hline
\small \centering $(25,21,7)$ & \footnotesize \centering $-108360$ & \footnotesize \centering $433601700$ & \footnotesize \centering $43209490800$ & \footnotesize \centering $12737766447912$ & \footnotesize \centering $-109920761915820$ & \footnotesize \centering $-1119504013993620$ & \footnotesize \centering $9538661136172440$ \tabularnewline
\hline
\small \centering $(25,21,11)$ & \footnotesize \centering $511128$ & \footnotesize \centering $-401727900$ & \footnotesize \centering $28143226416$ & \footnotesize \centering $9867684455112$ & \footnotesize \centering $90846882696468$ & \footnotesize \centering $2611978425209196$ & \footnotesize \centering $3887087995313400$ \tabularnewline
\hline
\small \centering $(25,21,15)$ & \footnotesize \centering $411516$ & \footnotesize \centering $-439386990$ & \footnotesize \centering $18155978232$ & \footnotesize \centering $-3315674449164$ & \footnotesize \centering $-10179464734614$ & \footnotesize \centering $657746166515382$ & \footnotesize \centering $19498517165502060$ \tabularnewline
\hline
\end{tabular} \caption{Liste des $p^{w_1 /2 }\cdot \sum_{\pi \in \Pi } \mathrm{Trace} \left( \mathrm{c}_p (\pi) \vert V_\mathrm{St} \right)$ pour $\Pi^\bot_{\mathrm{alg}}\left( \mathrm{PGL}_6\right)$ et $3\leq p\leq 19$} \label{tableau3SO7} \end{center} \end{sidewaystable}
\begin{sidewaystable} \begin{center} \renewcommand\arraystretch{1.2} \begin{tabular}{|c|c|c|c|c|c|}
\hline
\backslashbox{$\overline{w}$}{$p$} & \centering $23$ & \centering $29$ & \centering $31$  & \centering $37$ & \centering $41$  \tabularnewline
\hline
\small \centering $(23,13,5)$ & \footnotesize \centering $213729467233464$ & \footnotesize \centering $-48303125789698494$ & \footnotesize \centering $102404919986257056$ & \footnotesize \centering $-854573389531170582$ & \footnotesize \centering $3039963284658810510$ \tabularnewline
\hline
\small \centering $(23,15,3)$ & \footnotesize \centering $4636605523096344$ & \footnotesize \centering $38242077134129826$ & \footnotesize \centering $62411711982932256$ & \footnotesize \centering $194846871383803338$ & \footnotesize \centering $-2444768682516441330$ \tabularnewline
\hline
\small \centering $(23,15,7)$ & \footnotesize \centering $-483871286752776$ & \footnotesize \centering $31281904260703746$ & \footnotesize \centering $-99338149426357344$ & \footnotesize \centering $488935778497554858$ & \footnotesize \centering $-761251774743469170$ \tabularnewline
\hline
\small \centering $(23,17,5)$ & \footnotesize \centering $-1778909130889896$ & \footnotesize \centering $41065622154151266$ & \footnotesize \centering $-49354387962315744$ & \footnotesize \centering $-699393293745596022$ & \footnotesize \centering $675213588259710990$ \tabularnewline
\hline
\small \centering $(23,17,9)$ & \footnotesize \centering $2740073764616568$ & \footnotesize \centering $9074529911413890$ & \footnotesize \centering $-106519618777533024$ & \footnotesize \centering $-845122486221439446$ & \footnotesize \centering $-2038917131601784434$ \tabularnewline
\hline
\small \centering $(23,19,3)$ & \footnotesize \centering $9794542472491128$ & \footnotesize \centering $-92800925124117630$ & \footnotesize \centering $-102827107514992224$ & \footnotesize \centering $909641598021211434$ & \footnotesize \centering $-1012591870481195634$ \tabularnewline
\hline
\small \centering $(23,19,11)$ & \footnotesize \centering $-2457101779651272$ & \footnotesize \centering $74027182945751490$ & \footnotesize \centering $55350292141154976$ & \footnotesize \centering $84118925862153834$ & \footnotesize \centering $-1668889536698238834$ \tabularnewline
\hline
\small \centering $(25,13,3)$ & \footnotesize \centering $-1193162553976248$ & \footnotesize \centering $253606518017413434$ & \footnotesize \centering $2735597919197168736$ & \footnotesize \centering $4732933858242304146$ & \footnotesize \centering $-149824999263679701570$ \tabularnewline
\hline
\small \centering $(25,13,7)$ & \footnotesize \centering $20178334110978792$ & \footnotesize \centering $432436992664549914$ & \footnotesize \centering $-3130027620660754464$ & \footnotesize \centering $2522703019906092786$ & \footnotesize \centering $-93544655494895631810$ \tabularnewline
\hline
\small \centering $(25,15,5)$ & \footnotesize \centering $-88220277023194008$ & \footnotesize \centering $-1155443769692920422$ & \footnotesize \centering $9370888247092906464$ & \footnotesize \centering $-31117677722924636046$ & \footnotesize \centering $-1528438521983962050$ \tabularnewline
\hline
\small \centering $(25,15,9)$ & \footnotesize \centering $-80036596250977464$ & \footnotesize \centering $-1707876195338103750$ & \footnotesize \centering $-780170604309503904$ & \footnotesize \centering $61060941903027848082$ & \footnotesize \centering $-22751573768921753154$ \tabularnewline
\hline
\small \centering $(25,17,3)$ & \footnotesize \centering $-104847531874172592$ & \footnotesize \centering $2031320319236853684$ & \footnotesize \centering $-2368965684580243008$ & \footnotesize \centering $5344348460239232868$ & \footnotesize \centering $-74090640376161130884$ \tabularnewline
\hline
\small \centering $(25,17,7)$ & \footnotesize \centering $-80888535775088112$ & \footnotesize \centering $-5992183465309870476$ & \footnotesize \centering $3003950696884185792$ & \footnotesize \centering $34645711154077546788$ & \footnotesize \centering $-10358280803181806724$ \tabularnewline
\hline
\small \centering $(25,17,11)$ & \footnotesize \centering $-109137642615383064$ & \footnotesize \centering $1797713610306337050$ & \footnotesize \centering $6528516996995314656$ & \footnotesize \centering $16749083457180770802$ & \footnotesize \centering $6852460464769529406$ \tabularnewline
\hline
\small \centering $(25,19,1)$ & \footnotesize \centering $-36279912858303576$ & \footnotesize \centering $711971194422472410$ & \footnotesize \centering $2456722774222969056$ & \footnotesize \centering $-37665796944733211982$ & \footnotesize \centering $196046761265779865406$ \tabularnewline
\hline
\small \centering $(25,19,5)$ & \footnotesize \centering $72168216158130192$ & \footnotesize \centering $2501005156335478260$ & \footnotesize \centering $-1340464530776071488$ & \footnotesize \centering $-20702446507555550556$ & \footnotesize \centering $-18731680328982176388$ \tabularnewline
\hline
\small \centering $(25,19,9)$ & \footnotesize \centering $31165081511786448$ & \footnotesize \centering $-108342006282833100$ & \footnotesize \centering $-317484792913639488$ & \footnotesize \centering $-27388402339916500764$ & \footnotesize \centering $-76102667868305446788$ \tabularnewline
\hline
\small \centering $(25,19,13)$ & \footnotesize \centering $-6790004985987384$ & \footnotesize \centering $-76851310696194630$ & \footnotesize \centering $-5659730220769482144$ & \footnotesize \centering $-14458992118594150638$ & \footnotesize \centering $242176317466276764606$ \tabularnewline
\hline
\small \centering $(25,21,3)$ & \footnotesize \centering $-3967628067667440$ & \footnotesize \centering $1872480211861343220$ & \footnotesize \centering $-2037559218626004288$ & \footnotesize \centering $27487981729794451620$ & \footnotesize \centering $263156609307726841212$ \tabularnewline
\hline
\small \centering $(25,21,7)$ & \footnotesize \centering $222258068930775120$ & \footnotesize \centering $2495104422809767860$ & \footnotesize \centering $-2498495348289012288$ & \footnotesize \centering $-74513456812450965660$ & \footnotesize \centering $12981476072926873212$ \tabularnewline
\hline
\small \centering $(25,21,11)$ & \footnotesize \centering $-170808940064948592$ & \footnotesize \centering $376161108771818100$ & \footnotesize \centering $-6240617945672417088$ & \footnotesize \centering $946755778368702756$ & \footnotesize \centering $71026063736978463612$ \tabularnewline
\hline
\small \centering $(25,21,15)$ & \footnotesize \centering $120369062052633576$ & \footnotesize \centering $-202084440446118630$ & \footnotesize \centering $1398887694440035296$ & \footnotesize \centering $-33770972241982657038$ & \footnotesize \centering $64345458637891946046$ \tabularnewline
\hline
\end{tabular} \caption{Liste des $p^{w_1 /2 }\cdot \sum_{\pi \in \Pi } \mathrm{Trace} \left( \mathrm{c}_p (\pi) \vert V_\mathrm{St} \right)$ pour $\Pi^\bot_{\mathrm{alg}}\left( \mathrm{PGL}_6\right)$ et $23\leq p\leq 41$} \label{tableau4SO7} \end{center} \end{sidewaystable}
\clearpage
\begin{table} \begin{center} \renewcommand\arraystretch{1.2} \begin{tabular}{|m{1.5cm}|c|c|c|}
\hline
\backslashbox{$\overline{w}$}{p} & \centering $43$ & \centering $47$ & \centering $53$ \tabularnewline
\hline
\small \centering $(23,13,5)$ & \small \centering $1585197541121400492$ & \small \centering $2888879429822981616$ & \small \centering $-401934470208658758$ \tabularnewline
\hline
\small \centering $(23,15,3)$ & \small \centering $1251580056673244892$ & \small \centering $729807353383997616$ & \small \centering $-103182932449233424998$ \tabularnewline
\hline
\small \centering $(23,15,7)$ & \small \centering $-6758592609864707508$ & \small \centering $-20510124050426653584$ & \small \centering $48013741730657079162$ \tabularnewline
\hline
\small \centering $(23,17,5)$ & \small \centering $4997229047209559292$ & \small \centering $-3083068930104075984$ & \small \centering $50560941459854437722$ \tabularnewline
\hline
\small \centering $(23,17,9)$ & \small \centering $11890923043443050508$ & \small \centering $15431114760408787824$ & \small \centering $-7026502567848047622$ \tabularnewline
\hline
\small \centering $(23,19,3)$ & \small \centering $-3198945438336050292$ & \small \centering $4575412865015044464$ & \small \centering $-33888522555375856902$ \tabularnewline
\hline
\small \centering $(23,19,11)$ & \small \centering $-2386037760238127892$ & \small \centering $-41570441127723864336$ & \small \centering $-19076488865676636102$ \tabularnewline
\hline
\small \centering $(25,13,3)$ & \small \centering $-118592663540334048444$ & \small \centering $265471738731534187152$ & \small \centering $6529626819380030334786$ \tabularnewline
\hline
\small \centering $(25,13,7)$ & \small \centering $-27876108969519548844$ & \small \centering $483419531351826739152$ & \small \centering $-67841640042648419934$ \tabularnewline
\hline
\small \centering $(25,15,5)$ & \small \centering $-310395560121687358956$ & \small \centering $2481990812763404305104$ & \small \centering $989150772174783875874$ \tabularnewline
\hline
\small \centering $(25,15,9)$ & \small \centering $-174530596427091285564$ & \small \centering $39703282543066180752$ & \small \centering $-5079143986594630602174$ \tabularnewline
\hline
\small \centering $(25,17,3)$ & \small \centering $182208972814755659112$ & \small \centering $-103615919209859815776$ & \small \centering $-816797105524166216508$ \tabularnewline
\hline
\small \centering $(25,17,7)$ & \small \centering $-20453253350286370488$ & \small \centering $1741863267899807506464$ & \small \centering $-4172927152787349895548$ \tabularnewline
\hline
\small \centering $(25,17,11)$ & \small \centering $-140783944305361504044$ & \small \centering $407136532197503992272$ & \small \centering $3183793193891665327266$ \tabularnewline
\hline
\small \centering $(25,19,1)$ & \small \centering $-438706021055601207756$ & \small \centering $-559258375196038145712$ & \small \centering $1807317125273707699554$ \tabularnewline
\hline
\small \centering $(25,19,5)$ & \small \centering $184893881031217770312$ & \small \centering $625764793414851116064$ & \small \centering $2026224497868971399172$ \tabularnewline
\hline
\small \centering $(25,19,9)$ & \small \centering $-272331427509356263512$ & \small \centering $1798384539618122498976$ & \small \centering $-5663863869660148328892$ \tabularnewline
\hline
\small \centering $(25,19,13)$ & \small \centering $157365362411733901956$ & \small \centering $252625491987210302352$ & \small \centering $5089431783552918322626$ \tabularnewline
\hline
\small \centering $(25,21,3)$ & \small \centering $-240449199626663273400$ & \small \centering $-701221992491721039840$ & \small \centering $2592653972992766998020$ \tabularnewline
\hline
\small \centering $(25,21,7)$ & \small \centering $-660835618776165010200$ & \small \centering $1125591434635074114720$ & \small \centering $4607672947886504889540$ \tabularnewline
\hline
\small \centering $(25,21,11)$ & \small \centering $-48343195754042760312$ & \small \centering $1145954720828549941536$ & \small \centering $5888654435488579217028$ \tabularnewline
\hline
\small \centering $(25,21,15)$ & \small \centering $-282475205135353880364$ & \small \centering $163067533980263907792$ & \small \centering $-4351329192379786592094$ \tabularnewline
\hline
\end{tabular} \caption{Liste des $p^{w_1 /2 }\cdot \sum_{\pi \in \Pi } \mathrm{Trace} \left( \mathrm{c}_p (\pi) \vert V_\mathrm{St} \right)$ pour $\Pi^\bot_{\mathrm{alg}}\left( \mathrm{PGL}_6\right)$ et $43\leq p\leq 53$} \label{tableau5SO7} \end{center} \end{table}
\begin{table} \begin{center} \renewcommand\arraystretch{1.2} \begin{tabular}{|c|c|c|c|c|}
\hline
\backslashbox{$\overline{w}$}{$p$} & \centering $2$ & \centering $3$ & \centering $5$  & \centering $7$ \tabularnewline
\hline
\small \centering $(25,17,9,5)$ & \small \centering $1104$ & \small \centering $-671328$ & \small \centering $-35927880$ & \small \centering $19973315264$ \tabularnewline
\hline
\small \centering $(25,17,13,5)$ & \small \centering $1632$ & \small \centering $492912$ & \small \centering $-45088680$ & \small \centering $-45797041696$ \tabularnewline
\hline
\small \centering $(25,19,9,3)$ & \small \centering $-5280$ & \small \centering $-317520$ & \small \centering $71568600$ & \small \centering $36073032800$ \tabularnewline
\hline
\small \centering $(25,19,11,5)$ & \small \centering $4800$ & \small \centering $-302400$ & \small \centering $-765121800$ & \small \centering $29642547200$ \tabularnewline
\hline
\small \centering $(25,19,13,3)$ & \small \centering $2640$ & \small \centering $-483840$ & \small \centering $23969400$ & \small \centering $3255324800$ \tabularnewline
\hline
\small \centering $(25,19,13,5)$ & \small \centering $-4416$ & \small \centering $-913248$ & \small \centering $-155434440$ & \small \centering $1629650624$ \tabularnewline
\hline
\small \centering $(25,19,13,7)$ & \small \centering $-3840$ & \small \centering $753840$ & \small \centering $-132911400$ & \small \centering $83659503200$ \tabularnewline
\hline
\small \centering $(25,19,13,9)$ & \small \centering $960$ & \small \centering $-498960$ & \small \centering $-500274600$ & \small \centering $-34659738400$ \tabularnewline
\hline
\small \centering $(25,19,15,5)$ & \small \centering $-6432$ & \small \centering $444528$ & \small \centering $985329240$ & \small \centering $14967875552$ \tabularnewline
\hline
\small \centering $(25,21,11,7)$ & \small \centering $7920$ & \small \centering $-274320$ & \small \centering $181517400$ & \small \centering $-414752800$ \tabularnewline
\hline
\small \centering $(25,21,13,5)$ & \small \centering $8928$ & \small \centering $139968$ & \small \centering $-181179144$ & \small \centering $-9742673920$ \tabularnewline
\hline
\small \centering $(25,21,13,7)$ & \small \centering $-1920$ & \small \centering $58320$ & \small \centering $-511607400$ & \small \centering $-7597141600$ \tabularnewline
\hline
\small \centering $(25,21,15,3)$ & \small \centering $-3072$ & \small \centering $995568$ & \small \centering $148022616$ & \small \centering $-806421280$ \tabularnewline
\hline
\small \centering $(25,21,15,5)$ & \small \centering $-1152$ & \small \centering $-994032$ & \small \centering $-652925160$ & \small \centering $-48906846688$ \tabularnewline
\hline
\small \centering $(25,21,15,7)$ & \small \centering $-3408$ & \small \centering $1215360$ & \small \centering $204437616$ & \small \centering $15834248704$ \tabularnewline
\hline
\small \centering $(25,21,15,9)$ & \small \centering $-7200$ & \small \centering $631200$ & \small \centering $6175800$ & \small \centering $25981995200$ \tabularnewline
\hline
\small \centering $(25,21,17,5)$ & \small \centering $6528$ & \small \centering $301968$ & \small \centering $180352536$ & \small \centering $15716429600$ \tabularnewline
\hline
\small \centering $(25,21,17,7)$ & \small \centering $6240$ & \small \centering $-894240$ & \small \centering $-877974600$ & \small \centering $-16347755200$ \tabularnewline
\hline
\small \centering $(25,21,17,9)$ & \small \centering $480$ & \small \centering $421920$ & \small \centering $451865400$ & \small \centering $12240996800$ \tabularnewline
\hline
\small \centering $(25,23,9,3)$ & \small \centering $-240$ & \small \centering $-675360$ & \small \centering $76659000$ & \small \centering $-8636958400$ \tabularnewline
\hline
\small \centering $(25,23,11,1)$ & \small \centering $-7440$ & \small \centering $-574560$ & \small \centering $-258371400$ & \small \centering $45468651200$ \tabularnewline
\hline
\small \centering $(25,23,11,5)$ & \small \centering $2832$ & \small \centering $758880$ & \small \centering $16184496$ & \small \centering $-20980586816$ \tabularnewline
\hline
\small \centering $(25,23,13,3)$ & \small \centering $288$ & \small \centering $843696$ & \small \centering $80271576$ & \small \centering $-6565786528$ \tabularnewline
\hline
\small \centering $(25,23,13,7)$ & \small \centering $3888$ & \small \centering $-861984$ & \small \centering $1188954936$ & \small \centering $448814912$ \tabularnewline
\hline
\small \centering $(25,23,15,1)$ & \small \centering $48$ & \small \centering $-950832$ & \small \centering $1608216$ & \small \centering $-5559691360$ \tabularnewline
\hline
\small \centering $(25,23,15,5)$ & \small \centering $-4608$ & \small \centering $-495072$ & \small \centering $94477608$ & \small \centering $49773071040$ \tabularnewline
\hline
\small \centering $(25,23,15,9)$ & \small \centering $-1392$ & \small \centering $-382032$ & \small \centering $-172266024$ & \small \centering $-6105235360$ \tabularnewline
\hline
\small \centering $(25,23,15,11)$ & \small \centering $-1056$ & \small \centering $538272$ & \small \centering $-360152520$ & \small \centering $356506304$ \tabularnewline
\hline
\small \centering $(25,23,17,3)$ & \small \centering $1488$ & \small \centering $1040256$ & \small \centering $-758350344$ & \small \centering $11560030592$ \tabularnewline
\hline
\small \centering $(25,23,17,5)$ & \small \centering $-2976$ & \small \centering $97632$ & \small \centering $34950840$ & \small \centering $-31527057856$ \tabularnewline
\hline
\small \centering $(25,23,17,7)$ & \small \centering $1488$ & \small \centering $-1687824$ & \small \centering $-107874984$ & \small \centering $26073028832$ \tabularnewline
\hline
\small \centering $(25,23,17,11)$ & \small \centering $480$ & \small \centering $-211680$ & \small \centering $98316600$ & \small \centering $28990400$ \tabularnewline
\hline
\small \centering $(25,23,19,5)$ & \small \centering $-6432$ & \small \centering $-950832$ & \small \centering $-477237864$ & \small \centering $-7262923360$ \tabularnewline
\hline
\end{tabular} \caption{Liste des $p^{w_1 /2 }\cdot \sum_{\pi \in \Pi } \mathrm{Trace} \left( \mathrm{c}_p (\pi) \vert V_\mathrm{St} \right)$ pour $\Pi^\bot_{\mathrm{alg}}\left( \mathrm{PGL}_8\right)$ et $2\leq p\leq 7$} \label{tableau1SO9} \end{center} \end{table}
\begin{table} \begin{center} \renewcommand\arraystretch{1.2} \begin{tabular}{|c|m{12cm}|}
\hline
\rule[-0.4cm]{0mm}{1cm} $(w_1,w_2,w_3)$ & \centering $\mathrm{det}\left( 2^{w_1/2} X\cdot \mathrm{Id} - \mathrm{c}_2\left( \Delta^*_{w_1,w_2 ,w_3} \right) \right)$ \tabularnewline
\hline
\small \centering $(24,16,8)$ & \small \centering $ 2^{ 84 }\cdot X^7 + \dots -35053633536\cdot X^3 - 5300736\cdot X^2 - 3016\cdot X - 1$ \tabularnewline\hline
\small \centering $(26,16,10)$ & \small \centering $ 2^{ 91 }\cdot X^7 + \dots  + 187711881216\cdot X^3 + 10444800\cdot X^2 + 1312\cdot X - 1$ \tabularnewline\hline
\small \centering $(26,20,6)$ & \small \centering $ 2^{ 91 }\cdot X^7 + \dots -184113168384\cdot X^3 - 6881280\cdot X^2 - 3008\cdot X - 1$ \tabularnewline\hline
\small \centering $(26,20,10)$ & \small \centering $ 2^{ 91 }\cdot X^7 + \dots  + 847484289024\cdot X^3 - 116293632\cdot X^2 + 9088\cdot X - 1$ \tabularnewline\hline
\small \centering $(26,20,14)$ & \small \centering $ 2^{ 91 }\cdot X^7 + \dots -582362333184\cdot X^3 + 4546560\cdot X^2 + 2752\cdot X - 1$ \tabularnewline\hline
\small \centering $(26,24,10)$ & \small \centering $ 2^{ 91 }\cdot X^7 + \dots  + 281257967616\cdot X^3 + 53460480\cdot X^2 - 5168\cdot X - 1$ \tabularnewline\hline
\small \centering $(26,24,14)$ & \small \centering $ 2^{ 91 }\cdot X^7 + \dots -17485529088\cdot X^3 - 57698304\cdot X^2 + 9088\cdot X - 1$ \tabularnewline\hline
\small \centering $(26,24,18)$ & \small \centering $ 2^{ 91 }\cdot X^7 + \dots  + 29700390912\cdot X^3 + 3495936\cdot X^2 - 8192\cdot X - 1$ \tabularnewline\hline
\end{tabular} \caption{Polynômes caractéristiques des $\mathrm{c}_2 (\pi)$ pour $\pi \in \Pi^\bot_{\mathrm{alg}}\left( \mathrm{PGL}_7\right)$} \label{tableau1SO8} \end{center} \end{table}
\begin{table} \begin{center} \renewcommand\arraystretch{1.2} \begin{tabular}{|c|c|c|c|c|c|}
\hline
\backslashbox{$\overline{w}$}{$p$} & \centering $3$ & \centering $5$ & \centering $7$  & \centering $11$ & \centering $13$  \tabularnewline
\hline
\small \centering $(24,16,8)$ & \small \centering $-350001$ & \small \centering $124371575$ & \small \centering $93528799$ & \small \centering $-2714242598353$ & \small \centering $3657301688599$ \tabularnewline
\hline
\small \centering $(26,16,10)$ & \small \centering $-624051$ & \small \centering $-1326755021$ & \small \centering $16022951833$ & \small \centering $9888917076709$ & \small \centering $-93579285696245$ \tabularnewline
\hline
\small \centering $(26,20,6)$ & \small \centering $-1721331$ & \small \centering $905950579$ & \small \centering $-9930145127$ & \small \centering $30034826719909$ & \small \centering $-337487394517685$ \tabularnewline
\hline
\small \centering $(26,20,10)$ & \small \centering $-608499$ & \small \centering $-280183181$ & \small \centering $84407107225$ & \small \centering $-11018813856347$ & \small \centering $-147086616834485$ \tabularnewline
\hline
\small \centering $(26,20,14)$ & \small \centering $404109$ & \small \centering $1190523379$ & \small \centering $-15973145447$ & \small \centering $-57600800963291$ & \small \centering $-151745371034165$ \tabularnewline
\hline
\small \centering $(26,24,10)$ & \small \centering $170829$ & \small \centering $1280278579$ & \small \centering $-54192968807$ & \small \centering $27176421262309$ & \small \centering $-273154746313205$ \tabularnewline
\hline
\small \centering $(26,24,14)$ & \small \centering $1183437$ & \small \centering $-185269325$ & \small \centering $5344579993$ & \small \centering $-6837258083483$ & \small \centering $-80849598511733$ \tabularnewline
\hline
\small \centering $(26,24,18)$ & \small \centering $-175923$ & \small \centering $-813224525$ & \small \centering $59374762393$ & \small \centering $40761042089317$ & \small \centering $560921705611147$ \tabularnewline
\hline
\end{tabular} \caption{Liste des $p^{w_1 /2 }\cdot \sum_{\pi \in \Pi } \mathrm{Trace} \left( \mathrm{c}_p (\pi) \vert V_\mathrm{St} \right)$ pour $\Pi^\bot_{\mathrm{alg}}\left( \mathrm{PGL}_7\right)$ et $3\leq p\leq 13$} \label{tableau2SO8} \end{center} \end{table}
\begin{table} \begin{center} \renewcommand\arraystretch{1.2} \begin{tabular}{|c|m{12cm}|}
\hline
\rule[-0.4cm]{0mm}{1cm} $(w_1,w_2,w_3,w_4)$ & \centering $\mathrm{det}\left( 2^{w_1/2} X\cdot \mathrm{Id} - \mathrm{c}_2\left( \Delta_{w_1,w_2 ,w_3,w_4} \right) \right)$ \tabularnewline
\hline
\small \centering $(24,18,10,4)$ & \footnotesize \centering $ 2^{ 96 }\cdot X^8 + \dots -71176198029312\cdot X^4 - 5324800\cdot X^2 + 1$ \tabularnewline\hline
\small \centering $(24,20,14,2)$ & \footnotesize \centering $ 2^{ 96 }\cdot X^8 + \dots  + 18937584549888\cdot X^4 + 34233384960\cdot X^3 - 3020800\cdot X^2 + 1440\cdot X + 1$ \tabularnewline\hline
\small \centering $(26,18,10,2)$ & \footnotesize \centering $ 2^{ 104 }\cdot X^8 + \dots  + 1664854951723008\cdot X^4 + 96153108480\cdot X^3 + 33190400\cdot X^2 + 11880\cdot X + 1$ \tabularnewline\hline
\small \centering $(26,18,14,6)$ & \footnotesize \centering $ 2^{ 104 }\cdot X^8 + \dots  + 5195343838838784\cdot X^4 + 113224974336\cdot X^3 - 89271808\cdot X^2 - 3672\cdot X + 1$ \tabularnewline\hline
\small \centering $(26,20,10,4)$ & \footnotesize \centering $ 2^{ 104 }\cdot X^8 + \dots -5907882134470656\cdot X^4 - 380356263936\cdot X^3 - 2646016\cdot X^2 + 6336\cdot X + 1$ \tabularnewline\hline
\small \centering $(26,20,14,8)$ & \footnotesize \centering $ 2^{ 104 }\cdot X^8 + \dots  + 2055025887019008\cdot X^4 - 320675512320\cdot X^3 - 63692800\cdot X^2 + 2880\cdot X + 1$ \tabularnewline\hline
\small \centering $(26,22,10,6)$ & \footnotesize \centering $ 2^{ 104 }\cdot X^8 + \dots -5275597450248192\cdot X^4 + 586910269440\cdot X^3 + 29043200\cdot X^2 - 3960\cdot X + 1$ \tabularnewline\hline
\small \centering $(26,22,14,2)$ & \footnotesize \centering $ 2^{ 104 }\cdot X^8 + \dots  + 3431375000567808\cdot X^4 - 183848140800\cdot X^3 + 70054400\cdot X^2 - 5400\cdot X + 1$ \tabularnewline\hline
\small \centering $(26,24,14,4)$ & \footnotesize \centering $ 2^{ 104 }\cdot X^8 + \dots -5487318936846336\cdot X^4 - 113397202944\cdot X^3 + 101384192\cdot X^2 + 16128\cdot X + 1$ \tabularnewline\hline
\small \centering $(26,24,16,2)$ & \footnotesize \centering $ 2^{ 104 }\cdot X^8 + \dots  + 935933503340544\cdot X^4 - 391440236544\cdot X^3 + 7104512\cdot X^2 - 4032\cdot X + 1$ \tabularnewline\hline
\small \centering $(26,24,18,8)$ & \footnotesize \centering $ 2^{ 104 }\cdot X^8 + \dots  + 5006554256375808\cdot X^4 - 522171187200\cdot X^3 + 65446400\cdot X^2 - 10800\cdot X + 1$ \tabularnewline\hline
\small \centering $(26,24,20,6)$ & \footnotesize \centering $ 2^{ 104 }\cdot X^8 + \dots -1289572520558592\cdot X^4 - 121173442560\cdot X^3 + 819200\cdot X^2 + 8640\cdot X + 1$ \tabularnewline\hline
\end{tabular} \caption{Polynômes caractéristiques des $\mathrm{c}_2 (\pi)$ pour $\pi \in \Pi^\bot_{\mathrm{alg}}\left( \mathrm{PGL}_8\right)$} \label{tableau3SO8} \end{center} \end{table}
\begin{table} \begin{center} \renewcommand\arraystretch{1.2} \begin{tabular}{|c|c|c|c|c|c|}
\hline
\backslashbox{$\overline{w}$}{$p$} & \centering $3$ & \centering $5$ & \centering $7$  & \centering $11$ & \centering $13$  \tabularnewline
\hline
\small \centering $(24,18,10,4)$ & \small \centering $-453600$ & \small \centering $-119410200$ & \small \centering $12572892800$ & \small \centering $-57063064032$ & \small \centering $-25198577349400$ \tabularnewline
\hline
\small \centering $(24,20,14,2)$ & \small \centering $-90720$ & \small \centering $-381691800$ & \small \centering $15880860800$ & \small \centering $1429298110368$ & \small \centering $1852311565160$ \tabularnewline
\hline
\small \centering $(26,18,10,2)$ & \small \centering $-1028160$ & \small \centering $93177000$ & \small \centering $-37259756800$ & \small \centering $24070317594048$ & \small \centering $-181403983972120$ \tabularnewline
\hline
\small \centering $(26,18,14,6)$ & \small \centering $36288$ & \small \centering $-407597400$ & \small \centering $-13480246528$ & \small \centering $-50783707225152$ & \small \centering $-553105612803352$ \tabularnewline
\hline
\small \centering $(26,20,10,4)$ & \small \centering $2231712$ & \small \centering $-2103821496$ & \small \centering $-49948420480$ & \small \centering $-33016093688160$ & \small \centering $297288355585928$ \tabularnewline
\hline
\small \centering $(26,20,14,8)$ & \small \centering $-2086560$ & \small \centering $-923239800$ & \small \centering $90060118400$ & \small \centering $34658502500448$ & \small \centering $251026605281480$ \tabularnewline
\hline
\small \centering $(26,22,10,6)$ & \small \centering $1321920$ & \small \centering $611173800$ & \small \centering $-13077433600$ & \small \centering $-21930073906752$ & \small \centering $-207826052609560$ \tabularnewline
\hline
\small \centering $(26,22,14,2)$ & \small \centering $-1166400$ & \small \centering $156076200$ & \small \centering $25574009600$ & \small \centering $-40811535001152$ & \small \centering $145995515911400$ \tabularnewline
\hline
\small \centering $(26,24,14,4)$ & \small \centering $-1851552$ & \small \centering $313754760$ & \small \centering $34598801792$ & \small \centering $-25141764069792$ & \small \centering $232075615185608$ \tabularnewline
\hline
\small \centering $(26,24,16,2)$ & \small \centering $-132192$ & \small \centering $188771400$ & \small \centering $-162950201728$ & \small \centering $-32763087987552$ & \small \centering $180651961034888$ \tabularnewline
\hline
\small \centering $(26,24,18,8)$ & \small \centering $1360800$ & \small \centering $1350102600$ & \small \centering $-40839971200$ & \small \centering $-45290750221152$ & \small \centering $230082936830600$ \tabularnewline
\hline
\small \centering $(26,24,20,6)$ & \small \centering $-341280$ & \small \centering $795285000$ & \small \centering $-117493532800$ & \small \centering $198712532448$ & \small \centering $54899265210440$ \tabularnewline
\hline
\end{tabular} \caption{Liste des $p^{w_1 /2 }\cdot \sum_{\pi \in \Pi } \mathrm{Trace} \left( \mathrm{c}_p (\pi) \vert V_\mathrm{St} \right)$ pour $\Pi^\bot_{\mathrm{alg}}\left( \mathrm{PGL}_8\right)$ et $3\leq p\leq 13$} \label{tableau4SO8} \end{center} \end{table}

\end{document}